\documentclass[11pt,reqno]{amsart}

\usepackage{amsmath}
\usepackage{amssymb}
\usepackage{amsthm}

\newtheorem{theorem}{Theorem}[section]
\newtheorem{prop}[theorem]{Proposition}
\newtheorem{lem}[theorem]{Lemma}
\newtheorem*{cor}{Corollary}
\theoremstyle{definition}
\newtheorem{defn}[theorem]{Definition}
\newtheorem{ex}[theorem]{Example}
\theoremstyle{remark}
\newtheorem*{rem}{Remark}
\numberwithin{equation}{section}

\begin{document}

\title[Multiplicative partial isometries]
{Multiplicative partial isometries, manageability, and $C^*$-algebraic quantum groupoids}

\author{Byung-Jay Kahng}
\date{}
\address{Department of Mathematics and Statistics\\ Canisius University\\
Buffalo, NY 14208, USA}
\email{kahngb@canisius.edu}


\begin{abstract}
Generalizing the notion of a multiplicative unitary operator, which plays a fundamental role in the theory of locally compact quantum groups, 
we develop in this paper the notion of a {\em multiplicative partial isometry\/}.  The axioms include the pentagon equation, but more is needed. 
Under the ``manageability'' condition on a multiplicative partial isometry (modified from the Woronowicz's condition for a multiplicative unitary), 
it is possible to construct from it a pair of $C^*$-algebras having almost the structure of a {\em $C^*$-algebraic quantum groupoid of separable type\/}.
\end{abstract}
\maketitle

\setcounter{section}{-1} 
\section{Introduction}

Let $G$ be a locally compact group with the left Haar measure $\nu$, and let ${\mathcal H}=L^2(G,\nu)$. We can identify 
${\mathcal H}\otimes{\mathcal H}=L^2(G\times G,\nu\times\nu)$. Define an operator $W\in{\mathcal B}({\mathcal H}\otimes{\mathcal H})$, 
by $W\xi(s,t)=\xi(s,s^{-1}t)$. It is an isometry due to the left invariance of $\nu$, and surjective because $(s,t)\mapsto(s,s^{-1}t)$ is 
a homeomorphism on $G\times G$ by the cancellation property. So $W$ becomes a unitary operator. In addition, the associativity 
of the group multiplication gives rise to a certain ``pentagon equation'', $W_{23}W_{12}=W_{12}W_{13}W_{23}$, as operators on 
${\mathcal H}\otimes{\mathcal H}\otimes{\mathcal H}$ (Here, we are using the standard three-leg notation, namely $W_{13}$ meaning 
the operator $W$ acting only on the 1st and 3rd legs, and such.) Similarly, we may instead work with the right Haar measure, $\nu^{-1}$, 
and consider ${\mathcal H}=L^2(G,\nu^{-1})$ and $V\in{\mathcal B}({\mathcal H}\otimes{\mathcal H})$, by $V\xi(s,t)=\xi(st,t)$.

It turns out that the operator $W$ (or $V$) captures lots of information about the group, from which one can recover the $C^*$-algebras 
$C_0(G)$ and $C^*_r(G)$ with the left (or right) regular representation, the coalgebra structures on them, a duality picture, and much more. 
Such a unitary operator satisfying the pentagon equation became a useful tool in efforts to extend the group duality picture and the Pontryagin 
duality, for instance the theory of Kac algebras \cite{ESbook}. The names like ``fundamental unitary operator'' or ``Kac--Takesaki operator'' 
have been used, but nowadays the preferred name for such an operator is the {\em multiplicative unitary operator\/}, since Baaj and Skandalis 
clarified the conditions axiomatically.  See \cite{BS}, \cite{Wr7}, and see also \cite{Timmbook}.

The multiplicative unitary operators play a fundamental role in the general theory of locally compact quantum groups. They give rise to 
the left/right regular representations of the associated quantum groups, while encoding their duality picture.  Refer to the general theory 
on locally compact quantum groups \cite{KuVa}, \cite{KuVavN}, \cite{MNW}, \cite{VDvN}. In addition, the multiplicative unitaries have been 
useful also in the construction of quantum groups, for instance as providing a way to describe their comultiplications  \cite{BS}, \cite{VV}, 
\cite{Rf5}, \cite{BJKjgp}.  

If we were to extend the framework to groupoids or quantum groupoids, where the multiplication is only partially defined and the 
set of units is nontrivial, it becomes apparent that we will need to work with a non-unitary operator. In this direction, Enock and Vallin 
introduced the notion of {\em pseudomultiplicative unitaries\/} \cite{EnVa1}, \cite{EnVa2}, \cite{Valpmu}. They are defined on relative 
tensor products of Hilbert spaces and are rather technical, but they play a fundamental role in the theory of {\em measured quantum 
groupoids\/} by Lesieur and Enock \cite{LesSMF}, \cite{EnSMF}. Measured quantum groupoids provide a general framework for studying 
quantum groupoids in the von Neumann algebra setting.  In the finite-dimensional case, they become {\em weak Hopf algebras\/} 
\cite{BNSwha1}, \cite{BSwha2} or {\em finite quantum groupoids\/} \cite{Valfqg}, \cite{NVfqg}.

In the $C^*$-algebra setting, the status is not as satisfactory.  Timmermann developed the notion of $C^*$-pseudomultiplicative 
unitaries and Hopf $C^*$-bimodules \cite{Timmbook}, \cite{Timm_jot}, but the most general theory of $C^*$-algebraic quantum 
groupoid seems elusive at present.  The reason is partly because the theory of psudomultiplicative unitaries and that of measured 
quantum groupoids use some primarily von Neumann algebraic tools such as the {\em fiber product\/}, whose $C^*$-algebraic 
counterpart is not clearly established.  A separate approach needs to be developed for the $C^*$-algebraic framework, which is 
on-going (see works by Timmermann  \cite{Timm_jot}, \cite{Timmpmu}).

At a reduced scale, the author, together with Van Daele, recently developed a $C^*$-algebraic framework for a subclass of quantum 
groupoids, namely the {\em locally compact quantum groupoids of separable type\/} \cite{BJKVD_qgroupoid1}, \cite{BJKVD_qgroupoid2}. 
In this theory, we naturally obtain certain ``multiplicative partial isometries''.  As in the case of multiplicative unitaries for quantum groups, 
such partial isometries give rise to the left/right regular representations, encode the duality picture, and play important roles in the construction 
of the antipode map.

Unlike the case of the multiplicative unitaries or that of the pseudomultiplicative unitaries, however, an axiomatic approach to multiplicative 
partial isometries has not been developed yet.  The aim of this paper is to address this situation.

It has been known since Enock and Vallin's work \cite{EnVa1} that in the finite-dimensional case, pseudomultiplicative unitaries become 
partial isometries, where the relative tensor product spaces associated with a pseudomultiplicative unitary become the initial and 
the terminal spaces of the corresponding partial isometry.  The associated measured quantum groupoids then become finite quantum 
groupoids or weak Hopf algebras.  Also B{\" o}hm and Szlachnyi provided a systematic treatment of finite-dimensional multiplicative 
partial isometries in \cite{BSzMultIso}, taking advantage of the results from the weak Hopf algebra theory.

Non-finite case is not as simple.  Loosely speaking, multiplicative partial isometries {\em should\/} be a special case of pseudomultiplicative 
unitaries.  At the same time, the locally compact quantum groupoids of separable type {\em should\/} be a special case of measured 
quantum groupoids.  However, as alluded to above, the situation is not as straightforward as one would hope.  Despite some known results, 
the way from pseudomultiplicative unitaries to multiplicative partial isometries is not completely understood even in the finite-dimensional 
setting (see comments given in \cite{BSzMultIso}). One primary reason is because of the von Neumann algebraic tools not translating well 
into the $C^*$-algebraic setting, and it has also to do with the fact that the theory of $C^*$-algebraic quantum groupoids based on 
multiplicative partial isometries (such as  \cite{BJKVD_qgroupoid1}, \cite{BJKVD_qgroupoid2}) has not been developed until recently.

The aim of this paper, as well as the theory of the $C^*$-algebraic quantum groups of separable type \cite{BJKVD_qgroupoid1}, 
\cite{BJKVD_qgroupoid2}, is an attempt at bridging this gap.  In particular, our modest goal is to establish and understand 
the relationship between the multiplicative partial isometries discussed below and the $C^*$-algebraic quantum groupoids 
of separable type.  By reducing the scope, the technical difficulties become milder.  On the other hand, while it is true that 
such quantum groupoids and the multiplicative partial isometries do not cover the full generality of the $C^*$-algebraic quantum 
groupoids, these intermediate steps have sufficiently rich structure to help us gain valuable insights toward the ultimate goal 
of developing a fully general $C^*$-algebraic theory of locally compact quantum groupoids. This is the underlying purpose.

This paper is organized as follows: In Section~\ref{sec1}, we gather some basic results and establish notations concerning partial isometries, 
then give conditions for a {\em multiplicative partial isometry\/}.  A couple of the conditions are variations of the pentagon equation
(in the sense of Baaj--Skandalis), and we also need two other conditions that would have been trivial in the unitary case. As a consequence, 
we can associate two subalgebras of ${\mathcal B}({\mathcal H})$.  While they are subalgebras and they show certain duality aspects, 
at this stage we do not know if they are ${}^*$-subalgebras.

In \S\ref{sub2.1} a certain ``fullness'' condition on $W$ is introduced, which helps us establish that these two subalgebras are represented 
non-degenerately on the Hilbert space.  This is a condition that is extra in our setting, which is automatically satisfied for a unitary operator.
Under the fullness assumption, we then define in \S\ref{sub2.2} the {\em manageability condition\/} for a multiplicative partial isometry 
operator $W$. This is our main assumption, which is motivated by Woronowicz's notion in the unitary case \cite{Wr7}.  Using this, 
it is possible to show that the pair of subalgebras obtained as a consequence of the multiplicativity property are in fact $C^*$-algebras. 

In Section~\ref{sec3}, we study the coalgebra structures on the $C^*$-algebra(s) associated with our partial isometry. The comultiplication 
is no longer nondegenerate, but the projection $E=W^*W$ can be regarded as $\Delta(1)$, and we can gather several of its properties. 
It plays an important role as the canonical idempotent.

We expect our $C^*$-bialgebra to have a quantum groupoid like structure. But then, in view of the general theory on locally compact 
quantum groups and quantum groupoids, that would mean further requiring the existence of certain left invariant and right invariant 
weights. Instead, by taking advantage of the manageability condition on our multiplicative partial isometry, it turns out that 
we can construct the antipode map without necessarily having to rely on the invariant weights. This is done in Section~\ref{sec4}. 
We construct the unitary antipode and the scaling group, from which the antipode is obtained in terms of a polar decomposition. It is 
noteworthy that even though the construction methods are different, we end up with the antipode map that satisfies the same properties 
as in the case of a {\em $C^*$-algebraic quantum groupoid of separable type\/}, as in \cite{BJKVD_qgroupoid1}, \cite{BJKVD_qgroupoid2}.

Finally, in Section~\ref{sec5}, we study four subspaces of ${\mathcal B}({\mathcal H})$ which are naturally associated with the projections $W^*W$ 
and $WW^*$. These spaces are also shown to be $C^*$-algebras. By studying their properties in relation to the canonical idempotent $E$ and 
by constructing certain densely-defined maps between them, we observe that they behave exactly like the source and the target algebras 
in the framework of a $C^*$-algebraic quantum groupoid of separable type. 

\bigskip

{\sc Acknowledgments.}

First, the author wishes to specifically acknowledge Alfons Van Daele, for his constant encouragement and support over the years, 
through various stages in the overall project on locally compact quantum groupoids.  For this particular project, the author wishes 
to thank Michel Enock and Jean Michel Vallin, who encouraged him to pursue this direction of research while the author was visiting 
Jussieu to report on the theory of the $C^*$-algebraic quantum groupoids of separable type in its initial stages. The author also 
wishes to thank Piotr Hajac and Banach Center (IMPAN, Warsaw) for their hospitality, where the first report on this work was given 
in 2019 (as part of the project ``New Geometry of Quantum Dynamics''), then further developed in a later visit. 

\bigskip

\section{Multiplicative partial isometries}\label{sec1}

Let ${\mathcal H}$ be a (separable) Hilbert space, not necessarily finite-dimensional.  Let $W\in{\mathcal B}({\mathcal H}\otimes{\mathcal H})$ 
be a partial isometry, satisfying $WW^*W=W$.

Write $E=W^*W$ and $G=WW^*$.  By the general theory on partial isometries, it is known that $E$ is a projection onto $\operatorname{Ran}(W^*W)
=\operatorname{Ran}(W^*)=\operatorname{Ker}(W)^\perp$, while $G$ is a projection onto $\operatorname{Ran}(WW^*)=\operatorname{Ran}(W)
=\operatorname{Ker}(W^*)^\perp$.  These spaces are necessarily closed in ${\mathcal H}\otimes{\mathcal H}$. In addition, $W$ is an isometry 
from $\operatorname{Ran}(W^*W)$ onto $\operatorname{Ran}(WW^*)$, and similarly, $W^*$ is an isometry from $\operatorname{Ran}(WW^*)$ 
onto $\operatorname{Ran}(W^*W)$. All these are standard results.

Write $\widehat{W}:=\Sigma W^*\Sigma$, where $\Sigma$ denotes the flip on ${\mathcal H}\otimes{\mathcal H}$. It is evident that $\widehat{W}$ 
is also a partial isometry, with the associated projections $\widehat{E}=\widehat{W}^*\widehat{W}=\Sigma WW^*\Sigma=\Sigma G\Sigma$ 
and $\widehat{G}=\widehat{W}\widehat{W}^*=\Sigma W^*W\Sigma=\Sigma E\Sigma$.

For $E=W^*W$, consider the following spaces:
$$
B:={\overline{\operatorname{span}\bigl\{(\operatorname{id}\otimes\omega)(W^*W):\omega
\in{\mathcal B}({\mathcal H})_*\bigr\}}}^{\|\ \|}\,\subseteq {\mathcal B}({\mathcal H}),
$$
$$
C:={\overline{\operatorname{span}\bigl\{(\omega\otimes\operatorname{id})(W^*W):\omega
\in{\mathcal B}({\mathcal H})_*\bigr\}}}^{\|\ \|}\,\subseteq {\mathcal B}({\mathcal H}).
$$
They are closed subspaces under the norm topology in ${\mathcal B}({\mathcal H})$, but at present we cannot expect them to be subalgebras. 
Similarly for $\widehat{E}=\Sigma WW^*\Sigma$, we can consider
$$
\widehat{B}:=
{\overline{\operatorname{span}\bigl\{(\operatorname{id}\otimes\omega)(\widehat{E}):\omega
\in{\mathcal B}({\mathcal H})_*\bigr\}}}^{\|\ \|}
={\overline{\operatorname{span}\bigl\{(\omega\otimes\operatorname{id})(WW^*):\omega
\in{\mathcal B}({\mathcal H})_*\bigr\}}}^{\|\ \|},
$$
$$
\widehat{C}:={\overline{\operatorname{span}\bigl\{(\omega\otimes\operatorname{id})(\widehat{E}):\omega
\in{\mathcal B}({\mathcal H})_*\bigr\}}}^{\|\ \|}
={\overline{\operatorname{span}\bigl\{(\operatorname{id}\otimes\omega)(WW^*):\omega
\in{\mathcal B}({\mathcal H})_*\bigr\}}}^{\|\ \|},
$$
which are also norm-closed subspaces in ${\mathcal B}({\mathcal H})$.

While we cannot yet claim that $B$, $C$, $\widehat{B}$, $\widehat{C}$ are subalgebras, they are all closed under taking adjoints.  
That will be useful.  We will come back to study these objects later (in Section~\ref{sec5}).

Let us begin our discussion on multiplicative partial isometries, first by giving the definition:

\begin{defn}\label{mpi_defn}
Let $W\in{\mathcal B}({\mathcal H}\otimes{\mathcal H})$ be a partial isometry. We will call $W$ a {\em multiplicative partial isometry\/}, 
if the following four conditions hold on ${\mathcal H}\otimes{\mathcal H}\otimes{\mathcal H}$:
\begin{align}
W_{23}W_{12}W^*_{23}&=W_{12}W_{13} \label{(mpi1)}  \\
W_{12}^*W_{23}W_{12}&=W_{13}W_{23} \label{(mpi2)}  \\
W_{23}^*W_{23}W_{12}&=W_{12}W_{23}^*W_{23} \label{(mpi3)} \\
W_{12}W_{12}^*W_{23}&=W_{23}W_{12}W_{12}^* \label{(mpi4)} 
\end{align}
\end{defn}

\begin{rem}
Here, we are using the standard three-leg notation.  Equations~\eqref{(mpi1)} and \eqref{(mpi2)} resemble the ``pentagon equation'', 
as in the case of multiplicative unitaries \cite{BS}, \cite{Wr7}. However, with $W$ not being a unitary, these two conditions are not 
necessarily equivalent.   Equations~\eqref{(mpi3)}, \eqref{(mpi4)} become trivial in the unitary case, but they are needed for our purposes. 
Equation~\eqref{(mpi3)} would imply that the elements of the spaces $B$ and $C$ commute, while Equation~\eqref{(mpi4)} gives 
the commutativity between the elements of $\widehat{B}$ and $\widehat{C}$.  See Proposition~\ref{BCcommute} below. Meanwhile, 
Equations~\eqref{(mpi1)} and \eqref{(mpi2)} allow the construction of two subalgebras ${\mathcal A}$ and $\widehat{\mathcal A}$ 
of ${\mathcal B}({\mathcal H})$, on which we can later build the quantum groupoid structures.  See Proposition~\ref{AAhatalgebras} below. 
\end{rem}

Here are some immediate consequences of Definition~\ref{mpi_defn}.

\begin{prop} \label{BCcommute} 
Let $W$ be a multiplicative partial isometry. Consider the subspaces $B$, $C$, $\widehat{B}$, $\widehat{C}$ in ${\mathcal B}({\mathcal H})$ 
as above. We have:
\begin{enumerate}
 \item For any $b\in B$ and $c\in C$, we have $bc=cb$.
 \item For any $\hat{b}\in\widehat{B}$ and $\hat{c}\in\widehat{C}$, we have $\hat{b}\hat{c}=\hat{c}\hat{b}$.
\end{enumerate}
\end{prop}

\begin{proof}
(1). Consider $b=(\operatorname{id}\otimes\omega)(W^*W)\in B$ and $c=(\omega'\otimes\operatorname{id})(W^*W)\in C$, 
for any $\omega,\omega'\in{\mathcal B}({\mathcal H})_*$.  We can see that 
$$
cb=(\omega'\otimes\operatorname{id}\otimes\omega)(W^*_{12}W_{12}W^*_{23}W_{23})
=(\omega'\otimes\operatorname{id}\otimes\omega)(W^*_{23}W_{23}W^*_{12}W_{12})
=bc,
$$
by applying Equation~\eqref{(mpi3)} twice. 

(2). Proof of (2) is similar, now using Equation~\eqref{(mpi4)}.
\end{proof}

Given a multiplicative partial isometry $W\in{\mathcal B}({\mathcal H}\otimes{\mathcal H})$, we can associate to it the following 
two subalgebras of ${\mathcal B}({\mathcal H})$. The proof given below is essentially the same as in \cite{BS}.

\begin{prop} \label{AAhatalgebras}
Let $W$ be a multiplicative partial isometry. Consider the subspaces ${\mathcal A}$ and $\widehat{\mathcal A}$ in ${\mathcal B}({\mathcal H})$, 
defined as follows:
$$
{\mathcal A}:=\operatorname{span}\bigl\{(\operatorname{id}\otimes\omega)(W):\omega\in{\mathcal B}({\mathcal H})_*\bigr\}
\ {\text { and }}
\ 
\widehat{\mathcal A}:=\operatorname{span}\bigl\{(\omega\otimes\operatorname{id})(W):\omega\in{\mathcal B}({\mathcal H})_*\bigr\}.
$$
Both ${\mathcal A}$ and $\widehat{\mathcal A}$ are subalgebras of ${\mathcal B}({\mathcal H})$. 
\end{prop}

\begin{proof}
(1). Consider $x=(\operatorname{id}\otimes\omega)(W),x'=(\operatorname{id}\otimes\omega')(W)
\in{\mathcal A}$, where $\omega,\omega'\in{\mathcal B}({\mathcal H})_*$ are arbitrary. 
By Equation~\eqref{(mpi1)}, we have
$$
xx'=(\operatorname{id}\otimes\omega\otimes\omega')(W_{12}W_{13})
=(\operatorname{id}\otimes\omega\otimes\omega')(W_{23}W_{12}W^*_{23}) 
=(\operatorname{id}\otimes\theta)(W)\,\in{\mathcal A},
$$
where $\theta\in{\mathcal B}({\mathcal H})_*$ is such that $\theta(X)=(\omega\otimes\omega')(W(X\otimes1)W^*)$, 
for $X\in{\mathcal B}({\mathcal H})$.

(2). Proof for $\widehat{\mathcal A}$ being also a subalgebra is similar.  For $y=(\omega\otimes\operatorname{id})(W),
y'=(\omega'\otimes\operatorname{id})(W)\in\widehat{\mathcal A}$, by using Equation~\eqref{(mpi2)}, we can show that 
$$
yy'=(\omega\otimes\omega'\otimes\operatorname{id})(W_{13}W_{23})=(\omega'\otimes\omega\otimes\operatorname{id})
(W_{12}^*W_{23}W_{12})=(\theta\otimes\operatorname{id})(W)\,\in\widehat{\mathcal A},
$$
where $\theta\in{\mathcal B}({\mathcal H})_*$ is such that $\theta(X)=(\omega\otimes\omega')(W^*(1\otimes X)W)$, 
$\forall X\in{\mathcal B}({\mathcal H})$.
\end{proof}

\begin{rem}
It is not difficult to see that if $W$ is a multiplicative partial isometry, then $\widehat{W}=\Sigma W^*\Sigma$ is also multiplicative. 
Note also that in terms of $\widehat{W}$, we have:
$$
\widehat{\mathcal A}^*=\operatorname{span}\bigl\{(\omega\otimes\operatorname{id})(W^*):\omega\in{\mathcal B}({\mathcal H})_*\bigr\}
=\operatorname{span}\bigl\{(\operatorname{id}\otimes\omega)(\widehat{W}):\omega\in{\mathcal B}({\mathcal H})_*\bigr\}.
$$
At present, we do not know if $\widehat{\mathcal A}^*=\widehat{\mathcal A}$, however.
\end{rem}

Here are some more consequences of Definition~\ref{mpi_defn}:

\begin{lem} \label{mpi_lem1}
Let $W$ be a multiplicative partial isometry. Then the following results hold:
\begin{align}
W_{12}W_{13}W_{23}&=W_{23}W_{12} \label{(mpi5)}  \\
W_{12}^*W_{12}W_{13}&=W_{13}W_{23}W_{23}^* \label{(mpi6)}
\end{align}
\end{lem}

\begin{proof}
From Equation~\eqref{(mpi2)}, we have $W^*_{12}W_{23}W_{12}=W_{13}W_{23}$. Multiply $W_{12}$ to both sides, 
to obtain $W_{12}W^*_{12}W_{23}W_{12}=W_{12}W_{13}W_{23}$.  Apply Equation~\eqref{(mpi4)} to the left side, 
which becomes $W_{23}W_{12}W^*_{12}W_{12}=W_{23}W_{12}$, as $W$ is a partial isometry.  In this way, 
we prove that $W_{23}W_{12}=W_{12}W_{13}W_{23}$.

By Equation~\eqref{(mpi1)}, we have $W_{12}^*W_{12}W_{13}=W^*_{12}W_{23}W_{12}W^*_{23}$. Apply Equation~\eqref{(mpi2)} 
to the right side, to obtain $W_{12}^*W_{12}W_{13}=W_{13}W_{23}W^*_{23}$.
\end{proof}

\begin{rem}
Equation~\eqref{(mpi5)} is exactly the pentagon equation of Baaj--Skandalis \cite{BS}.  Here, we obtain it as a consequence. 
Note that the Equations~\eqref{(mpi5)}, \eqref{(mpi6)}, \eqref{(mpi3)}, \eqref{(mpi4)} have been chosen as the axioms 
by B{\" o}hm and Szlachnyi in \cite{BSzMultIso}.  It is not difficult to show that these four imply the four conditions 
\eqref{(mpi1)}, \eqref{(mpi2)}, \eqref{(mpi3)}, \eqref{(mpi4)} chosen in our Definition~\ref{mpi_defn}, and vice versa.
\end{rem}

Let us construct maps $\Delta$ and $\widehat{\Delta}$, which would become comultiplications later, at first as maps from 
${\mathcal B}({\mathcal H})$ into ${\mathcal B}({\mathcal H}\otimes{\mathcal H})$:

\begin{prop}\label{deltadeltahat}
Define the following two maps $\Delta$ and $\widehat{\Delta}$, from ${\mathcal B}({\mathcal H})$ into ${\mathcal B}({\mathcal H}\otimes{\mathcal H})$:
$$
\Delta(X)=W^*(1\otimes X)W \ {\text { and }} \  \widehat{\Delta}(X)=\Sigma W(X\otimes1)W^*\Sigma, 
\  {\text { for $X\in{\mathcal B}({\mathcal H})$.}}
$$
We have the following ``coassociativity'' property, which will be useful later:
\begin{enumerate}
 \item $\Delta:{\mathcal B}({\mathcal H})\to{\mathcal B}({\mathcal H}\otimes{\mathcal H})$ satisfies the property: 
 $(\Delta\otimes\operatorname{id})\Delta(X)=(\operatorname{id}\otimes\Delta)\Delta(X)$, $\forall X\in{\mathcal B}({\mathcal H})$.
 \item $\widehat{\Delta}:{\mathcal B}({\mathcal H})\to{\mathcal B}({\mathcal H}\otimes{\mathcal H})$ satisfies the property: 
$(\widehat{\Delta}\otimes\operatorname{id})\widehat{\Delta}(X)=(\operatorname{id}\otimes\widehat{\Delta})\widehat{\Delta}(X)$, 
$\forall X\in{\mathcal B}({\mathcal H})$.
\end{enumerate}
\end{prop}

\begin{proof}
For $X\in{\mathcal B}({\mathcal H})$, we have: 
\begin{align}
(\Delta\otimes\operatorname{id})\Delta(X)&=W^*_{12}W^*_{23}(1\otimes1\otimes X)W_{23}W_{12}
=W^*_{23}W^*_{13}W^*_{12}(1\otimes1\otimes X)W_{12}W_{13}W_{23} \notag \\
&=W^*_{23}W^*_{13}(1\otimes1\otimes X)W^*_{12}W_{12}W_{13}W_{23}
=W^*_{23}W^*_{13}(1\otimes1\otimes X)W_{13}W_{23}W^*_{23}W_{23} \notag \\
&=W^*_{23}W^*_{13}(1\otimes1\otimes X)W_{13}W_{23}=(\operatorname{id}\otimes\Delta)\Delta(X),
\notag
\end{align}
where we used Equations~\eqref{(mpi5)} and \eqref{(mpi6)}, together with the fact that $W$ is a partial isometry. 

Proof for $\widehat{\Delta}$ is similar.  We may either give the proof directly, or use the multiplicativity property 
of $\widehat{W}$ and use the result above, as we can write $\widehat{\Delta}(X)=\widehat{W}^*(1\otimes X)\widehat{W}$.
\end{proof}

Consider the norm-closures of the algebras ${\mathcal A}$ and $\widehat{\mathcal A}$ in ${\mathcal B}({\mathcal H})$. 
That is, define:
$$
A:=\overline{\operatorname{span}\bigl\{(\operatorname{id}\otimes\omega)(W):\omega\in{\mathcal B}({\mathcal H})_*\bigr\}}^{\|\ \|}
\  {\text { and }}
\ 
\widehat{A}:=\overline{\operatorname{span}\bigl\{(\omega\otimes\operatorname{id})(W):\omega\in{\mathcal B}({\mathcal H})_*\bigr\}}^{\|\ \|}.
$$
Eventually, they will be shown to be $C^*$-algebras and become our main objects of study. The maps $\Delta$ and $\widehat{\Delta}$ will be 
restricted to $A$ and $\widehat{A}$, on which we will construct the quantum groupoid structures. 

However, some extra conditions need to be introduced for our program to work.  For instance, unlike 
$B$, $C$, $\widehat{B}$, $\widehat{C}$, there is no reason to believe that $A$ and $\widehat{A}$ would be self-adjoint.  
This was already a problem even when $W$ is a multiplicative unitary, so some extra conditions like the ``regularity'' 
(see section 3 of \cite{BS}) or the ``manageability'' (see \cite{Wr7}) had to be assumed to ensure that $A$ and $\widehat{A}$ 
are closed under taking the involution.  We will discuss this matter in the ensuing sections.

Before wrapping up this section, here are some more consequences of the operator $W$ being a multiplicative partial isometry:

\begin{lem} 
Let $W$ be a multiplicative partial isometry. Then the following results hold:
\begin{align}
W_{12}W^*_{23}&=W^*_{23}W_{12}W_{13} \label{(mpi7)}  \\
W_{12}^*W_{23}&=W_{13}W_{23}W^*_{12} \label{(mpi8)}  \\
W_{13}^*W_{13}W_{23}&=W_{23}W^*_{12}W_{12} \label{(mpi9)}  \\
W_{12}W_{13}W^*_{13}&=W_{23}W^*_{23}W_{12} \label{(mpi10)} 
\end{align}
\end{lem}

\begin{proof}
From Equation~\eqref{(mpi1)}, we have $W^*_{23}W_{23}W_{12}W^*_{23}=W^*_{23}W_{12}W_{13}$.  
Then apply Equation~\eqref{(mpi3)} to the left side, which becomes $W_{12}W^*_{23}W_{23}W^*_{23}=W_{12}W^*_{23}$, 
because $W^*WW^*=W^*$.  Combining, we prove Equation~\eqref{(mpi7)}.
Similarly, for Equation~\eqref{(mpi8)}, use Equation~\eqref{(mpi2)} and Equation~\eqref{(mpi4)}.

From Equation~\eqref{(mpi2)}, we have: $W^*_{13}W_{13}W_{23}=W^*_{13}W^*_{12}W_{23}W_{12}$.
Note that from Equation~\eqref{(mpi7)} we know $W^*_{13}W^*_{12}W_{23}=W_{23}W^*_{12}$. 
Combining, we obtain $W^*_{13}W_{13}W_{23}=W_{23}W^*_{12}W_{12}$, thereby proving Equation~\eqref{(mpi9)}. 
Equation~\eqref{(mpi10)} can be proved similarly, using Equations~\eqref{(mpi1)} and \eqref{(mpi8)}.
\end{proof}

\section{The manageability condition}\label{sec2}

\subsection{Fullness condition}\label{sub2.1}

Let $W$ be a multiplicative partial isometry, and consider the subalgebras ${\mathcal A}$ and $\widehat{\mathcal A}$. 
As we do not know if they are ${}^*$-algebras, even in the finite-dimensional case we cannot be sure whether they are 
unital subalgebras.  This is different from the case of multiplicative unitaries: For a multiplicative unitary, if the Hilbert space 
on which it is acting is finite-dimensional, then it is known that the norm-closures of ${\mathcal A}$ and $\widehat{\mathcal A}$ 
always become unital $C^*$-algebras (finite-dimensional Kac algebras).  See Theorem~4.10 of \cite{BS}.

To see what can happen in the general case, observe the example below. (This is essentially the example given by 
B{\" o}hm and Szlachnyi in \cite{BSzMultIso}, with only minor differences.)

\begin{ex} \label{example_nonunital}
Let ${\mathcal H}=\mathbb{C}^2$ and consider $W=e_{21}\otimes e_{11} + e_{22}\otimes e_{22}$, where the $e_{ij}
\in{\mathcal B}({\mathcal H})$, $1\le i,j\le 2$, are the matrix units such that $e_{ij}(\mathbf{v}):=\langle\mathbf{v},\xi_j\rangle\xi_i$. 
[Here $(\xi_k)$ denotes the standard orthonormal basis for ${\mathcal H}$. Note that $e_{ij}e_{kl}=\delta_{jk}e_{il}$ and $(e_{ij})^*=e_{ji}$.]

By Linear Algebra, it is easy to verify that $W$ is a partial isometry ($WW^*W=W$). We also have:
\begin{align}
W_{23}W_{12}W^*_{23}&=(1\otimes e_{21}\otimes e_{11}+1\otimes e_{22}\otimes e_{22})
(e_{21}\otimes e_{11}\otimes1 + e_{22}\otimes e_{22}\otimes1)W^*_{23} \notag \\
&=(e_{21}\otimes e_{21}\otimes e_{11} + e_{22}\otimes e_{22}\otimes e_{22})(1\otimes e_{12}\otimes e_{11}+1\otimes e_{22}\otimes e_{22}) 
\notag \\
&=e_{21}\otimes e_{22}\otimes e_{11} + e_{22}\otimes e_{22}\otimes e_{22},
\notag
\end{align}
while 
\begin{align}
W_{12}W_{13}&
=(e_{21}\otimes e_{11}\otimes1+e_{22}\otimes e_{22}\otimes1)(e_{21}\otimes 1\otimes e_{11} + e_{22}\otimes 1\otimes e_{22}) 
\notag \\
&=e_{21}\otimes e_{22}\otimes e_{11} + e_{22}\otimes e_{22}\otimes e_{22}.
\notag
\end{align}
Comparing, we verify Equation~\eqref{(mpi1)}: $W_{23}W_{12}W^*_{23}=W_{12}W_{13}$.
Equations~\eqref{(mpi2)}, \eqref{(mpi3)}, \eqref{(mpi4)} are also easily verified, so $W$ is indeed a multiplicative partial 
isometry.

However, if we consider ${\mathcal A}:=\operatorname{span}\bigl\{(\operatorname{id}\otimes\omega)(W):\omega
\in{\mathcal B}({\mathcal H})_*\bigr\}$, we can quickly observe that $e_{11}+e_{22}\notin{\mathcal A}$. In fact, the subalgebra 
${\mathcal A}$ is given by the matrices with the first row equal to zero, and thus non-unital. 
[The other subalgebra, $\widehat{\mathcal A}$, given by all diagonal matrices, is actually unital.]
\end{ex}

Considering that the associated subalgebras for a multiplicative unitary are always unital, this observation means that already 
in the finite-dimensional case an additional condition on $W$ is required. See  \cite{BSzMultIso}.

In our infinite-dimensional case, we cannot expect  ${\mathcal A}$ and $\widehat{\mathcal A}$ to be unital.  Nonetheless, 
it is apparent that some additional assumption on $W$ is needed.  That turns out to be related to non-degeneracy, 
so we will from now on require the following {\em fullness\/} condition on $W$:

\begin{defn} \label{fullness_condition}
Let $W\in{\mathcal B}({\mathcal H}\otimes{\mathcal H})$.  We will say that $W$ is {\em full\/}, if for any $\xi\in{\mathcal H}$, $\xi\ne0$, 
we can find $p,q,r,s\in{\mathcal H}$ such that $W(\xi\otimes p)\ne0$, $W(q\otimes\xi)\ne0$, $W^*(\xi\otimes r)\ne0$, $W^*(s\otimes\xi)\ne0$.
\end{defn}

The fullness condition may be defined for any operator in ${\mathcal B}({\mathcal H}\otimes{\mathcal H})$. But for a multiplicative partial isometry $W$, 
being full implies the non-degeneracy of its associated algebras. See below: 

\begin{lem}\label{Anondegnerate}
Let $W\in{\mathcal B}({\mathcal H}\otimes{\mathcal H})$ be a multiplicative partial isometry satisfying the fullness condition. Then its associated 
subalgebras ${\mathcal A},\widehat{\mathcal A}\,\bigl(\subseteq{\mathcal B}({\mathcal H})\bigr)$, as well as 
${\mathcal A}^*,\widehat{\mathcal A}^*\,\bigl(\subseteq{\mathcal B}({\mathcal H})\bigr)$ are such that ${\mathcal A}{\mathcal H}$, $\widehat{\mathcal A}{\mathcal H}$, 
${\mathcal A}^*{\mathcal H}$, $\widehat{\mathcal A}^*{\mathcal H}$ are dense subspaces in ${\mathcal H}$
\end{lem}

\begin{proof}
Suppose $\xi\in{\mathcal H}$, $\xi\ne0$, is such that $\langle a\zeta,\xi\rangle=0$,  $\forall a\in{\mathcal A}$, $\forall \zeta\in{\mathcal H}$. By the 
fullness of $W$, we can find $r\in{\mathcal H}$ such that $W^*(\xi\otimes r)\ne0$. Then we can find $p,q\in{\mathcal H}$ such that $\bigl\langle p\otimes q,
W^*(\xi\otimes r)\bigr\rangle\ne0$, or equivalently $\bigl\langle W(p\otimes q),\xi\otimes r\bigr\rangle\ne0$. But then, this means 
$\bigl\langle(\operatorname{id}\otimes\omega_{q,r})(W)p,\xi\bigr\rangle\ne0$, where $\omega_{q,r}$ is the linear functional defined by 
$\omega_{q,r}(X)=\langle Xq,r\rangle$, for $X\in{\mathcal B}({\mathcal H})$. As $a=(\operatorname{id}\otimes\omega_{q,r})(W)\in{\mathcal A}$, 
this is a contradiction. So the zero vector is the only vector that is orthogonal to all the elements of the form $a\zeta$, for $a\in{\mathcal A}$, $\zeta\in{\mathcal H}$. 
This shows that ${\mathcal A}{\mathcal H}$ is dense in ${\mathcal H}$. 

Similar proof can be given for $\widehat{\mathcal A}{\mathcal H}$, ${\mathcal A}^*{\mathcal H}$, $\widehat{\mathcal A}^*{\mathcal H}$.
\end{proof}

\begin{rem}
The fullness condition is automatically satisfied if $W$ is unitary. On the other hand, the partial isometry operator considered in Example~\ref{example_nonunital} 
does not satisfy the fullness condition, as $\eta=(1,0)\in\mathbb{C}$ is orthogonal to ${\mathcal A}{\mathcal H}$.

Meanwhile, the linear functional introduced in the proof, namely $\omega_{q,r}\in{\mathcal B}({\mathcal H})_*$ for $q,r\in{\mathcal H}$, is a standard 
notation, defined by $\omega_{q,r}(X)=\langle Xq,r\rangle$, for $X\in{\mathcal B}({\mathcal H})$.  We will use this a lot.  
Such functionals are dense in ${\mathcal B}({\mathcal H})_*$.
\end{rem}

\subsection{Manageability condition}\label{sub2.2}

From now on, we will assume that $W\in{\mathcal B}({\mathcal H}\otimes{\mathcal H})$ is a multiplicative partial isometry satisfying the fullness condition 
(see \S\ref{sec1} and \S\ref{sub2.1}).  Motivated by Woronowicz's notion of the manageability for a multiplicative unitary \cite{Wr7}, let us now introduce 
the {\em manageability\/} condition for a multiplicative partial isometry, then gather some resulting properties.

We denote by $\overline{\mathcal H}$ the complex conjugate of the Hilbert space ${\mathcal H}$. For any $\xi\in{\mathcal H}$, the corresponding element 
will be denoted by $\bar{\xi}$.  The map ${\mathcal H}\ni\xi\mapsto\bar{\xi}\in\overline{\mathcal H}$ is a ${}^*$-anti-isomorphism.  For $\xi,\eta\in{\mathcal H}$, 
we will have $\langle\bar{\xi},\bar{\eta}\rangle=\langle\eta,\xi\rangle$.

If $m$ is a closed operator on ${\mathcal H}$, then its transpose, written $m^{\top}$, is the operator on $\overline{\mathcal H}$ such that 
${\mathcal D}(m^{\top})=\overline{{\mathcal D}(m^*)}$ and $m^{\top}\bar{\xi}=\overline{m^*\xi}$, for $\xi\in{\mathcal D}(m^*)$.  In particular, 
if $m\in{\mathcal B}({\mathcal H})$, then $m^{\top}\in{\mathcal B}(\overline{\mathcal H})$ such that 
$\langle m^{\top}\bar{\eta},\bar{\xi}\rangle=\langle\xi,m^*\eta\rangle=\langle m\xi,\eta\rangle$, for $\xi,\eta\in{\mathcal H}$. 
It is clear that $m\mapsto m^{\top}$ is a ${}^*$-anti-isomorphism.  We may identify $\overline{\overline{\mathcal H}}={\mathcal H}$, by 
$\overline{\overline{\xi}}=\xi$.  Then we have $(m^{\top})^{\top}=m$, for any $m\in{\mathcal B}({\mathcal H})$.

With these notations set, we now give the definition for the manageability condition:

\begin{defn}\label{manageable}
Let $W\in{\mathcal B}({\mathcal H}\otimes{\mathcal H})$ be a multiplicative partial isometry. We say $W$ is {\em manageable}, if there exist 
a densely-defined positive closed operator $Q$ acting on ${\mathcal H}$, $\operatorname{Ker}(Q)=\{0\}$, and an operator 
$\widetilde{W}\in{\mathcal B}(\overline{\mathcal H}\otimes{\mathcal H})$, such that 
\begin{enumerate}
 \item $W(Q\otimes Q)\subseteq (Q\otimes Q)W$.
 \item 
$\bigl\langle W(\xi\otimes v),\eta\otimes u\bigr\rangle=\bigl\langle\widetilde{W}(\bar{\eta}\otimes Q^{-1}v),\bar{\xi}\otimes Qu\bigr\rangle$, 
for any $\xi,\eta\in{\mathcal H}$, $v\in{\mathcal D}(Q^{-1})$, $u\in{\mathcal D}(Q)$.
 \item We also require: $\widetilde{W}_{13}\widetilde{W}_{23}\widetilde{W}^*_{23}=W^{\top\otimes\top}_{12}[W^*_{12}]^{\top\otimes\top}\widetilde{W}_{13}$ 
 (as operators on $\overline{\mathcal H}\otimes\overline{\mathcal H}\otimes{\mathcal H}$), and that 
 $W_{23}W_{23}^*\widetilde{W}_{13}=\widetilde{W}_{13}\widetilde{W}_{12}{\widetilde{W}}^*_{12}$ 
 (as operators on $\overline{\mathcal H}\otimes{\mathcal H}\otimes{\mathcal H}$).
\end{enumerate}
\end{defn}

\begin{rem}
This is a modification of Woronowicz's notion (see Definition~1.2 in \cite{Wr7}).  In (1), we replaced his condition $W^*(Q\otimes Q)W
=Q\otimes Q$, which is no longer true as $W$ is not unitary, with the inclusion above.  The characterizing equation in (2) is the 
same as in the unitary case.  Meanwhile we included the two conditions in (3), which would have been trivial when $W$ and 
$\widetilde{W}$ are unitaries.
\end{rem}

\begin{rem}
In case $W$ is unitary, the operator $\widetilde{W}$ is also unitary \cite{Wr7}. We will not require any condition on $\widetilde{W}$ here, 
other than being bounded and satisfying the conditions above. It eventually turns out that the operator $\widetilde{W}$ is itself a partial isometry. 
The proof will be given later in this section.
\end{rem}

In the below, we give some consequences of the inclusion, $W(Q\otimes Q)\subseteq (Q\otimes Q)W$.

\begin{lem}\label{WQQinequality}
Write $E=W^*W$ and $G=WW^*$ as before. We have:
\begin{enumerate}
\item $(Q\otimes Q)E=E(Q\otimes Q)E$ and $(Q\otimes Q)G=G(Q\otimes Q)G$.
\item It follows as a result that $(Q\otimes Q)|_{\operatorname{Ran}(E)}$, $(Q\otimes Q)|_{\operatorname{Ran}(G)}$, 
$(Q\otimes Q)|_{\operatorname{Ker}(W)}$, $(Q\otimes Q)|_{\operatorname{Ker}(W^*)}$ become operators on the respective subspaces 
$\operatorname{Ran}(E)$, $\operatorname{Ran}(G)$, $\operatorname{Ker}(W)$, $\operatorname{Ker}(W^*)$.
\item For any $z\in\mathbb{C}$, we have: 
$W(Q^z\otimes Q^z)\subseteq (Q^z\otimes Q^z)W$ and $W^*(Q^z\otimes Q^z)\subseteq (Q^z\otimes Q^z)W^*$.
\end{enumerate}
\end{lem}

\begin{proof}
See the results and proofs for Propositions~4.12 -- 4.17 in \cite{BJKVD_qgroupoid2}, where a similar argument was carried out 
in more detail.

These results are consequences of the fact that $W$ and $W^*$ are partial isometries.  When restricted to subspaces, 
we may regard $W|_{\operatorname{Ran}(E)}$ and $W^*|_{\operatorname{Ran}(G)}$ as onto isometries between the subspaces 
$\operatorname{Ran}(E)$ and $\operatorname{Ran}(G)$.  As a consequence, a version of functional calculus can be applied. 
\end{proof}

In Lemma~\ref{WQQinequality}\,(3), we cannot do better than ``$\subseteq$'' in general.  However, if $z\in\mathbb{C}$ is purely imaginary, 
that is $z=it$ for $t\in\mathbb{R}$, the operator $Q^{it}$ is bounded.  So the domain ${\mathcal D}(Q^{it}\otimes Q^{it})$ becomes the whole 
space ${\mathcal H}\otimes{\mathcal H}$, and we obtain the following result: 

\begin{prop}\label{WQQequality}
Let $t\in\mathbb{R}$. Then the following equality holds on the whole space ${\mathcal H}\otimes{\mathcal H}$:
$$
(Q^{it}\otimes Q^{it})W(Q^{-it}\otimes Q^{-it})=W.
$$
\end{prop}

\begin{proof}
Since $Q^{it}$ is a bounded operator, there is no issue with the domains.  As we already know 
$W(Q^{it}\otimes Q^{it})\subseteq(Q^{it}\otimes Q^{it})W$ from Lemma~\ref{WQQinequality}, we indeed have the equality: 
$W(Q^{it}\otimes Q^{it})=(Q^{it}\otimes Q^{it})W$.  This is equivalent to $(Q^{it}\otimes Q^{it})W(Q^{-it}\otimes Q^{-it})=W$.
\end{proof}

Let us turn our attention back to the operator $\widetilde{W}$.  As a consequence of Proposition~\ref{WQQequality} 
and the characterizing equation for the manageability, we obtain the following result:

\begin{prop}\label{WtildeQQequality}
Let $W$ be a manageable multiplicative partial isometry, and let $Q$ and $\widetilde{W}$ be the associated operators 
given in Definition~\ref{manageable}. For any $t\in\mathbb{R}$, we have the following equality on the whole space 
$\overline{\mathcal H}\otimes{\mathcal H}$: 
$$
\bigl([Q^{\top}]^{-it}\otimes Q^{it}\bigr)\widetilde{W}\bigl([Q^{\top}]^{it}\otimes Q^{-it}\bigr)=\widetilde{W}.
$$
\end{prop}

\begin{proof}
Suppose $\xi,\eta\in{\mathcal H}$ and $v\in{\mathcal D}(Q^{-1})$, $u\in{\mathcal D}(Q)$.  Then for $t\in\mathbb{R}$, we can 
also say that $Q^{-it}\xi,Q^{-it}\eta\in{\mathcal H}$ and $Q^{-it}v\in{\mathcal D}(Q^{-1})$, $Q^{-it}u\in{\mathcal D}(Q)$, for instance 
by writing $QQ^{-it}u=Q^{-it}Qu$.

By (2) of Definition~\ref{manageable}, we have:
$$
\bigl\langle W(\xi\otimes v),\eta\otimes u\bigr\rangle=\bigl\langle\widetilde{W}(\bar{\eta}\otimes Q^{-1}v),\bar{\xi}\otimes Qu\bigr\rangle.
$$
Since we know $W=(Q^{it}\otimes Q^{it})W(Q^{-it}\otimes Q^{-it})$ from Proposition~\ref{WQQequality}, the left side of the above equation 
can be expressed as follows:
\begin{align}
(LHS)&=\bigl\langle(Q^{it}\otimes Q^{it})W(Q^{-it}\xi\otimes Q^{-it}v),\eta\otimes u\bigr\rangle
=\bigl\langle W(Q^{-it}\xi\otimes Q^{-it}v),Q^{-it}\eta\otimes Q^{-it}u\bigr\rangle \notag \\
&=\bigl\langle\widetilde{W}(\overline{Q^{-it}\eta}\otimes Q^{-1}Q^{-it}v),\overline{Q^{-it}\xi}\otimes QQ^{-it}u\bigr\rangle
=\bigl\langle\widetilde{W}([Q^{\top}]^{it}\bar{\eta}\otimes Q^{-it}Q^{-1}v),([Q^{\top}]^{it}\bar{\xi}\otimes Q^{-it}Qu\bigr\rangle  \notag \\
&=\bigl\langle([Q^{\top}]^{-it}\otimes Q^{it})\widetilde{W}([Q^{\top}]^{it}\otimes Q^{-it})(\bar{\eta}\otimes Q^{-1}v),\bar{\xi}\otimes Qu\bigr\rangle,
\notag
\end{align}
where we used Definition~\ref{manageable}\,(2) in the third equality.  As $\xi,\eta,v,u$ are arbitrary, it follows that 
$$
\widetilde{W}=\bigl([Q^{\top}]^{-it}\otimes Q^{it}\bigr)\widetilde{W}\bigl([Q^{\top}]^{it}\otimes Q^{-it}\bigr)
=(Q^{\top}\otimes Q^{-1})^{-it}\widetilde{W}(Q^{\top}\otimes Q^{-1})^{it},
$$
which is true for all $t\in\mathbb{R}$.
\end{proof}

As a consequence of Proposition~\ref{WtildeQQequality}, which holds true for all $t\in\mathbb{R}$, we can see that the operators 
$\widetilde{W}$ and $\bigl([Q^{-1}]^{\top}\otimes Q\bigr)\widetilde{W}(Q^{\top}\otimes Q^{-1})$ will agree whenever they are valid. 
Considering the domains, we thus obtain the following result:
\begin{equation}\label{(WtildeQQinclusion)}
\widetilde{W}(Q^{\top}\otimes Q^{-1})\subseteq (Q^{\top}\otimes Q^{-1})\widetilde{W} 
\ {\text { and also }} \ 
\widetilde{W}\bigl([Q^{-1}]^{\top}\otimes Q\bigr)\subseteq \bigl([Q^{-1}]^{\top}\otimes Q\bigr)\widetilde{W}.
\end{equation}

We formulate below an alternative characterizing equation that is equivalent to (2) of Definition~\ref{manageable}.  This 
will be useful throughout the paper.

\begin{prop}\label{manageable_alt}
Let $W$ be a manageable multiplicative partial isometry, and let $Q$ and $\widetilde{W}$ be as in Definition~\ref{manageable}.  Then for 
any $\xi\in{\mathcal D}(Q)$, $\eta\in{\mathcal D}(Q^{-1})$ and any $v,u\in{\mathcal H}$, we have:
$$
\bigl\langle W(\xi\otimes v),\eta\otimes u\bigr\rangle=\bigl\langle\widetilde{W}([Q^{-1}]^{\top}\bar{\eta}\otimes v),Q^{\top}\bar{\xi}\otimes u\bigr\rangle.
$$
\end{prop}

\begin{proof}
Let $\xi\in{\mathcal D}(Q)$, $\eta\in{\mathcal D}(Q^{-1})$. For the time being let $v\in{\mathcal D}(Q^{-1})$, $u\in{\mathcal D}(Q)$ and write $u=Q^{-1}Qu$. 
Use the inclusion $\widetilde{W}(Q^{\top}\otimes Q^{-1})\subseteq(Q^{\top}\otimes Q^{-1})\widetilde{W}$, and compute. Then 
\begin{align}
\bigl\langle\widetilde{W}([Q^{-1}]^{\top}\bar{\eta}\otimes v),Q^{\top}\bar{\xi}\otimes u\bigr\rangle 
&=\bigl\langle(Q^{\top}\otimes Q^{-1})\widetilde{W}([Q^{-1}]^{\top}\bar{\eta}\otimes v),\bar{\xi}\otimes Qu\bigr\rangle \notag \\
&=\bigl\langle\widetilde{W}(\bar{\eta}\otimes Q^{-1}v),\bar{\xi}\otimes Qu\bigr\rangle 
=\bigl\langle W(\xi\otimes v),\eta\otimes u\bigr\rangle,
\notag
\end{align}
which is valid because  $[Q^{-1}]^{\top}\bar{\eta}\otimes v\in{\mathcal D}(Q^{\top}\otimes Q^{-1})$. This is true for any $v\in{\mathcal D}(Q^{-1})$ 
and $u\in{\mathcal D}(Q)$, but considering that $W$ and $\widetilde{W}$ are bounded operators, we may extend this result to all $v,u\in{\mathcal H}$.
\end{proof}

Recall that if $W$ is a multiplicative partial isometry, then so is $\widehat{W}=\Sigma W^*\Sigma$.  If $W$ is further known to be 
a manageable multiplicative partial isometry, then it can be shown that $\widehat{W}$ is also manageable. See below (see also a similar result 
in Proposition~1.4 of \cite{Wr7}): 

\begin{prop}\label{manageableWhat}
Let $W$ be a manageable multiplicative partial isometry, and let $Q$ and $\widetilde{W}$ be the associated operators given in Definition~\ref{manageable}.  
Then the operator $\widehat{W}=\Sigma W^*\Sigma$ is also a manageable multiplicative partial isometry, with the same $Q$ and 
$\widetilde{\widehat{W}}=(\Sigma\widetilde{W}^*\Sigma)^{\top\otimes\top}$.
\end{prop}

\begin{proof}
(1). From $W(Q\otimes Q)\subseteq (Q\otimes Q)W$, it is easy to see that $\widehat{W}(Q\otimes Q)\subseteq (Q\otimes Q)\widehat{W}$.

(2). Write $\widetilde{\widehat{W}}=(\Sigma\widetilde{W}^*\Sigma)^{\top\otimes\top}$.  For any $\xi,\eta\in{\mathcal H}$ and $v\in{\mathcal D}(Q^{-1})$, 
$u\in{\mathcal D}(Q)$, observe that
\begin{align}
\bigl\langle\widetilde{\widehat{W}}(\bar{\eta}\otimes Q^{-1}v),\bar{\xi}\otimes Qu\bigr\rangle
&=\bigl\langle(\Sigma\widetilde{W}^*\Sigma)^{\top\otimes\top}(\bar{\eta}\otimes Q^{-1}v),\bar{\xi}\otimes Qu\bigr\rangle 
=\bigl\langle \Sigma\widetilde{W}^*\Sigma(\xi\otimes Q^{\top}\bar{u}),\eta\otimes[Q^{-1}]^{\top}\bar{v}\bigr\rangle \notag \\
&=\bigl\langle\widetilde{W}^*(Q^{\top}\bar{u}\otimes\xi),[Q^{-1}]^{\top}\bar{v}\otimes\eta\bigr\rangle 
=\overline{\bigl\langle\widetilde{W}([Q^{-1}]^{\top}\bar{v}\otimes\eta),Q^{\top}\bar{u}\otimes\xi\bigr\rangle}  \notag \\
&=\overline{\bigl\langle W(u\otimes\eta),v\otimes\xi\bigr\rangle}=\bigl\langle W^*(v\otimes\xi),u\otimes\eta\bigr\rangle
=\bigl\langle\widehat{W}(\xi\otimes v),\eta\otimes u\bigr\rangle.
\notag
\end{align}
Here we used the result that $m^{\top}\bar{\xi}=\overline{m^*\xi}$ and that $\langle m^{\top}\bar{\eta},\bar{\xi}\rangle=\langle m\xi,\eta\rangle$, 
for $\xi,\eta\in{\mathcal H}$. Meanwhile in the fifth equality, we used the alternative characterizing equation given in Proposition~\ref{manageable_alt}. 

(3). Finally, we need to verify the two conditions 
$\widetilde{\widehat{W}}_{13}\widetilde{\widehat{W}}_{23}\bigl[\widetilde{\widehat{W}}_{23}\bigr]^*
=\widehat{W}_{12}^{\top\otimes\top}[\widehat{W}_{12}^*]^{\top\otimes\top}\widetilde{\widehat{W}}_{13}$ 
and $\widehat{W}_{23}\widehat{W}_{23}^*\widetilde{\widehat{W}}_{13}
=\widetilde{\widehat{W}}_{13}\widetilde{\widehat{W}}_{12}\bigl[\widetilde{\widehat{W}}_{12}\bigr]^*$. 
Indeed we have:
\begin{align}
\widetilde{\widehat{W}}_{13}\widetilde{\widehat{W}}_{23}\bigl[\widetilde{\widehat{W}}_{23}\bigr]^*
&=\Sigma_{13}[\widetilde{W}_{13}^*]^{\top\otimes\top}\Sigma_{13}\Sigma_{23}[\widetilde{W}_{23}^*]^{\top\otimes\top}\Sigma_{23}
\Sigma_{23}[\widetilde{W}_{23}]^{\top\otimes\top}\Sigma_{23}  \notag \\
&=\Sigma_{13}[\widetilde{W}_{13}^*]^{\top\otimes\top}[\widetilde{W}_{12}^*]^{\top\otimes\top}[\widetilde{W}_{12}]^{\top\otimes\top}\Sigma_{13} 
=\Sigma_{13}\bigl[(\widetilde{W}_{13}\widetilde{W}_{12}\widetilde{W}^*_{12})^*\bigr]^{\top\otimes\top\otimes\top}\Sigma_{13} \notag \\
&=\Sigma_{13}\bigl[(W_{23}W_{23}^*\widetilde{W}_{13})^*\bigr]^{\top\otimes\top\otimes\top}\Sigma_{13} 
=\Sigma_{13}[W^*_{23}]^{\top\otimes\top}W_{23}^{\top\otimes\top}[\widetilde{W}^*_{13}]^{\top\otimes\top}\Sigma_{13}  \notag \\
&=[W^*_{21}]^{\top\otimes\top}W_{21}^{\top\otimes\top}[\widetilde{W}^*_{31}]^{\top\otimes\top}
=\widehat{W}_{12}^{\top\otimes\top}[\widehat{W}_{12}^*]^{\top\otimes\top}\widetilde{\widehat{W}}_{13},
\notag
\end{align}
where we used the fact that $m\mapsto m^{\top}$ is a ${}^*$-anti-isomorphism, and the second condition in (3) of Definition~\ref{manageable} 
(for the fourth equality). Also we have:
\begin{align}
\widetilde{\widehat{W}}_{13}\widetilde{\widehat{W}}_{12}\bigl[\widetilde{\widehat{W}}_{12}\bigr]^*
&=\Sigma_{13}[\widetilde{W}_{13}^*]^{\top\otimes\top}\Sigma_{13}\Sigma_{12}[\widetilde{W}_{12}^*]^{\top\otimes\top}\Sigma_{12}
\Sigma_{12}\widetilde{W}_{12}^{\top\otimes\top}\Sigma_{12}  \notag \\
&=\Sigma_{13}[\widetilde{W}_{13}^*]^{\top\otimes\top}[\widetilde{W}_{23}^*]^{\top\otimes\top}\widetilde{W}_{23}^{\top\otimes\top}\Sigma_{13} 
=\Sigma_{13}\bigl[(\widetilde{W}_{13}\widetilde{W}_{23}\widetilde{W}^*_{23})^*\bigr]^{\top\otimes\top\otimes\top}\Sigma_{13}  \notag \\
&=\Sigma_{13}\bigl[(W^{\top\otimes\top}_{12}{W^*}^{\top\otimes\top}_{12}\widetilde{W}_{13})^*\bigr]^{\top\otimes\top\otimes\top}\Sigma_{13}  
=\Sigma_{13}W^*_{12}W_{12}[\widetilde{W}^*_{13}]^{\top\otimes\top}\Sigma_{13} \notag \\
&=W^*_{32}W_{32}[\widetilde{W}^*_{31}]^{\top\otimes\top}
=\widehat{W}_{23}\widehat{W}^*_{23}\widetilde{\widehat{W}}_{13},
\notag 
\end{align}
where we used the first condition in (3) of Definition~\ref{manageable} (the fourth equality).

By (1),\,(2),\,(3), we see that $\widehat{W}=\Sigma W^*\Sigma$ is manageable, 
with $\widetilde{\widehat{W}}=(\Sigma\widetilde{W}^*\Sigma)^{\top\otimes\top}$ and same $Q$.
\end{proof}

We need to wait a little further to show that $\widetilde{W}$ and $\widetilde{\widehat{W}}$ are also partial isometries. Still, we can at least show that 
they satisfy the fullness condition (Definition~\ref{fullness_condition}):

\begin{lem}\label{Wtilde_full}
Let $W$ be a manageable multiplicative partial isometry satisfying the fullness condition (Definition~\ref{fullness_condition}). 
Then the operator $\widetilde{W}$ is also full, in a modified sense that for any $\xi\in{\mathcal H}$, $\xi\ne0$, we can find $p,q,r,s\in{\mathcal H}$ 
such that $\widetilde{W}(\bar{\xi}\otimes p)\ne0$, $\widetilde{W}^*(\bar{\xi}\otimes r)\ne0$, $\widetilde{W}(\bar{q}\otimes\xi)\ne0$, $\widetilde{W}^*(\bar{s}\otimes\xi)\ne0$.

By considering instead $\widehat{W}=\Sigma W^*\Sigma$, we can see immediately that $\widetilde{\widehat{W}}$ is also full.
\end{lem}

\begin{proof}
Suppose $\xi\in{\mathcal H}$, $\xi\ne0$. As $W$ is full, we can find $u\in{\mathcal D}(Q)$ such that $W^*(\xi\otimes u)\ne0$. So we can further find 
$\eta\in{\mathcal H}$, $v\in{\mathcal D}(Q^{-1})$ such that $\bigl\langle \eta\otimes v,W^*(\xi\otimes u)\bigr\rangle\ne0$. [Note that ${\mathcal D}(Q)$ and 
${\mathcal D}(Q^{-1})$ are dense in ${\mathcal H}$.]
It follows that 
$$
\bigl\langle \widetilde{W}(\bar{\xi}\otimes Q^{-1}v),\bar{\eta}\otimes Qu\bigr\rangle=\bigl\langle W(\eta\otimes v),\xi\otimes u\bigr\rangle
=\bigl\langle \eta\otimes v,W^*(\xi\otimes u)\bigr\rangle\ne0.
$$
This shows that $\widetilde{W}(\bar{\xi}\otimes Q^{-1}v)\ne0$. Similar proof can be given for the other three conditions.
\end{proof}

In the lemmas below, we obtain some results that relate the operators $W$, $\widetilde{W}$, $Q$, and the transpose map ${}^{\top}$.  The linear 
functional $\omega_{a,b}\in{\mathcal B}({\mathcal H})_*$, $a,b\in{\mathcal H}$, is as defined earlier.

\begin{lem}\label{lem_QT}
Let $\xi,u\in{\mathcal D}(Q)$ and $\eta,v\in{\mathcal D}(Q^{-1})$. Then we have:

\begin{enumerate}
\item $(\operatorname{id}\otimes\omega_{Q^{-1}v,Qu})(\widetilde{W})=(\operatorname{id}\otimes\omega_{v,u})(W)^{\top}$.
\item$(\omega_{\bar{\xi},\bar{\eta}}\otimes\operatorname{id})(\widetilde{W})
=(\omega_{Q^{\top}\bar{\xi},[Q^{-1}]^{\top}\bar{\eta}}\otimes\operatorname{id})(W^{\top\otimes\top})^{\top}$.
\end{enumerate}
\end{lem}

\begin{proof}
(1). Let $\xi,\eta\in{\mathcal H}$ be arbitrary.  We have:
\begin{align}
\bigl\langle (\operatorname{id}\otimes\omega_{v,u})(W)^{\top}\bar{\eta},\bar{\xi}\bigr\rangle
&=\bigl\langle (\operatorname{id}\otimes\omega_{v,u})(W)\xi,\eta\bigr\rangle 
=\bigl\langle W(\xi\otimes v),\eta\otimes u\bigr\rangle  \notag \\
&=\bigl\langle\widetilde{W}(\bar{\eta}\otimes Q^{-1}v),\bar{\xi}\otimes Qu\bigr\rangle
=\bigl\langle (\operatorname{id}\otimes\omega_{Q^{-1}v,Qu})(\widetilde{W})\bar{\eta},\bar{\xi}\bigr\rangle.
\notag
\end{align}
This proves the result. We used Definition~\ref{manageable}\,(2), the characterizing equation for $\widetilde{W}$.

(2). Let $u,v\in{\mathcal H}$ be arbitrary. Then, this time using the characterization of $\widetilde{W}$ given in Proposition~\ref{manageable_alt}, we have:
\begin{align}
\bigl\langle(\omega_{\bar{\xi},\bar{\eta}}\otimes\operatorname{id})(\widetilde{W})u,v\bigr\rangle
&=\bigl\langle\widetilde{W}(\bar{\xi}\otimes u),\bar{\eta}\otimes v\bigr\rangle=\bigl\langle W(Q^{-1}\eta\otimes u),Q\xi\otimes v\bigr\rangle \notag \\
&=\bigl\langle W^{\top\otimes\top}(Q^{\top}\bar{\xi}\otimes\bar{v}),[Q^{-1}]^{\top}\bar{\eta}\otimes\bar{u}\bigr\rangle
=\bigl\langle(\omega_{Q^{\top}\bar{\xi},[Q^{-1}]^{\top}\bar{\eta}}\otimes\operatorname{id})(W^{\top\otimes\top})\bar{v},\bar{u}\bigr\rangle
\notag \\
&=\bigl\langle(\omega_{Q^{\top}\bar{\xi},[Q^{-1}]^{\top}\bar{\eta}}\otimes\operatorname{id})(W^{\top\otimes\top})^{\top}u,v\bigr\rangle.
\notag
\end{align}
\end{proof}

\begin{lem}\label{lem_QTT}
Let $W,Q,\widetilde{W}$ be as above. Then for any $u\in{\mathcal D}(Q)$, $v\in{\mathcal D}(Q^{-1})$, we have: 
\begin{enumerate}
\item 
$(\operatorname{id}\otimes\omega_{v,u}\otimes\operatorname{id})(W_{12}W_{23}^*)
=(\operatorname{id}\otimes\omega_{Q^{\top}\bar{u},[Q^{-1}]^{\top}\bar{v}}\otimes\operatorname{id})(\widetilde{W}_{23}^*\widetilde{W}_{12}^{\top\otimes\top})$.
\item 
$(\operatorname{id}\otimes\omega_{u,v}\otimes\operatorname{id})(W_{23}W_{12})
=(Q^{-2}\otimes1)(\operatorname{id}\otimes\omega_{[Q^{-1}]^{\top}\bar{v},Q^{\top}\bar{u}}\otimes\operatorname{id})(\widetilde{W}_{12}^{\top\otimes\top}\widetilde{W}_{23})
(Q^2\otimes1)$.
\end{enumerate}
\end{lem}

\begin{proof}
(1). Let $\xi,r,u\in{\mathcal D}(Q)$ and $\eta,s,v\in{\mathcal D}(Q^{-1})$ be arbitrary. Also let $(e_j)_{j\in J}$ be an orthonormal basis for ${\mathcal H}$. 
By basic Linear Algebra we have:
\begin{equation}\label{(QTT_eq1)}
\bigl\langle W_{12}W_{23}^*(r\otimes v\otimes\xi),s\otimes u\otimes\eta\bigr\rangle
=\sum_{j\in J}\bigl\langle W(r\otimes e_j),s\otimes u\bigr\rangle\bigl\langle W^*(v\otimes \xi),e_j\otimes\eta\bigr\rangle.
\end{equation}
By Proposition~\ref{manageable_alt}, we have:
$\bigl\langle W(r\otimes e_j),s\otimes u\bigr\rangle=\bigl\langle \widetilde{W}([Q^{-1}]^{\top}\bar{s}\otimes e_j),Q^{\top}\bar{r}\otimes u\bigr\rangle$. 
Meanwhile by Definition~\ref{manageable}, we have: 
$$
\bigl\langle W^*(v\otimes \xi),e_j\otimes\eta\bigr\rangle=\bigl\langle v\otimes \xi,W(e_j\otimes\eta)\bigr\rangle
=\bigl\langle \bar{e}_j\otimes Q\xi,\widetilde{W}(\bar{v}\otimes Q^{-1}\eta)\bigr\rangle 
=\bigl\langle [\widetilde{W}^*]^{\top\otimes\top}(v\otimes[Q^{-1}]^{\top}\bar{\eta}),e_j\otimes Q^{\top}\bar{\xi}\bigr\rangle.
$$
Putting these together, Equation~\eqref{(QTT_eq1)} becomes
\begin{align}
\bigl\langle W_{12}W_{23}^*(r\otimes v\otimes\xi),s\otimes u\otimes\eta\bigr\rangle
&=\sum_{j\in J}\bigl\langle \widetilde{W}([Q^{-1}]^{\top}\bar{s}\otimes e_j),Q^{\top}\bar{r}\otimes u\bigr\rangle
\bigl\langle [\widetilde{W}^*]^{\top\otimes\top}(v\otimes[Q^{-1}]^{\top}\bar{\eta}),e_j\otimes Q^{\top}\bar{\xi}\bigr\rangle  \notag \\
&=\bigl\langle \widetilde{W}_{12}[\widetilde{W}_{23}^*]^{\top\otimes\top}([Q^{-1}]^{\top}\bar{s}\otimes v\otimes[Q^{-1}]^{\top}\bar{\eta}),
Q^{\top}\bar{r}\otimes u\otimes Q^{\top}\bar{\xi}\bigr\rangle
\notag \\
&=\bigl\langle \widetilde{W}_{12}[\widetilde{W}_{23}^*]^{\top\otimes\top}(\bar{s}\otimes Q^{-1}v\otimes\bar{\eta}),
\bar{r}\otimes Qu\otimes\bar{\xi}\bigr\rangle
\notag \\
&=\bigl\langle \widetilde{W}_{23}^*\widetilde{W}_{12}^{\top\otimes\top}(r\otimes Q^{\top}\bar{u}\otimes\xi),s\otimes [Q^{-1}]^{\top}\bar{v}\otimes\eta\bigr\rangle.
\label{(QTT_eq2)}
\end{align}
Note that for the third equality, we used the result that $\widetilde{W}(Q^{\top}\otimes Q^{-1})\subseteq(Q^{\top}\otimes Q^{-1})\widetilde{W}$, from which 
we have
$\widetilde{W}_{12}[\widetilde{W}_{23}^*]^{\top\otimes\top}\bigl([Q^{-1}]^{\top}\otimes Q\otimes[Q^{-1}]^{\top}\bigr)
\subseteq\bigl([Q^{-1}]^{\top}\otimes Q\otimes[Q^{-1}]^{\top}\bigr)\widetilde{W}_{12}[\widetilde{W}_{23}^*]^{\top\otimes\top}$.

As the vectors $\xi,r,\eta,s$ are arbitrary, Equation~\eqref{(QTT_eq2)} shows that 
$$
(\operatorname{id}\otimes\omega_{v,u}\otimes\operatorname{id})(W_{12}W_{23}^*)
=(\operatorname{id}\otimes\omega_{Q^{\top}\bar{u},[Q^{-1}]^{\top}\bar{v}}\otimes\operatorname{id})(\widetilde{W}_{23}^*\widetilde{W}_{12}^{\top\otimes\top}).
$$

(2). As in (1), let $\xi,r,u\in{\mathcal D}(Q)$ and $\eta,s,v\in{\mathcal D}(Q^{-1})$ be arbitrary, and let $(e_j)_{j\in J}$ be an orthonormal basis for ${\mathcal H}$. 
Then similarly as above, we obtain:
$$
\bigl\langle W_{12}^*W_{23}^*(r\otimes v\otimes\xi),s\otimes u\otimes\eta\bigr\rangle
=\sum_{j\in J}\bigl\langle \widetilde{W}^*(Q^{\top}\bar{s}\otimes e_j),[Q^{-1}]^{\top}\bar{r}\otimes u\bigr\rangle
\bigl\langle [\widetilde{W}^*]^{\top\otimes\top}(v\otimes[Q^{-1}]^{\top}\bar{\eta}),e_j\otimes Q^{\top}\bar{\xi}\bigr\rangle.
$$
It follows that
\begin{align}
\bigl\langle W_{12}^*W_{23}^*(r\otimes v\otimes\xi),s\otimes u\otimes\eta\bigr\rangle
&=\bigl\langle \widetilde{W}_{12}^*[\widetilde{W}_{23}^*]^{\top\otimes\top}(Q^{\top}\bar{s}\otimes v\otimes[Q^{-1}]^{\top}\bar{\eta}),
[Q^{-1}]^{\top}\bar{r}\otimes u\otimes Q^{\top}\bar{\xi}\bigr\rangle
\notag \\
&=\bigl\langle \widetilde{W}_{12}^*[\widetilde{W}_{23}^*]^{\top\otimes\top}([Q^2]^{\top}\bar{s}\otimes Q^{-1}v\otimes\bar{\eta}),
[Q^{-2}]^{\top}\bar{r}\otimes Qu\otimes\bar{\xi}\bigr\rangle
\notag \\
&=\bigl\langle \widetilde{W}_{23}^*[\widetilde{W}_{12}^*]^{\top\otimes\top}(Q^{-2}r\otimes Q^{\top}\bar{u}\otimes\xi),Q^2s\otimes [Q^{-1}]^{\top}\bar{v}\otimes\eta\bigr\rangle.
\label{(QTT_eq4)}
\end{align}
For the second equality, we again used $\widetilde{W}(Q^{\top}\otimes Q^{-1})\subseteq(Q^{\top}\otimes Q^{-1})\widetilde{W}$. We thus showed that 
$$
(\operatorname{id}\otimes\omega_{v,u}\otimes\operatorname{id})(W_{12}^*W_{23}^*)
=(Q^2\otimes1)(\operatorname{id}\otimes\omega_{Q^{\top}\bar{u},[Q^{-1}]^{\top}\bar{v}}\otimes\operatorname{id})(\widetilde{W}_{23}^*[\widetilde{W}_{12}^*]^{\top\otimes\top})
(Q^{-2}\otimes1).
$$
Next, take the adjoint of this expression. Then we have:
$$
(\operatorname{id}\otimes\omega_{u,v}\otimes\operatorname{id})(W_{23}W_{12})
=(Q^{-2}\otimes1)(\operatorname{id}\otimes\omega_{[Q^{-1}]^{\top}\bar{v},Q^{\top}\bar{u}}\otimes\operatorname{id})(\widetilde{W}_{12}^{\top\otimes\top}\widetilde{W}_{23})
(Q^2\otimes1)
$$
\end{proof}

The next proposition provides some key observations:

\begin{prop}\label{hashcomposable}
Let $W$ be a manageable multiplicative partial isometry, and let $Q$ and $\widetilde{W}$ be the associated operators 
given in Definition~\ref{manageable}.  Then we have:
\begin{enumerate}
 \item $W_{12}^{\top\otimes\top}\widetilde{W}_{23}[W_{12}^*]^{\top\otimes\top}=\widetilde{W}_{13}\widetilde{W}_{23}$
 \smallskip
 \item $[W_{12}^*]^{\top\otimes\top}W_{12}^{\top\otimes\top}\widetilde{W}_{23}=\widetilde{W}_{23}[W_{12}^*]^{\top\otimes\top}W_{12}^{\top\otimes\top}$
 \smallskip
 \item $\widetilde{W}_{23}[W_{12}^*]^{\top\otimes\top}\widetilde{W}_{23}^*=[W_{12}^*]^{\top\otimes\top}\widetilde{W}_{13}$
 \smallskip
 \item $\widetilde{W}_{12}W_{23}^*\widetilde{W}_{12}^*=W_{23}^*\widetilde{W}_{13}$
 \smallskip
\end{enumerate}
\end{prop}

\begin{proof}
(1). Let $\xi,\eta,r,s\in{\mathcal H}$, and $u\in{\mathcal D}(Q)$, $v\in{\mathcal D}(Q^{-1})$.  Then:
\begin{equation}\label{(eqnA)}
\bigl\langle W_{12}^{\top\otimes\top}\widetilde{W}_{23}[W_{12}^*]^{\top\otimes\top}(\bar{\eta}\otimes\bar{r}\otimes Q^{-1}v),
\bar{\xi}\otimes\bar{s}\otimes Qu\bigr\rangle 
=\bigl\langle W^{\top\otimes\top}[1\otimes(\operatorname{id}\otimes\omega_{Q^{-1}v,Qu})(\widetilde{W})][W^*]^{\top\otimes\top}
(\bar{\eta}\otimes\bar{r}),\bar{\xi}\otimes\bar{s}\bigr\rangle.
\end{equation}

By Lemma~\ref{lem_QT}\,(1), we have:
$$
W^{\top\otimes\top}[1\otimes(\operatorname{id}\otimes\omega_{Q^{-1}v,Qu})(\widetilde{W})][W^*]^{\top\otimes\top}
=W^{\top\otimes\top}[1\otimes(\operatorname{id}\otimes\omega_{v,u})(W)^{\top}][W^*]^{\top\otimes\top},
$$
which is equal to $\bigl(W^*[1\otimes(\operatorname{id}\otimes\omega_{v,u})(W)]W\bigr)^{\top\otimes\top}$ because 
${}^{\top\otimes\top}$ is an anti-homomorphism.  Moreover, 
$$
W^*[1\otimes(\operatorname{id}\otimes\omega_{v,u})(W)]W=
(\operatorname{id}\otimes\operatorname{id}\otimes\omega_{v,u})(W^*_{12}W_{23}W_{12})
=(\operatorname{id}\otimes\operatorname{id}\otimes\omega_{v,u})(W_{13}W_{23}),
$$
by Equation~\eqref{(mpi2)}.  We thus see that the right hand side of Equation~\eqref{(eqnA)} is equal to 
$$
(RHS)=\bigl\langle (\operatorname{id}\otimes\operatorname{id}\otimes\omega_{v,u})(W_{13}W_{23})^{\top\otimes\top}
(\bar{\eta}\otimes\bar{r}),\bar{\xi}\otimes\bar{s}\bigr\rangle 
=\bigl\langle W_{13}W_{23}(\xi\otimes s\otimes v),\eta\otimes r\otimes u\bigr\rangle.
$$

Use the characterization in Proposition~\ref{manageable_alt}, and write $v=QQ^{-1}v$ and $u=Q^{-1}Qu$. Then 
\begin{align}
(RHS)&=\bigl\langle W_{13}\widetilde{W}_{23}(\xi\otimes [Q^{-1}]^{\top}\bar{r}\otimes QQ^{-1}v),\eta\otimes Q^{\top}\bar{s}\otimes Q^{-1}Qu\bigr\rangle 
\notag \\
&=\bigl\langle W_{13}(\xi\otimes[([Q^{-1}]^{\top}\otimes Q)\widetilde{W}(\bar{r}\otimes Q^{-1}v)]),\eta\otimes Q^{\top}\bar{s}\otimes Q^{-1}Qu\bigr\rangle 
\notag \\
&=\bigl\langle\widetilde{W}_{13}(\bar{\eta}\otimes[([Q^{-1}]^{\top}\otimes\operatorname{id})\widetilde{W}(\bar{r}\otimes Q^{-1}v)]),
\bar{\xi}\otimes Q^{\top}\bar{s}\otimes Qu\bigr\rangle  \notag \\
&=\bigl\langle\widetilde{W}_{13}\widetilde{W}_{23}(\bar{\eta}\otimes\bar{r}\otimes Q^{-1}v),\bar{\xi}\otimes\bar{s}\otimes Qu\bigr\rangle.
\notag
\end{align}
In the second equality, we used the fact that $\widetilde{W}([Q^{-1}]^{\top}\otimes Q)\subseteq ([Q^{-1}]^{\top}\otimes Q)\widetilde{W}$ 
from \eqref{(WtildeQQinclusion)}, which is all right because $\bar{r}\otimes Q^{-1}v\in{\mathcal D}({Q^{-1}}^{\top}\otimes Q)$. The third equality uses 
Definition~\ref{manageable}.

Putting this result back into Equation~\eqref{(eqnA)} above, we obtain:
$$
\bigl\langle W_{12}^{\top\otimes\top}\widetilde{W}_{23}[W_{12}^*]^{\top\otimes\top}(\bar{\eta}\otimes\bar{r}\otimes Q^{-1}v),
\bar{\xi}\otimes\bar{s}\otimes Qu\bigr\rangle
=\bigl\langle\widetilde{W}_{13}\widetilde{W}_{23}(\bar{\eta}\otimes\bar{r}\otimes Q^{-1}v),\bar{\xi}\otimes\bar{s}\otimes Qu\bigr\rangle.
$$
As $\xi,\eta,r,s,u,v$ are arbitrary, this proves that 
$W_{12}^{\top\otimes\top}\widetilde{W}_{23}[W_{12}^*]^{\top\otimes\top}=\widetilde{W}_{13}\widetilde{W}_{23}$.

\bigskip

(2).  Let $\xi,\eta,r,s\in{\mathcal H}$, and $u\in{\mathcal D}(Q)$, $v\in{\mathcal D}(Q^{-1})$.  Then:
$$
\bigl\langle [W_{12}^*]^{\top\otimes\top}W_{12}^{\top\otimes\top}\widetilde{W}_{23}(\bar{\eta}\otimes\bar{r}\otimes Q^{-1}v),
\bar{\xi}\otimes\bar{s}\otimes Qu\bigr\rangle  
=\bigl\langle [W^*]^{\top\otimes\top}W^{\top\otimes\top}[1\otimes(\operatorname{id}\otimes\omega_{Q^{-1}v,Qu})(\widetilde{W})]
(\bar{\eta}\otimes\bar{r}),\bar{\xi}\otimes\bar{s}\bigr\rangle.
$$

By using Lemma~\ref{lem_QT}\,(1) and the fact that ${}^{\top\otimes\top}$ is an anti-homomorphism, we have: 
\begin{align}
&[W^*]^{\top\otimes\top}W^{\top\otimes\top}[1\otimes(\operatorname{id}\otimes\omega_{Q^{-1}v,Qu})(\widetilde{W})]
= [W^*]^{\top\otimes\top}W^{\top\otimes\top}[1\otimes(\operatorname{id}\otimes\omega_{v,u})(W)]^{\top\otimes\top}
\notag \\
&=(\operatorname{id}\otimes\operatorname{id}\otimes\omega_{v,u})(W_{23}W_{12}W_{12}^*)^{\top\otimes\top}
=(\operatorname{id}\otimes\operatorname{id}\otimes\omega_{v,u})(W_{12}W_{12}^*W_{23})^{\top\otimes\top}
\notag \\
&=[1\otimes(\operatorname{id}\otimes\omega_{Q^{-1}v,Qu})(\widetilde{W})][W^*]^{\top\otimes\top}W^{\top\otimes\top},
\notag 
\end{align}
where we also used Equation~\eqref{(mpi4)}.  From this it follows that 
$$
\bigl\langle [W_{12}^*]^{\top\otimes\top}W_{12}^{\top\otimes\top}\widetilde{W}_{23}(\bar{\eta}\otimes\bar{r}\otimes Q^{-1}v),
\bar{\xi}\otimes\bar{s}\otimes Qu\bigr\rangle  
=\bigl\langle \widetilde{W}_{23}[W_{12}^*]^{\top\otimes\top}W_{12}^{\top\otimes\top}(\bar{\eta}\otimes\bar{r}\otimes Q^{-1}v),
\bar{\xi}\otimes\bar{s}\otimes Qu\bigr\rangle.
$$
As $\xi,\eta,r,s,u,v$ are arbitrary, this proves that $[W_{12}^*]^{\top\otimes\top}W_{12}^{\top\otimes\top}\widetilde{W}_{23}
=\widetilde{W}_{23}[W_{12}^*]^{\top\otimes\top}W_{12}^{\top\otimes\top}$.

\bigskip

(3). From (1), we have: $W_{12}^{\top\otimes\top}\widetilde{W}_{23}[W_{12}^*]^{\top\otimes\top}=\widetilde{W}_{13}\widetilde{W}_{23}$. 
Multiply $[W_{12}^*]^{\top\otimes\top}$ from the left and multiply $\widetilde{W}_{23}^*$ from the right.  then it becomes:
$$
[W_{12}^*]^{\top\otimes\top}W_{12}^{\top\otimes\top}\widetilde{W}_{23}[W_{12}^*]^{\top\otimes\top}\widetilde{W}_{23}^*
=[W_{12}^*]^{\top\otimes\top}\widetilde{W}_{13}\widetilde{W}_{23}\widetilde{W}_{23}^*.
$$
In the (LHS), apply (2), while in the (RHS), apply the condition Definition~\ref{manageable}\,(3).  Then it becomes:
\begin{equation}\label{(eqnB)}
\widetilde{W}_{23}[W_{12}^*]^{\top\otimes\top}W_{12}^{\top\otimes\top}[W_{12}^*]^{\top\otimes\top}\widetilde{W}_{23}^*
=[W_{12}^*]^{\top\otimes\top}W^{\top\otimes\top}_{12}[W^*]^{\top\otimes\top}_{12}\widetilde{W}_{13}.
\end{equation}
Recall that $W$ is a partial isometry, so $WW^*W=W$.  Apply here the involution and the transpose map, which are both 
anti-homomorphisms.  We have: $[W^*]^{\top\otimes\top}W^{\top\otimes\top}[W^*]^{\top\otimes\top}=[W^*]^{\top\otimes\top}$. 
From this observation, it follows from Equation~\eqref{(eqnB)} that
$\widetilde{W}_{23}[W_{12}^*]^{\top\otimes\top}\widetilde{W}_{23}^*=[W_{12}^*]^{\top\otimes\top}\widetilde{W}_{13}$.

\bigskip

(4). In (3), by Proposition~\ref{manageableWhat}, we may replace $W$ with $\widehat{W}=\Sigma W^*\Sigma$ and $\widetilde{W}$ with $\widetilde{\widehat{W}}
=(\Sigma\widetilde{W}^*\Sigma)^{\top\otimes\top}$. Then $\widetilde{W}_{23}[W_{12}^*]^{\top\otimes\top}\widetilde{W}_{23}^*=[W_{12}^*]^{\top\otimes\top}\widetilde{W}_{13}$ 
becomes
$$
[\widetilde{W}^*_{32}]^{\top\otimes\top}W_{21}^{\top\otimes\top}[\widetilde{W}_{32}]^{\top\otimes\top}=W_{21}^{\top\otimes\top}[\widetilde{W}^*_{31}]^{\top\otimes\top}.
$$
Apply here $\bigl[(\,\cdot\,)^*\bigr]^{\top\otimes\top}$, obtaining $\widetilde{W}_{32}W^*_{21}\widetilde{W}^*_{32}=W^*_{21}\widetilde{W}_{31}$. 
Finally, switch the first and the third legs, to obtain: $\widetilde{W}_{12}W^*_{23}\widetilde{W}^*_{12}=W^*_{23}\widetilde{W}_{13}$.
\end{proof}

\begin{rem}
In section 2 of \cite{Wr7}, Woronowicz showed that for a manageable multiplicative unitary $W$, we have: 
$V_{12}^{\top\otimes\top}\widetilde{W}_{23}[V_{12}^*]^{\top\otimes\top}\widetilde{W}_{23}^*=\widetilde{V}_{13}$, 
where $V$ is a unitary operator ``adapted'' to $W$, meaning that $W_{23}V_{12}=V_{12}V_{13}W_{23}$.  
As $W$ itself is adapted to $W$ (by the pentagon equation), as a special case it holds that 
$W_{12}^{\top\otimes\top}\widetilde{W}_{23}[W_{12}^*]^{\top\otimes\top}\widetilde{W}_{23}^*=\widetilde{W}_{13}$. 
This was a key observation by Woronowicz that enabled him to prove several other results that followed in \cite{Wr7}.  
In an unpublished manuscript of his, Woronowicz further develops this point, by introducing the notion of the {\em $\#$-composability\/} 
(The author is indebted to him for showing his manuscript, as well as for his valuable comments on this topic.).  In particular, one would say 
$\bigl([V^*]^{\top\otimes\top},\widetilde{W}\bigr)$ is $\#$-composable, written $[V^*]^{\top\otimes\top}\#\widetilde{W}=\widetilde{V}$.

Woronowicz's observation was that $[W^*]^{\top\otimes\top}\#\widetilde{W}=\widetilde{W}$, namely 
$W_{12}^{\top\otimes\top}\widetilde{W}_{23}[W_{12}^*]^{\top\otimes\top}\widetilde{W}_{23}^*=\widetilde{W}_{13}$. 
Unlike the case of a unitary $W$, however, this statement as written is actually no longer true in our case. Nonetheless, the properties 
(1) and (3) of the Proposition~\ref{hashcomposable} above can be seen as variations of this. Note also that (1) and (3) are not necessarily 
equivalent, so we needed separate proofs. Similar comment for (4) of the proposition, which can be seen as a version of Woronowicz's 
observation that $\widetilde{W}^*\# W=\widetilde{W}^*$. 
\end{rem}

Next proposition is similar in nature as the results in the previous Proposition~\ref{hashcomposable}, but due to the length of the proof and its unique role later, 
we made it a separate proposition:

\begin{prop}\label{hashcomposable_s}
Let $W$ be a manageable multiplicative partial isometry, and let $Q$ and $\widetilde{W}$ be the associated operators as before.  Then we have:
$$
[\widetilde{W}_{23}^*]^{\top\otimes\top}\widetilde{W}_{12}\widetilde{W}_{23}^{\top\otimes\top}=[W_{13}^*]^{\top\otimes\top}\widetilde{W}_{12}.
$$
\end{prop}

\begin{proof}
From (1) of Proposition~\ref{hashcomposable}, we know $W_{12}^{\top\otimes\top}\widetilde{W}_{23}[W_{12}^*]^{\top\otimes\top}=\widetilde{W}_{13}\widetilde{W}_{23}$. 
Multiplying to both sides $[W_{12}^*]^{\top\otimes\top}(\,\cdots\,)\widetilde{W}_{23}^*$, we have:
$$
[W_{12}^*]^{\top\otimes\top}W_{12}^{\top\otimes\top}\widetilde{W}_{23}[W_{12}^*]^{\top\otimes\top}\widetilde{W}_{23}^*
=[W_{12}^*]^{\top\otimes\top}\widetilde{W}_{13}\widetilde{W}_{23}\widetilde{W}_{23}^*.
$$
Apply here (2) of Proposition~\ref{hashcomposable}. Then the equation becomes:
$$
\widetilde{W}_{23}[W_{12}^*]^{\top\otimes\top}W_{12}^{\top\otimes\top}[W_{12}^*]^{\top\otimes\top}\widetilde{W}_{23}^*
=[W_{12}^*]^{\top\otimes\top}\widetilde{W}_{13}\widetilde{W}_{23}\widetilde{W}_{23}^*.
$$
Using the manageability condition (3) of Definition~\ref{manageable}, as well as the fact that $[W^*]^{\top\otimes\top}$ is a partial isometry (because $W^*$ is), this 
is equivalent to the following:
\begin{equation}\label{(hashcomposable_s_eq1)}
\widetilde{W}_{23}[W_{12}^*]^{\top\otimes\top}\widetilde{W}_{23}^*
=[W_{12}^*]^{\top\otimes\top}\widetilde{W}_{13}.
\end{equation}

Let $\xi,u\in{\mathcal D}(Q)$, $\eta,v\in{\mathcal D}(Q^{-1})$, $r,s\in{\mathcal H}$ be arbitrary, and compute:
\begin{align}
&\bigl\langle [\widetilde{W}_{23}^*]^{\top\otimes\top}\widetilde{W}_{12}\widetilde{W}_{23}^{\top\otimes\top}(\bar{\xi}\otimes r\otimes\bar{u}),\bar{\eta}\otimes s\otimes\bar{v}\bigr\rangle
=\bigl\langle [\widetilde{W}^*]^{\top\otimes\top}[(\omega_{\bar{\xi},\bar{\eta}}\otimes\operatorname{id})(\widetilde{W})\otimes1]\widetilde{W}^{\top\otimes\top}(r\otimes\bar{u}),
s\otimes\bar{v}\bigr\rangle \notag \\
&\qquad\qquad\qquad\qquad=\bigl\langle [\widetilde{W}^*]^{\top\otimes\top}[(\omega_{Q^{\top}\bar{\xi},[Q^{-1}]^{\top}\bar{\eta}}\otimes\operatorname{id})
(W^{\top\otimes\top})^{\top}\otimes1]\widetilde{W}^{\top\otimes\top}(r\otimes\bar{u}),s\otimes\bar{v}\bigr\rangle,
\label{(hashcomposable_s_eq2)}
\end{align}
by using Lemma~\ref{lem_QT}\,(2). As $(\ )^{\top\otimes\top}$ is anti-multiplicative, we have:
\begin{align}
&[\widetilde{W}^*]^{\top\otimes\top}[(\omega_{Q^{\top}\bar{\xi},[Q^{-1}]^{\top}\bar{\eta}}\otimes\operatorname{id})(W^{\top\otimes\top})^{\top}\otimes1]\widetilde{W}^{\top\otimes\top}
=\bigl(\widetilde{W}[(\omega_{Q^{\top}\bar{\xi},[Q^{-1}]^{\top}\bar{\eta}}\otimes\operatorname{id})(W^{\top\otimes\top})\otimes1]\widetilde{W}^*\bigr)^{\top\otimes\top}
\notag \\
&=\bigl[(\omega_{Q^{\top}\bar{\xi},[Q^{-1}]^{\top}\bar{\eta}}\otimes\operatorname{id}\otimes\operatorname{id})(\widetilde{W}_{23}W_{12}^{\top\otimes\top}
\widetilde{W}_{23}^*)\bigr]^{\top\otimes\top}
=\bigl[(\omega_{Q^{\top}\bar{\xi},[Q^{-1}]^{\top}\bar{\eta}}\otimes\operatorname{id}\otimes\operatorname{id})(\widetilde{W}_{13}^*W_{12}^{\top\otimes\top})\bigr]^{\top\otimes\top},
\label{(hashcomposable_s_eq3)}
\end{align}
where we also used property (3) in Proposition~\ref{hashcomposable}.

Using the result of Equation~\eqref{(hashcomposable_s_eq3)} into Equation~\eqref{(hashcomposable_s_eq2)}, we have: 
\begin{align}
&\bigl\langle [\widetilde{W}_{23}^*]^{\top\otimes\top}\widetilde{W}_{12}\widetilde{W}_{23}^{\top\otimes\top}(\bar{\xi}\otimes r\otimes\bar{u}),\bar{\eta}\otimes s\otimes\bar{v}\bigr\rangle
=\bigl\langle \widetilde{W}_{13}^*W_{12}^{\top\otimes\top}(Q^{\top}\bar{\xi}\otimes\bar{s}\otimes v),[Q^{-1}]^{\top}\bar{\eta}\otimes\bar{r}\otimes u\bigr\rangle
\notag \\
&\ =\bigl\langle [(\operatorname{id}\otimes\omega_{v,u})(\widetilde{W}^*)\otimes1]W^{\top\otimes\top}(Q^{\top}\bar{\xi}\otimes\bar{s}),[Q^{-1}]^{\top}\bar{\eta}\otimes\bar{r}\bigr\rangle
\notag \\
&\ =\bigl\langle [m\otimes1]W^{\top\otimes\top}(Q^{\top}\bar{\xi}\otimes\bar{s}),([Q^{-1}]^{\top}\bar{\eta}\otimes\bar{r}\bigr\rangle
=\bigl\langle W^{\top\otimes\top}(Q^{\top}\bar{\xi}\otimes\bar{s}),m^*[Q^{-1}]^{\top}\bar{\eta}\otimes\bar{r}\bigr\rangle,
\label{(hashcomposable_s_eq4)}
\end{align}
where we wrote $m=(\operatorname{id}\otimes\omega_{v,u})(\widetilde{W}^*)$, for convenience. Continuing the computation, the (RHS) of Equation~\eqref{(hashcomposable_s_eq4)} 
becomes 
$$
\dots=\bigl\langle W(m^{\top}Q^{-1}\eta\otimes r),Q\xi\otimes s\bigr\rangle
=\bigl\langle\widetilde{W}(\bar{\xi}\otimes r),(Q^{\top}m^*[Q^{-1}]^{\top})\bar{\eta}\otimes s\bigr\rangle
=\bigl\langle([Q^{-1}]^{\top}mQ^{\top})\widetilde{W}(\bar{\xi}\otimes r),\bar{\eta}\otimes s\bigr\rangle.
$$
Note here that $[Q^{-1}]^{\top}mQ^{\top}=[Q^{-1}]^{\top}(\operatorname{id}\otimes\omega_{v,u})(\widetilde{W}^*)Q^{\top}$. But then, by using a modified version of 
Lemma~\ref{lem_QT}\,(1) using the fact $\widetilde{W}\bigl([Q^{-1}]^{\top}\otimes Q\bigr)\subseteq\bigl([Q^{-1}]^{\top}\otimes Q\bigr)\widetilde{W}$, this becomes
$$
[Q^{-1}]^{\top}mQ^{\top}=[Q^{-1}]^{\top}(\operatorname{id}\otimes\omega_{v,u})(\widetilde{W}^*)Q^{\top}
=(\operatorname{id}\otimes\omega_{\bar{u},\bar{v}})\bigl([W^*]^{\top\otimes\top}\bigr).
$$

By putting all these observations together, we now have:
\begin{align}
&\bigl\langle [\widetilde{W}_{23}^*]^{\top\otimes\top}\widetilde{W}_{12}\widetilde{W}_{23}^{\top\otimes\top}(\bar{\xi}\otimes r\otimes\bar{u}),\bar{\eta}\otimes s\otimes\bar{v}\bigr\rangle
 \notag \\
& \ =\dots=\bigl\langle(\operatorname{id}\otimes\omega_{\bar{u},\bar{v}})([W^*]^{\top\otimes\top})\widetilde{W}(\bar{\xi}\otimes r),\bar{\eta}\otimes s\bigr\rangle
=\bigl\langle [W_{13}^*]^{\top\otimes\top}\widetilde{W}_{12}(\bar{\xi}\otimes r\otimes\bar{u}),\bar{\eta}\otimes s\otimes\bar{v}\bigr\rangle.
\notag
\end{align}
As the vectors $\xi,u,\eta,v,r,s$ arbitrary, this proves the result: $ [\widetilde{W}_{23}^*]^{\top\otimes\top}\widetilde{W}_{12}\widetilde{W}_{23}^{\top\otimes\top}
= [W_{13}^*]^{\top\otimes\top}\widetilde{W}_{12}$.
\end{proof}

Our next aim is to prove that $\widetilde{W}$ is itself a partial isometry. While this seemed plausible, the proof could not be given until now. (When $W$ is unitary, 
it is straightforward to recognize that $\widetilde{W}$ is also unitary.)  As a useful tool for that purpose, let us define a linear map $\widehat{R}$ from 
$\widehat{\mathcal A}\cup\widehat{\mathcal A}^*\,\bigl(\subseteq{\mathcal B}({\mathcal H})\bigr)$ into itself, as given in the following proposition:

\begin{prop}\label{Rhat_pre}
\begin{enumerate}
\item For any $a\in\widehat{\mathcal A}\cup\widehat{\mathcal A}^*$, there exists $a^{\widehat{R}}\in\widehat{\mathcal A}\cup\widehat{\mathcal A}^*$ such that
$$
Y\bigl(1\otimes a^{\widehat{R}}\bigr)Y^*=Z\bigl(a^{\top}\otimes1\bigr)Z^*,
$$
where $a^{\top}$ is the transpose by regarding $a\in{\mathcal B}({\mathcal H})$, and $Y,Z\in{\mathcal B}(\overline{\mathcal H}\otimes{\mathcal H})$ are the operators 
defined by $Y=\Sigma \widetilde{W}^{\top\otimes\top}[\widetilde{W}^*]^{\top\otimes\top}\Sigma$ and $Z=\Sigma\widetilde{W}^{\top\otimes\top}\Sigma\widetilde{W}$.
In this way, we have a well-defined linear map $\widehat{\mathcal A}\cup\widehat{\mathcal A}^*\ni a\mapsto a^{\widehat{R}}\in\widehat{\mathcal A}\cup\widehat{\mathcal A}^*$, 
which is also ${}^*$-preserving.

\item The map $\widehat{\mathcal A}\cup\widehat{\mathcal A}^*\ni a\mapsto a^{\widehat{R}}\in\widehat{\mathcal A}\cup\widehat{\mathcal A}^*$
can be alternatively characterized as follows:
\begin{align}
&\widehat{R}:(\omega\otimes\operatorname{id})(W)\mapsto Q(\omega\otimes\operatorname{id})(W^*)Q^{-1}=\bigl(\omega(Q^{-1}\,\cdot Q)\otimes\operatorname{id}\bigr)(W^*), 
{\text { for $\omega\in{\mathcal B}({\mathcal H})_*$,}}
\notag \\
&\widehat{R}:(\omega\otimes\operatorname{id})(W^*)\mapsto Q^{-1}(\omega\otimes\operatorname{id})(W)Q=\bigl(\omega(Q\,\cdot Q^{-1})\otimes\operatorname{id}\bigr)(W), 
{\text { for $\omega\in{\mathcal B}({\mathcal H})_*$.}}
\notag
\end{align}

\item We have $\widehat{R}\circ\widehat{R}\equiv\operatorname{Id}$.

\item For $a,b\in\widehat{\mathcal A}\cup\widehat{\mathcal A}^*$, we have: $ab\in\widehat{\mathcal A}\cup\widehat{\mathcal A}^*$, and 
$\widehat{R}(ab)=\widehat{R}(b)\widehat{R}(a)$.
\end{enumerate}
\end{prop}

\begin{proof}
(1). Suppose $a=0$. Then obviously $a^{\top}=0$, so $Y\bigl(1\otimes a^{\widehat{R}}\bigr)Y^*=0$. Recall here that $\widetilde{W}$ 
is full (see Lemma~\ref{Wtilde_full}), which means $\widetilde{W}^*\widetilde{W}$ is also full and so is 
$Y=\Sigma \widetilde{W}^{\top\otimes\top}[\widetilde{W}^*]^{\top\otimes\top}\Sigma=\Sigma[\widetilde{W}^*\widetilde{W}]^{\top\otimes\top}\Sigma$. 
Therefore, the only way to have $Y\bigl(1\otimes a^{\widehat{R}}\bigr)Y^*=0$ is when we also have $a^{\widehat{R}}=0$. As $a=0$ implies $a^{\widehat{R}}=0$, 
we see that $a\mapsto a^{\widehat{R}}$ becomes a well-defined linear map. 

To see that our map is ${}^*$-preserving, note the following:
$$
Y\bigl(1\otimes [a^*]^{\widehat{R}}\bigr)Y^*=Z\bigl([a^*]^{\top}\otimes1\bigr)Z^*=\bigl[Z(a^{\top}\otimes1)Z^*\bigr]^*=\bigl[Y(1\otimes a^{\widehat{R}})Y^*\bigr]^*
=Y\bigl(1\otimes [a^{\widehat{R}}]^*\bigr)Y^*,
$$
where we used the fact that $[a^*]^{\top}=[a^{\top}]^*$. Comparing the two sides, we see that $[a^*]^{\widehat{R}}=[a^{\widehat{R}}]^*$.

(2).
In particular, let $a=(\omega\otimes\operatorname{id})(W^*)\in\widehat{\mathcal A}^*$, for $\omega\in{\mathcal B}({\mathcal H})_*$. Then 
$$
a^{\top}=(\omega\otimes\operatorname{id})(W^*)^{\top}=(\omega^{\top}\otimes\operatorname{id})\bigl([W^*]^{\top\otimes\top}\bigr).
$$
We then have 
\begin{align}
Z(a^{\top}\otimes1)Z^*&=\Sigma\widetilde{W}^{\top\otimes\top}\Sigma\widetilde{W}\bigl[(\omega^{\top}\otimes\operatorname{id})([W^*]^{\top\otimes\top})\otimes1\bigr]
\widetilde{W}^*\Sigma[\widetilde{W}^*]^{\top\otimes\top}\Sigma
\notag \\
&=(\omega^{\top}\otimes\operatorname{id}\otimes\operatorname{id})(\widetilde{W}_{32}^{\top\otimes\top}\widetilde{W}_{23}[W_{12}^*]^{\top\otimes\top}
\widetilde{W}_{23}^*[\widetilde{W}_{32}^*]^{\top\otimes\top}).
\label{(Rhat_pre1)}
\end{align}

Meanwhile, from Proposition~\ref{hashcomposable_s}, we know 
$[\widetilde{W}_{23}^*]^{\top\otimes\top}\widetilde{W}_{12}\widetilde{W}_{23}^{\top\otimes\top}=[W_{13}^*]^{\top\otimes\top}\widetilde{W}_{12}$. By multiplying 
$\widetilde{W}_{23}^{\top\otimes\top}(\,\cdots\,)[\widetilde{W}_{23}^*]^{\top\otimes\top}$ to both sides of the equation, we have
$$
\widetilde{W}_{23}^{\top\otimes\top}[\widetilde{W}_{23}^*]^{\top\otimes\top}\widetilde{W}_{12}\widetilde{W}_{23}^{\top\otimes\top}[\widetilde{W}_{23}^*]^{\top\otimes\top}
=\widetilde{W}_{23}^{\top\otimes\top}[W_{13}^*]^{\top\otimes\top}\widetilde{W}_{12}[\widetilde{W}_{23}^*]^{\top\otimes\top}.
$$
Then by switching the legs 2 and 3, this becomes
\begin{align}
\widetilde{W}_{32}^{\top\otimes\top}[\widetilde{W}_{32}^*]^{\top\otimes\top}\widetilde{W}_{13}\widetilde{W}_{32}^{\top\otimes\top}[\widetilde{W}_{32}^*]^{\top\otimes\top}
&=\widetilde{W}_{32}^{\top\otimes\top}[W_{12}^*]^{\top\otimes\top}\widetilde{W}_{13}[\widetilde{W}_{32}^*]^{\top\otimes\top}  \notag \\
&=\widetilde{W}_{32}^{\top\otimes\top}\widetilde{W}_{23}[W_{12}^*]^{\top\otimes\top}\widetilde{W}_{23}^*[\widetilde{W}_{32}^*]^{\top\otimes\top},
\label{(Rhat_pre2)}
\end{align}
using the property (3) of Proposition~\ref{hashcomposable} in the last line.

The result of Equation~\eqref{(Rhat_pre2)} can be inserted into Equation~\eqref{(Rhat_pre1)}. Then
\begin{align}
Z(a^{\top}\otimes1)Z^*&=\dots
=(\omega^{\top}\otimes\operatorname{id}\otimes\operatorname{id})
\bigl(\widetilde{W}_{32}^{\top\otimes\top}[\widetilde{W}_{32}^*]^{\top\otimes\top}\widetilde{W}_{13}\widetilde{W}_{32}^{\top\otimes\top}[\widetilde{W}_{32}^*]^{\top\otimes\top}\bigr)
\notag \\
&=\Sigma\widetilde{W}^{\top\otimes\top}[\widetilde{W}^*]^{\top\otimes\top}\Sigma\bigl[1\otimes(\omega^{\top}\otimes\operatorname{id})(\widetilde{W})\bigr]
\Sigma\widetilde{W}^{\top\otimes\top}[\widetilde{W}^*]^{\top\otimes\top}\Sigma   \notag \\
&=Y\bigl[1\otimes(\omega^{\top}\otimes\operatorname{id})(\widetilde{W})\bigr]Y^*.
\notag
\end{align}
In view of the definition of the map, $a\mapsto a^{\widehat{R}}$, this means for $a=(\omega\otimes\operatorname{id})(W^*)$, we have:
$$
a^{\widehat{R}}=(\omega^{\top}\otimes\operatorname{id})(\widetilde{W}).
$$
In particular, if $\omega=\omega_{\eta,\xi}$, for $\xi\in{\mathcal D}(Q),\eta\in{\mathcal D}(Q^{-1})$, then $\omega^{\top}=\omega_{\bar{\xi},\bar{\eta}}$. So for 
any $r,s\in{\mathcal H}$, we have
$$
\bigl\langle(\omega^{\top}\otimes\operatorname{id})(\widetilde{W})r,s\bigr\rangle
=\bigl\langle\widetilde{W}(\bar{\xi}\otimes r),\bar{\eta}\otimes s\bigr\rangle=\bigl\langle W(Q^{-1}\eta\otimes r),Q\xi\otimes s\bigr\rangle
=\bigl\langle\omega_{Q^{-1}\eta,Q\xi}\otimes\operatorname{id})(\widetilde{W})r,s\bigr\rangle,
$$
showing that $([\omega_{\eta,\xi}]^{\top}\otimes\operatorname{id})(\widetilde{W})=(\omega_{Q^{-1}\eta,Q\xi}\otimes\operatorname{id})(\widetilde{W})$. 
In general, for $a=(\omega\otimes\operatorname{id})(W^*)$, we have: 
\begin{equation}\label{(Rhat_pre3)}
a^{\widehat{R}}=(\omega^{\top}\otimes\operatorname{id})(\widetilde{W})=\bigl(\omega(Q\,\cdot\,Q^{-1})\otimes\operatorname{id}\bigr)(W).
\end{equation}

In case $b=(\omega\otimes\operatorname{id})(W)\in\widehat{\mathcal A}$, regard it as $b=\bigl[(\bar{\omega}\otimes\operatorname{id})(W^*)\bigr]^*$. Since we saw 
that our map is a ${}^*$-map, and knowing the result in Equation~\eqref{(Rhat_pre3)}, we have
\begin{equation}\label{(Rhat_pre4)}
b^{\widehat{R}}=\bigl[(\bar{\omega}\otimes\operatorname{id})(W^*)^{\widehat{R}}\bigr]^*
=\bigl[\bigl(\bar{\omega}(Q\,\cdot\,Q^{-1})\otimes\operatorname{id}\bigr)(W)\bigr]^*=(\omega(Q^{-1}\,\cdot\,Q)\otimes\operatorname{id}\bigr)(W^*).
\end{equation}

The results observed in Equations~\eqref{(Rhat_pre3)} and \eqref{(Rhat_pre4)} show that for any  $a\in\widehat{\mathcal A}\cup\widehat{\mathcal A}^*$, 
the following exactly characterizes the map $a\mapsto a^{\widehat{R}}$.
\begin{align}
&\widehat{R}:(\omega\otimes\operatorname{id})(W)\mapsto Q(\omega\otimes\operatorname{id})(W^*)Q^{-1}=\bigl(\omega(Q^{-1}\,\cdot Q)\otimes\operatorname{id}\bigr)(W^*), 
{\text { for $\omega\in{\mathcal B}({\mathcal H})_*$,}}
\notag \\
&\widehat{R}:(\omega\otimes\operatorname{id})(W^*)\mapsto Q^{-1}(\omega\otimes\operatorname{id})(W)Q=\bigl(\omega(Q\,\cdot Q^{-1})\otimes\operatorname{id}\bigr)(W), 
{\text { for $\omega\in{\mathcal B}({\mathcal H})_*$.}}
\notag
\end{align}
Note here that $Q^{-1}(\omega\otimes\operatorname{id})(W)Q=\bigl(\omega(Q\,\cdot\,Q^{-1})\otimes\operatorname{id}\bigr)(W)$ is a consequence 
of $W(Q\otimes Q)\subseteq(Q\otimes Q)W$.  From this point on, we may just regard $a^{\widehat{R}}=\widehat{R}(a)$.

(3). For $b=(\omega\otimes\operatorname{id})(W)\in\widehat{\mathcal A}$, we have:
$$
(\widehat{R}\circ\widehat{R})(b)=\widehat{R}\bigl((\omega(Q^{-1}\,\cdot\,Q)\otimes\operatorname{id})(W^*)\bigr)=(\omega\otimes\operatorname{id})(W)=b.
$$
The same result holds also for any element in $\widehat{\mathcal A}^*$.

(4). Suppose $a,b\in\widehat{\mathcal A}$ be such that $a=(\omega_{u,v}\otimes\operatorname{id})(W)$ and $b=(\omega_{\xi,\eta}\otimes\operatorname{id})(W)$, 
with $u,v,\xi,\eta$ in appropriate (dense) domains in ${\mathcal H}$. From Proposition~\ref{AAhatalgebras}, we know that $ab\in\widehat{\mathcal A}$, such that
$ab=(\theta\otimes\operatorname{id})(W)$, where $\theta\in{\mathcal B}({\mathcal H})_*$ is such that $\theta(X)=(\omega_{u,v}\otimes\omega_{\xi,\eta})(W^*(1\otimes X)W)$, 
$\forall X\in{\mathcal B}({\mathcal H})$.
Compute:
\begin{align}
\widehat{R}(b)\widehat{R}(a)&=(\omega_{Q\xi,Q^{-1}\eta}\otimes\operatorname{id})(W^*)(\omega_{Qu,Q^{-1}v}\otimes\operatorname{id})(W^*)
=(\omega_{Qu,Q^{-1}v}\otimes\omega_{Q\xi,Q^{-1}\eta}\otimes\operatorname{id})(W^*_{23}W^*_{13}) \notag \\
&=(\omega_{Qu,Q^{-1}v}\otimes\omega_{Q\xi,Q^{-1}\eta}\otimes\operatorname{id})(W_{12}^*W_{23}^*W_{12}) \notag \\
&=(\omega_{u,v}\otimes\omega_{\xi,\eta}\otimes\operatorname{id})\bigl((Q^{-1}\otimes Q^{-1}\otimes1)W_{12}^*W_{23}^*W_{12}(Q\otimes Q\otimes1)\bigr)  \notag \\
&=(\omega_{u,v}\otimes\omega_{\xi,\eta}\otimes\operatorname{id})\bigl(W_{12}^*[(1\otimes Q^{-1}\otimes1)W^*_{23}(1\otimes Q\otimes1)]W_{12}\bigr)
\notag
\end{align}
where we used the multiplicativity property \eqref{(mpi2)} and the property that $W(Q\otimes Q)\subseteq(Q\otimes Q)W$. On the other hand, 
$$
\widehat{R}(ab)=\bigl(\theta(Q^{-1}\,\cdot\,Q)\otimes\operatorname{id})(W^*)
=(\omega_{u,v}\otimes\omega_{\xi,\eta}\otimes\operatorname{id})\bigl(W_{12}^*[(1\otimes Q^{-1}\otimes1)W^*_{23}(1\otimes Q\otimes1)]W_{12}\bigr),
$$
which shows that $\widehat{R}(ab)=\widehat{R}(b)\widehat{R}(a)$.

A similar computation can show that for $a,b\in\widehat{\mathcal A}^*$ also, we have $\widehat{R}(ab)=\widehat{R}(b)\widehat{R}(a)$.

At issue, however, would be the case when we have a product of one element from $\widehat{\mathcal A}$ and the other from $\widehat{\mathcal A}^*$. 
At present, we do not know the product is even in $\widehat{\mathcal A}\cup\widehat{\mathcal A}^*$. Fortunately, that happens to be the case.  
To see this, without loss of generality we may consider $x=(\omega_{u,v}\otimes\operatorname{id})(W^*)\in\widehat{\mathcal A}^*$ and 
$y=(\omega_{Q^{-1}\eta,Q\xi}\otimes\operatorname{id})(W)\in\widehat{\mathcal A}$. Again assume that $u,v,\xi,\eta$ are in appropriate (dense) domains 
in ${\mathcal H}$. It becomes convenient to consider an alternative characterization for $y$. For this, let $r,s\in{\mathcal H}$ be in the appropriate domains 
and compute:
\begin{align}
\langle yr,s\rangle&=\bigl\langle(\omega_{Q^{-1}\eta,Q\xi}\otimes\operatorname{id})(W)r,s\bigr\rangle
=\bigl\langle W(Q^{-1}\eta\otimes r),Q\xi\otimes s\bigr\rangle
=\bigl\langle W(\eta\otimes Qr),\xi\otimes Q^{-1}s\bigr\rangle \notag \\
&=\bigl\langle\widetilde{W}(\bar{\xi}\otimes r),\bar{\eta}\otimes s\bigr\rangle=\bigl\langle(\omega_{\bar{\xi},\bar{\eta}}\otimes\operatorname{id})(\widetilde{W})r,s\bigr\rangle.
\notag
\end{align}
We used again the property $(Q\otimes Q)W\subseteq W(Q\otimes Q)$ for the third equality, and the definition of $\widetilde{W}$. From this computation, 
we see that $y=(\omega_{Q^{-1}\eta,Q\xi}\otimes\operatorname{id})(W)=(\omega_{\bar{\xi},\bar{\eta}}\otimes\operatorname{id})(\widetilde{W})$. We thus have
\begin{align}
xy&=(\omega_{u,v}\otimes\operatorname{id})(W^*)(\omega_{Q^{-1}\eta,Q\xi}\otimes\operatorname{id})(W)
=(\omega_{u,v}\otimes\operatorname{id})(W^*)(\omega_{\bar{\xi},\bar{\eta}}\otimes\operatorname{id})(\widetilde{W})  \notag \\
&=(\omega_{\bar{\xi},\bar{\eta}}\otimes\omega_{u,v}\otimes\operatorname{id})(W_{23}^*\widetilde{W}_{13}) 
=(\omega_{\bar{\xi},\bar{\eta}}\otimes\omega_{u,v}\otimes\operatorname{id})(\widetilde{W}_{12}W_{23}^*\widetilde{W}_{12}^*),
\label{(productAhat*A)}
\end{align}
using the result from Proposition~\ref{hashcomposable}\,(4). From this we can write $xy=(\rho\otimes\operatorname{id})(W^*)$, where $\rho\in{\mathcal B}({\mathcal H})_*$ 
is such that $\rho(X)=(\omega_{\bar{\xi},\bar{\eta}}\otimes\omega_{u,v})(\widetilde{W}(1\otimes X)\widetilde{W}^*)$. In particular, $xy\in\widehat{\mathcal A}^*$.

We can now compute $\widehat{R}(xy)$ and $\widehat{R}(y)\widehat{R}(x)$ then compare. Note that
$$
\widehat{R}(xy)=\bigl(\rho(Q\,\cdot\,Q^{-1})\otimes\operatorname{id}\bigr)(W)
=(\omega_{\bar{\xi},\bar{\eta}}\otimes\omega_{u,v}\otimes\operatorname{id})(\widetilde{W}_{12}(1\otimes Q\otimes1)W_{23}(1\otimes Q^{-1}\otimes1)\widetilde{W}_{12}^*).
$$
While we have
\begin{align}
\widehat{R}(y)\widehat{R}(x)&=\widehat{R}\bigl((\omega_{Q^{-1}\eta,Q\xi}\otimes\operatorname{id})(W)\bigr)
\widehat{R}\bigl((\omega_{u,v}\otimes\operatorname{id})(W^*)\bigr)
=(\omega_{\eta,\xi}\otimes\operatorname{id})(W^*)(\omega_{Q^{-1}u,Qv}\otimes\operatorname{id})(W)  \notag \\
&=\bigl[(\omega_{Q^{-1}u,Qv}\otimes\operatorname{id})(W)^*(\omega_{\eta,\xi}\otimes\operatorname{id})(W^*)^*\bigr]^*
=\bigl[(\omega_{Qv,Q^{-1}u}\otimes\operatorname{id})(W^*)(\omega_{\xi,\eta}\otimes\operatorname{id})(W)\bigr]^*   \notag \\
&=\bigl[(\omega_{[Q^{-1}]^{\top}\bar{\eta},Q^{\top}\bar{\xi}}\otimes\omega_{Qv,Q^{-1}u}\otimes\operatorname{id})(\widetilde{W}_{12}W_{23}^*\widetilde{W}_{12}^*)\bigr]^*   
\notag \\
&=(\omega_{Q^{\top}\bar{\xi},[Q^{-1}]^{\top}\bar{\eta}}\otimes\omega_{Q^{-1}u,Qv}\otimes\operatorname{id})(\widetilde{W}_{12}W_{23}\widetilde{W}_{12}^*).
\notag
\end{align}
Note that we used the result of Equation~\eqref{(productAhat*A)} for the fifth equality. Finally, use again the result 
$\widetilde{W}(Q^{\top}\otimes Q^{-1})\subseteq (Q^{\top}\otimes Q^{-1})\widetilde{W}$. Then the expression above becomes
$$
\widehat{R}(xy)=\dots
=(\omega_{\bar{\xi},\bar{\eta}}\otimes\omega_{u,v}\otimes\operatorname{id})(\widetilde{W}_{12}(1\otimes Q\otimes1)W_{23}(1\otimes Q^{-1}\otimes1)\widetilde{W}_{12}^*),
$$
thereby proving that $\widehat{R}(xy)=\widehat{R}(y)\widehat{R}(x)$.
\end{proof}

\begin{rem}
It is shown later that $\widehat{R}$ can be extended to the $C^*$-algebra level $\widehat{A}={\overline{\widehat{\mathcal A}}}^{\|\ \|}$, as an involutive anti-multiplicative 
isomorphism into itself. However, that can be properly made sense after we show that $\widehat{A}$ is self-adjoint so that $\widehat{A}$ becomes a $C^*$-algebra. 
At present we only need the results at the dense subspace level, so we will postpone the extension result to a later section.  Having said this, note that even without 
knowing that $\widehat{R}$ can be extended to a ${}^*$-anti-isomorphism on $A$, the results at the dense subspace level holds as above. 
\end{rem}

As a consequence of Proposition~\ref{Rhat_pre}, we can now prove that $\widetilde{W}$ is itself a partial isometry:

\begin{prop}\label{Wtildepartialiso}
The operator $\widetilde{W}$ is a partial isometry.
\end{prop}

\begin{proof}
Let $\xi,u\in{\mathcal D}(Q)$ and $\eta,v\in{\mathcal D}(Q^{-1})$, which are dense in ${\mathcal H}$. From Definition~\ref{manageable}, we know
$$
\bigl\langle \widetilde{W}(\bar{\xi}\otimes u),\bar{\eta},v\bigr\rangle=\bigl\langle W(\eta\otimes Qu),\xi\otimes Q^{-1}v\bigr\rangle.
$$
Since $m^{\top}\bar{\xi}=\overline{m^*\xi}$ in general, we see that the first leg of $\widetilde{W}$ acts as the transpose on the first leg of $W^*$. 
Meanwhile from $\widehat{R}\bigl((\omega_{\eta,\xi}\otimes\operatorname{id})(W^*)\bigr)=(\omega_{Q^{-1}\eta,Q\xi}\otimes\operatorname{id})(W)
=Q^{-1}(\omega_{\eta,\xi}\otimes\operatorname{id})(W)Q$, we see that the second leg of $\widetilde{W}$ acts as $\widehat{R}$ on the second leg of $W^*$.
These observations suggest the following heuristic characterization of $\widetilde{W}$:
$$\widetilde{W}=[W^*]^{\top\otimes\widehat{R}},$$ 
where the $\widehat{R}$ map is given using the exponential notation.

A problem with this characterization is that while $\widetilde{W}$ is a bounded operator, at present we do not know whether $[W^*]^{\top\otimes\widehat{R}}$ 
extends further. Nonetheless, as long as we work with the vectors in appropriate domains, we indeed have the following:
\begin{align}
\bigl\langle [W^*]^{\top\otimes\widehat{R}}(\bar{\xi}\otimes u),\bar{\eta},v\bigr\rangle
&=\bigl\langle[W^*]^{\operatorname{id}\otimes\widehat{R}}(\eta\otimes u),\xi\otimes v\bigr\rangle
=\bigl\langle\widehat{R}((\omega_{\eta,\xi}\otimes\operatorname{id})(W^*))u,v\bigr\rangle  \notag \\
&=\bigl\langle Q^{-1}(\omega_{\eta,\xi}\otimes\operatorname{id})(W)Qu,v\bigr\rangle
=\bigl\langle W(\eta\otimes Qu),\xi\otimes Q^{-1}v\bigr\rangle,
\notag
\end{align}
confirming the characterization. 

For our purposes, this is more or less sufficient, so we will wait until later to fully extend the $\widehat{R}$ map, thereby making the expression 
$[W^*]^{\top\otimes\widehat{R}}$ proper. At least we can see that it is densely-defined, and it agrees with $\widetilde{W}$ when it makes sense. 
Note also that both $(\,)^{\operatorname{\top}}$ and $\widehat{R}$ are involutive and anti-multiplicative. Therefore, at least at the level of a dense subspace 
in ${\mathcal H}\otimes{\mathcal H}$, we have:
$$
\widetilde{W}\widetilde{W}^*\widetilde{W}=([W^*]^{\top\otimes\widehat{R}})([W^*]^{\top\otimes\widehat{R}})^*([W^*]^{\top\otimes\widehat{R}})
=[W^*WW^*]^{\top\otimes\widehat{R}}=[W^*]^{\top\otimes\widehat{R}}=\widetilde{W},
$$
because $W^*WW^*=W^*$. Since this is apparently valid in a dense subspace of ${\mathcal H}\otimes{\mathcal H}$ while $\widetilde{W}$ is a bounded operator, 
this means $\widetilde{W}\widetilde{W}^*\widetilde{W}=\widetilde{W}$ everywhere on  ${\mathcal H}\otimes{\mathcal H}$, showing that $\widetilde{W}$ is indeed 
a partial isometry.
\end{proof}

We are now ready to prove that $\widehat{A}$ and $A$, namely the norm closures of $\widehat{\mathcal A}$ and ${\mathcal A}$, are both $C^*$-algebras.

\begin{theorem}\label{c*algebras}
Let $W$ be a manageable multiplicative partial isometry satisfying the fullness condition, and let $Q$ and $\widetilde{W}$ be the associated operators 
given in Definition~\ref{manageable}.  Also let $A$ and $\widehat{A}$ be the norm closures of the subalgebras ${\mathcal A}$ and $\widehat{\mathcal A}$, 
respectively:
$$
A=\overline{\operatorname{span}\bigl\{(\operatorname{id}\otimes\omega)(W):\omega\in{\mathcal B}({\mathcal H})_*\bigr\}}^{\|\ \|} \ {\text { and }} \ 
\widehat{A}=\overline{\operatorname{span}\bigl\{(\omega\otimes\operatorname{id})(W):\omega\in{\mathcal B}({\mathcal H})_*\bigr\}}^{\|\ \|}.
$$
Then $A$ and $\widehat{A}$ are separable $C^*$-algebras acting on ${\mathcal H}$ in a non-degenerate way.
\end{theorem}

\begin{proof}
While carrying out the proof of Proposition~\ref{Rhat_pre}, we saw that for elements of the form $x=(\omega_{u,v}\otimes\operatorname{id})(W^*)
\in\widehat{\mathcal A}^*$ and $y=(\omega_{Q^{-1}\eta,Q\xi}\otimes\operatorname{id})(W)\in\widehat{\mathcal A}$, we have 
$$
xy=(\omega_{u,v}\otimes\operatorname{id})(W^*)(\omega_{Q^{-1}\eta,Q\xi}\otimes\operatorname{id})(W)=(\rho\otimes\operatorname{id})(W^*)\in\widehat{\mathcal A}^*,
$$
where $\rho\in{\mathcal B}({\mathcal H})_*$ given by $\rho(X)=(\omega_{\bar{\xi},\bar{\eta}}\otimes\omega_{u,v})(\widetilde{W}(1\otimes X)\widetilde{W}^*)$. 
See Equation~\eqref{(productAhat*A)}, which used the result from Proposition~\ref{hashcomposable}\,(4).

It is evident that the elements of the form $x$ and $y$ above are dense in $\widehat{\mathcal A}^*$ and ${\mathcal A}$, respectively. Meanwhile, we can 
also show that the functionals of the form $\rho$ above are dense in ${\mathcal B}({\mathcal H})_*$. To see this, recall that $\widetilde{W}$ is a partial isometry 
satisfying the fullness condition (Proposition~\ref{Wtildepartialiso} and Lemma~\ref{Wtilde_full}). As $\widetilde{W}$ is a partial isometry, we can write 
$\overline{\mathcal H}\otimes{\mathcal H}=\operatorname{Ker}(\widetilde{W})\oplus\operatorname{Ran}(\widetilde{W}^*\widetilde{W})$, where 
$\widetilde{W}^*\widetilde{W}$ is the initial projection of $\widetilde{W}$. We know that $\widetilde{W}$ is an isometry from 
$\operatorname{Ran}(\widetilde{W}^*\widetilde{W})$ onto $\operatorname{Ran}(\widetilde{W}\widetilde{W}^*)$. Similar for $\widetilde{W}^*$, 
which is an isometry from $\operatorname{Ran}(\widetilde{W}\widetilde{W}^*)$ onto $\operatorname{Ran}(\widetilde{W}^*\widetilde{W})$.

Consider an arbitrary nonzero operator $X\in{\mathcal B}({\mathcal H})$. A quick application of the fullness condition of $\widetilde{W}$ shows that 
the restriction operator $\widetilde{W}^*\widetilde{W}(1\otimes X)\widetilde{W}^*\widetilde{W}\bigl|_{\operatorname{Ran}(\widetilde{W}^*\widetilde{W})}$ is 
also a nonzero operator. 
But then, this means we can find $\bar{\xi}\otimes u\in\operatorname{Ran}(\widetilde{W}\widetilde{W}^*)$ such that 
$\widetilde{W}^*\widetilde{W}(1\otimes X)\widetilde{W}^*(\bar{\xi}\otimes u)\ne0$. 
As we are restricting down to the projection subspace, we can find $\eta,v\in{\mathcal H}$ with 
$\widetilde{W}^*(\bar{\eta}\otimes v)\in\operatorname{Ran}(\widetilde{W}^*\widetilde{W})$, such that 
$\bigl\langle \widetilde{W}^*\widetilde{W}(1\otimes X)\widetilde{W}^*(\bar{\xi}\otimes u),\widetilde{W}^*(\bar{\eta}\otimes v)\bigr\rangle\ne0$. This is equivalent to 
saying 
$$
0\ne\bigl\langle \widetilde{W}(1\otimes X)\widetilde{W}^*(\bar{\xi}\otimes u),\bar{\eta}\otimes v\bigr\rangle
=(\omega_{\bar{\xi},\bar{\eta}}\otimes\omega_{u,v})(\widetilde{W}(1\otimes X)\widetilde{W}^*)=\rho(X),
$$
because $\widetilde{W}\widetilde{W}^*\widetilde{W}=\widetilde{W}$. This can be aways done for any nonzero operator $X$, which means that the functionals 
of the form $\rho$ are dense in ${\mathcal B}({\mathcal H})_*$.  So the elements $(\rho\otimes\operatorname{id})(W^*)$ are dense in $\widehat{\mathcal A}^*$.

Putting these observations together, we can conclude that $\widehat{\mathcal A}^*\widehat{\mathcal A}$ is linearly dense in $\widehat{\mathcal A}^*$. 
Considering the norm-closures, it thus follows that $\widehat{A}^*\widehat{A}=\widehat{A}^*$.  Since $\widehat{A}^*\widehat{A}$ is ${}^*$-closed, 
so should $\widehat{A}^*$, or $\widehat{A}^*=\widehat{A}$. Therefore $\widehat{A}$ is a norm-closed ${}^*$-subalgebra of ${\mathcal B}({\mathcal H})$, 
meaning that $\widehat{A}$ is a $C^*$-algebra. Meanwhile, since $\widehat{\mathcal A}$ is nondegenerately represented in ${\mathcal B}({\mathcal H})$ 
(Lemma~\ref{Anondegnerate}), so should $\widehat{A}$.  Since ${\mathcal H}$ is separable, the $C^*$-algebra $\widehat{A}$ is separable.

Similarly, by working instead with (3) of Proposition~\ref{hashcomposable}, which is equivalent to replacing $W$ and $\widetilde{W}$ with 
$\widehat{W}=\Sigma W^*\Sigma$ and $\widetilde{\widehat{W}}=(\Sigma\widetilde{W}^*\Sigma)^{\top\otimes\top}$, we can also show that 
${\mathcal A}{\mathcal A}^*$ is dense in ${\mathcal A}$.  We have $A^*=A$ and that $A$ is also a (separable) $C^*$-algebra nondegenerately 
represented in ${\mathcal B}({\mathcal H})$.
\end{proof}

\section{The coalgebra structures on $A$ and $\widehat{A}$}\label{sec3}

Rest of the way, we assume that $W$ is a full manageable multiplicative partial isometry, with the associated $C^*$-algebras 
$A$ and $\widehat{A}$.  In this section, we wish to explore the restrictions of the maps $\Delta$ and $\widehat{\Delta}$ considered 
in Proposition~\ref{deltadeltahat} to the subalgebras $A$ and $\widehat{A}$, respectively, and show that they determine comultiplications 
on these subalgebras.

Before we construct the comultiplication map on $A$, let us prove the following lemma: 

\begin{lem}\label{lem_Deltahom}
\begin{enumerate}
 \item For any $x\in A$, we have: $(1\otimes x)WW^*=WW^*(1\otimes x)$.
 \item For any $y\in\widehat{A}$, we have: $(y\otimes 1)W^*W=W^*W(y\otimes 1)$.
\end{enumerate}
\end{lem}

\begin{proof}
(1). Let $x=(\operatorname{id}\otimes\omega)(W)$, for an arbitrary $\omega\in{\mathcal B}({\mathcal H})_*$. Note that 
by Equation~\eqref{(mpi4)}, 
$$
(1\otimes x)WW^*=(\operatorname{id}\otimes\operatorname{id}\otimes\omega)(W_{23}W_{12}W_{12}^*)
=(\operatorname{id}\otimes\operatorname{id}\otimes\omega)(W_{12}W_{12}^*W_{23})=WW^*(1\otimes x). 
$$

(2). The proof that $(y\otimes 1)W^*W=W^*W(y\otimes 1)$, for $y\in\widehat{A}$, is similarly done, using Equation~\eqref{(mpi3)}.
\end{proof}

\begin{cor}
\begin{enumerate}
 \item For any $m\in M(A)$, we have: $(1\otimes m)WW^*=WW^*(1\otimes m)$.
 \item For any $n\in M(\widehat{A})$, we have: $(n\otimes 1)W^*W=W^*W(n\otimes 1)$.
\end{enumerate}
\end{cor}

\begin{proof}
(1). Let $m\in M(A)$. Then for any $a\in A$, we know that $am\in A$.  Then by the above lemma, we have $(1\otimes am)WW^*
=WW^*(1\otimes am)$, or $(1\otimes a)(1\otimes m)WW^*=WW^*(1\otimes a)(1\otimes m)$. 
By applying the lemma again, this becomes
$$
(1\otimes a)(1\otimes m)WW^*=WW^*(1\otimes a)(1\otimes m)=(1\otimes a)WW^*(1\otimes m).
$$
As this result is true for any $a\in A$, it follows that $(1\otimes m)WW^*=WW^*(1\otimes m)$.  

Proof for (2) is similar.
\end{proof}

Consider the map $\Delta:{\mathcal B}({\mathcal H})\to{\mathcal B}({\mathcal H}\otimes{\mathcal H})$ introduced earlier, 
and consider its restriction to the subalgebra $A$.  The next proposition shows that it determines a ${}^*$-homomorphism 
on $A$, which extends to a ${}^*$-homomorphism on $M(A)$.  

\begin{prop}\label{Delta_hom}
Consider the map $\Delta:A\to{\mathcal B}({\mathcal H}\otimes{\mathcal H})$, given by 
$$
\Delta(x)=W^*(1\otimes x)W, \quad x\in A.
$$
It is a ${}^*$-homomorphism on $A$, which extends to a ${}^*$-homomorphism 
$\Delta:M(A)\to{\mathcal B}({\mathcal H}\otimes{\mathcal H})$.
\end{prop}

\begin{proof}
It is evident that $\Delta$ is a ${}^*$-map.  Meanwhile, let $a,b\in A$.  Then we have:
$$
\Delta(a)\Delta(b)=W^*(1\otimes a)WW^*(1\otimes b)W=W^*WW^*(1\otimes a)(1\otimes b)W=W^*(1\otimes ab)W=\Delta(ab),
$$
by Lemma~\ref{lem_Deltahom}\,(1).  By its Corollary, we see also that $\Delta$ extends to a ${}^*$-homomorphism on $M(A)$.
\end{proof}

In fact, we can show later that $\Delta(A)\subseteq M(A\otimes A)$.  But for the time being, let us turn our attention to exploring the properties of the projection 
$E=W^*W$.

\begin{prop}\label{Delta(1)}
\begin{enumerate}
\item We have: $E=W^*W\in M(A\otimes A)$.
\item We have: $E=W^*W=\Delta(1_{M(A)})=\Delta(1)$.
\item For any $a\in A$, we have: $E(\Delta a)=\Delta a = (\Delta a)E$.
\end{enumerate}
\end{prop}

\begin{proof}
(1). Knowing $A=A^*$, consider $x=(\operatorname{id}\otimes\omega)(W)\in A$ and $y=(\operatorname{id}\otimes\omega')(W^*)\in A$, 
for $\omega,\omega'\in{\mathcal B}({\mathcal H})_*$.  Such elements are dense in $A$.  Then we have: 
\begin{align}
E(x\otimes y)&=(\operatorname{id}\otimes\operatorname{id}\otimes\omega\otimes\omega')(W^*_{12}W_{12}W_{13}W^*_{24})
\notag \\
&=(\operatorname{id}\otimes\operatorname{id}\otimes\omega\otimes\omega')(W_{13}W_{23}W_{23}^*W_{24}^*) 
=(\operatorname{id}\otimes\operatorname{id}\otimes\omega\otimes\omega')(W_{13}W_{24}^*W_{43}W_{43}^*),
\notag
\end{align}
where we used Equation~\eqref{(mpi6)} for the second equality and Equation~\eqref{(mpi10)} for the third.

Without loss of generality, we may take $\omega=\omega(\,\cdot\,k)$ and $\omega'=\omega(\,\cdot\,k')$, where $k,k'\in{\mathcal B}_0({\mathcal H})$ 
are arbitrary compact operators. As $W\in M\bigl({\mathcal B}_0({\mathcal H}\otimes{\mathcal H})\bigr)$, we know we can approximate 
$W_{21}W_{21}^*(k\otimes k')$ by the elements of the form $p\otimes p'$, where $p,p'\in{\mathcal B}_0({\mathcal H})$.  This means that 
we can approximate $(\omega\otimes\omega')\bigl(\,\cdot\,W_{21}W_{21}^*(k\otimes k')\bigr)$ by the functionals of the form $\theta\otimes\theta'$, 
where $\theta=\omega(\,\cdot\,p)$ and $\theta'=\omega(\,\cdot\,p')$.  It follows that $E(x\otimes y)$ can be approximated by the elements 
of the form $(\operatorname{id}\otimes\operatorname{id}\otimes\theta\otimes\theta')(W_{13}W_{24}^*)=a_1\otimes a_2$, where 
$a_1=(\operatorname{id}\otimes\theta)(W)\in A$, $a_2=(\operatorname{id}\otimes\theta)(W^*)\in A$, 
so $E(x\otimes y)\in A\otimes A$. As $x,y$ are arbitrary, this shows that $E\in M(A\otimes A)$.

(2), (3). The result of Proposition~\ref{Delta_hom} confirms that $E=\Delta(1_{M(A)})$. Meanwhile, we know that $A$ acts on ${\mathcal H}$ in 
a nondegenerate way, which means $1_{M(A)}=\operatorname{Id}_{{\mathcal B}({\mathcal H})}=1$. So we may write $E=\Delta(1)$. 
It is also evident that for any $a\in A$, we have:
$$
E(\Delta a)=W^*WW^*(1\otimes a)W=W^*(1\otimes a)W=\Delta a,
$$
and similarly $(\Delta a)E=\Delta a$. 
\end{proof}

Having $\Delta(1)=E\ne1\otimes1$ indicates that $\Delta$ cannot be a nondegenerate map. That also means $\overline{\Delta(A)({\mathcal H}\otimes{\mathcal H})}
\subsetneq{\mathcal H}\otimes{\mathcal H}$. Instead it turns out that $\overline{\Delta(A)({\mathcal H}\otimes{\mathcal H})}=E({\mathcal H}\otimes{\mathcal H})
=\operatorname{Ran}(E)$:

\begin{prop}\label{DeltaAdegenerate}
Let $E=W^*W$ defined above. We have:
$$
\overline{\Delta(A)({\mathcal H}\otimes{\mathcal H})}=E({\mathcal H}\otimes{\mathcal H})=\operatorname{Ran}(E).
$$
\end{prop}

\begin{proof}
As $\Delta a=E(\Delta a)$ for any $a\in A$, it is clear that $\overline{\Delta(A)({\mathcal H}\otimes{\mathcal H})}\subseteq E({\mathcal H}\otimes{\mathcal H})$. 

To see if they are actually equal, consider an arbitrary nonzero element $u\otimes v\in E({\mathcal H}\otimes{\mathcal H})$. As $W$ acts as an isometry on 
$E({\mathcal H}\otimes{\mathcal H})$ and since $A$ acts nondegenerately on ${\mathcal H}$, we can find $x\in A$ such that $(\Delta x)(u\otimes v)
=W^*(1\otimes x)W(u\otimes v)\ne 0$. So we can find $\xi\otimes\eta\in{\mathcal H}\otimes{\mathcal H}$ such that 
$\bigl\langle(\Delta x)(u\otimes v),\xi\otimes\eta\bigr\rangle\ne0$, or equivalently $\bigl\langle u\otimes v,\Delta(x^*)(\xi\otimes\eta)\bigr\rangle\ne0$. 
What this means is that there is no nonzero $u\otimes v\in E({\mathcal H}\otimes{\mathcal H})$ that is orthogonal to 
$\overline{\Delta(A)({\mathcal H}\otimes{\mathcal H})}$. From this, we conclude that 
$\overline{\Delta(A)({\mathcal H}\otimes{\mathcal H})}=E({\mathcal H}\otimes{\mathcal H})$.
\end{proof}

Here are some other results that seem natural, from the observation $E=\Delta(1)$:

\begin{prop}\label{Estrictlimit}
Let $(a_i)$ be an approximate unit in the $C^*$-algebra $A$.  Then we have:
$$
\Delta(a_i)\xrightarrow{\text{ strictly }}E=\Delta(1).
$$
\end{prop}

\begin{proof}
For any $x\in A$, we have $a_ix\xrightarrow{\text{ norm }}x$. As $\Delta$ is a ${}*$-homomorphism, it follows that $\Delta(a_ix)\xrightarrow{\text{ norm }}\Delta x$ 
which in turn means $\Delta(a_i)(\Delta x)\xrightarrow{\text{ norm }}E(\Delta x)$, by Proposition~\ref{Delta(1)}\,(3). 
Therefore, for any $\xi\otimes\eta\in{\mathcal H}\otimes{\mathcal H}$ we have:
$\Delta(a_i)(\Delta x)(\xi\otimes\eta)\longrightarrow E(\Delta x)(\xi\otimes\eta)$ in ${\mathcal H}\otimes{\mathcal H}$.

We saw that $\overline{\Delta(A)({\mathcal H}\otimes{\mathcal H})}=E({\mathcal H}\otimes{\mathcal H})$, 
so the elements of the form $(\Delta x)(\xi\otimes\eta)$ are dense in $E({\mathcal H}\otimes{\mathcal H})$. Therefore, the above observation means that 
$\Delta(a_i)(\xi\otimes\eta)\longrightarrow E(\xi\otimes\eta)$, for any $\xi\otimes\eta\in E({\mathcal H}\otimes{\mathcal H})$. Meanwhile, 
on $E({\mathcal H}\otimes{\mathcal H})^{\perp}=(1-E)({\mathcal H}\otimes{\mathcal H})$, it is evident that $\Delta(a_i)\,|_{(1-E)({\mathcal H}\otimes{\mathcal H})}
=E\,|_{(1-E)({\mathcal H}\otimes{\mathcal H})}=0$. By putting these results together, we can see that for any $\xi\otimes\eta\in{\mathcal H}$, we have
$\Delta(a_i)(\xi\otimes\eta)\longrightarrow E(\xi\otimes\eta)$ in ${\mathcal H}\otimes{\mathcal H}$. That is, we have:
$\Delta(a_i)\xrightarrow{\text{ SOT }}E$.

Recall that $A$ acts non-degenerately on ${\mathcal H}$. Therefore $A{\mathcal H}\otimes A{\mathcal H}$ is dense in ${\mathcal H}\otimes{\mathcal H}$. As a result, 
the convergence in SOT for the bounded sequence $\bigl(\Delta(a_i)\bigr)$ to $E$ actually coincides with the strict convergence in $M(A\otimes A)$. In other words, 
we have: $\Delta(a_i)\xrightarrow{\text{ strictly }}E$.
\end{proof}

\begin{prop}\label{E_weakcomultiplicative}
$E\otimes1$ and $1\otimes E$ commute in $M(A\otimes A\otimes A)\,\subseteq{\mathcal B}({\mathcal H}\otimes{\mathcal H}\otimes{\mathcal H})$. 

Furthermore, we have:
$$
(\Delta\otimes\operatorname{id})\Delta(1)=(\operatorname{id}\otimes\Delta)\Delta(1).
$$
\end{prop}

\begin{proof}
By Equation~\eqref{(mpi3)}, we have:
$(E\otimes1)(1\otimes E)=W_{12}^*W_{12}W_{23}^*W_{23}=W_{12}^*W_{23}^*W_{23}W_{12}$. Meanwhile, by twice using 
Equation~\eqref{(mpi9)}, we have: $(1\otimes E)(E\otimes1)=W_{23}^*W_{23}W_{12}^*W_{12}=W_{23}^*W_{13}^*W_{13}W_{23}=W_{12}^*W_{12}W_{23}^*W_{23}$,
showing that
$$
(E\otimes1)(1\otimes E)=(1\otimes E)(E\otimes1)=W_{12}^*W_{23}^*W_{23}W_{12}=W_{23}^*W_{13}^*W_{13}W_{23}.
$$
The commutativity of $E\otimes1$ and $1\otimes E$ can be observed. Moreover. we also observe 
$W_{12}^*W_{23}^*W_{23}W_{12}=W_{23}^*W_{13}^*W_{13}W_{23}$. Remembering that $\Delta(1)=E=W^*W$ 
(Proposition~\ref{Delta(1)}), this is none other than saying $(\Delta\otimes\operatorname{id})\Delta(1)=(\operatorname{id}\otimes\Delta)\Delta(1)$.
\end{proof}

The result $(\Delta\otimes\operatorname{id})\Delta(1)=(\operatorname{id}\otimes\Delta)\Delta(1)$ was to be expected, considering Proposition~\ref{deltadeltahat}. 
But this also gives an indication that the comultiplication can be considered fully within the multiplier algebra setting. The following proposition indeed shows that 
$\Delta(A)\subseteq M(A\otimes A)$, but we actually have stronger results:

\begin{prop}\label{delta_multiplier}
Let $a,b\in A$ be arbitrary.  We have:
$$
(a\otimes1)(\Delta b)\in A\otimes A, \quad (\Delta a)(1\otimes b)\in A\otimes A,
$$
$$
(\Delta a)(b\otimes1)\in A\otimes A, \quad (1\otimes a)(\Delta b)\in A\otimes A.
$$
\end{prop}

\begin{proof}
Let $a=(\operatorname{id}\otimes\omega)(W)$, $b=(\operatorname{id}\otimes\omega')(W)$, for arbitrary 
$\omega,\omega'\in{\mathcal B}({\mathcal H})_*$.  Such elements are dense in $A$.  We have: 
\begin{align}
(a\otimes1)(\Delta b)&=(\operatorname{id}\otimes\operatorname{id}\otimes\omega\otimes\omega')(W_{13}W_{12}^*W_{24}W_{12})
=(\operatorname{id}\otimes\operatorname{id}\otimes\omega\otimes\omega')(W_{13}W_{14}W_{24})  \notag \\
&=(\operatorname{id}\otimes\operatorname{id}\otimes\omega\otimes\omega')(W_{34}W_{13}W_{34}^*W_{24}),
\notag 
\end{align}
where we used Equations~\eqref{(mpi2)} and \eqref{(mpi1)} in the second and the third equalities, respectively. Using a similar trick as before, since 
$W\in M\bigl({\mathcal B}_0({\mathcal H}\otimes{\mathcal H})\bigr)$ and since we may take the functionals $\omega(k\,\cdot\,),\omega'(k'\,\cdot\,)$ 
for $k,k'\in{\mathcal B}_0({\mathcal H})$, without loss of generality we can regard the above 
expression as 
$$(a\otimes1)(\Delta b)=(\operatorname{id}\otimes\operatorname{id}\otimes\theta\otimes\theta')(W_{13}W_{34}^*W_{24}),$$ 
for suitable functionals $\theta,\theta'$. Apply here Lemma~\ref{lem_QTT}\,(1). Then it can be further written that 
$$
(a\otimes1)(\Delta b)=(\operatorname{id}\otimes\operatorname{id}\otimes\phi\otimes\theta')(\widetilde{W}_{34}^*\widetilde{W}_{13}^{\top\otimes\top}W_{24}), 
$$
for a suitable functional $\phi\in{\mathcal B}(\overline{\mathcal H})_*$. Again, with $\widetilde{W}^*\in M\bigl({\mathcal B}_0(\overline{\mathcal H}\otimes{\mathcal H})\bigr)$ 
and using a similar trick as before, we can now write:
$$
(a\otimes1)(\Delta b)=(\operatorname{id}\otimes\operatorname{id}\otimes\rho\otimes\rho')(\widetilde{W}_{13}^{\top\otimes\top}W_{24})
=(\operatorname{id}\otimes\rho)(\widetilde{W}^{\top\otimes\top})\otimes(\operatorname{id}\otimes\rho')(W),
$$
for suitable functionals $\rho$, $\rho'$. Here, we see right away that $(\operatorname{id}\otimes\rho')(W)\in A$. Meanwhile, 
by a computation similar to the one done in Lemma~\ref{lem_QT}, we can show that for functionals of the form $\rho=\omega_{\bar{v},\bar{u}}$, 
for $u\in{\mathcal D}(Q),v\in{\mathcal D}(Q^{-1})$, we have 
$(\operatorname{id}\otimes\omega_{\bar{v},\bar{u}})(\widetilde{W}^{\top\otimes\top})=(\operatorname{id}\otimes\omega_{Qu,Q^{-1}v})(W)$. This means 
that $(\operatorname{id}\otimes\rho)(\widetilde{W}^{\top\otimes\top})\in A$ as well, and in this way we can prove that $(a\otimes1)(\Delta b)\in A\otimes A$.

For $a=(\operatorname{id}\otimes\omega)(W)$, $b=(\operatorname{id}\otimes\omega')(W)$, $\omega,\omega'\in{\mathcal B}({\mathcal H})_*$, using 
the multiplicative properties of the operator $W$, we have: 
\begin{align}
(1\otimes a)(\Delta b)&=(\operatorname{id}\otimes\operatorname{id}\otimes\omega'\otimes\omega)(W_{24}W_{12}^*W_{23}W_{12}) 
\notag \\
&=(\operatorname{id}\otimes\operatorname{id}\otimes\omega'\otimes\omega)(W_{24}W_{13}W_{23})  
=(\operatorname{id}\otimes\operatorname{id}\otimes\omega'\otimes\omega)(W_{13}W_{24}W_{23})  \notag \\
&=(\operatorname{id}\otimes\operatorname{id}\otimes\omega'\otimes\omega)(W_{13}W_{43}W_{24}W_{43}^*).
\notag
\end{align}
Without loss of generality, we can consider the functionals $\omega'(\,\cdot\,p),\omega(\,\cdot\,q)$, $p,q\in{\mathcal B}_0({\mathcal H})$, 
and again use the fact that $W^*\in M\bigl({\mathcal B}_0({\mathcal H}\otimes{\mathcal H})\bigr)$. Then we can write this as 
$$
(1\otimes a)(\Delta b)=(\operatorname{id}\otimes\operatorname{id}\otimes\omega'\otimes\omega)(W_{13}W_{43}W_{24}).
$$
Here we may use the result of Lemma~\ref{lem_QTT}\,(2), for the legs 2,4,3. Then for a suitable functional $\phi\in{\mathcal B}(\overline{\mathcal H})_*$, 
we can write this as 
$$
(1\otimes a)(\Delta b)=(\operatorname{id}\otimes\operatorname{id}\otimes\omega'\otimes\phi)\bigl((1\otimes Q^{-2}\otimes1\otimes1)
W_{13}\widetilde{W}_{24}^{\top\otimes\top}\widetilde{W}_{43}(1\otimes Q^2\otimes1\otimes1)\bigr).
$$
Once again using the fact that $\Sigma\widetilde{W}\Sigma\in M\bigl({\mathcal B}_0({\mathcal H}\otimes\overline{\mathcal H})\bigr)$ and adjusting 
the functionals, without loss of generality we can regard this expression as follows:
$$
(1\otimes a)(\Delta b)=(\operatorname{id}\otimes\operatorname{id}\otimes\rho'\otimes\rho)\bigl((1\otimes Q^{-2}\otimes1\otimes1)
W_{13}\widetilde{W}_{24}^{\top\otimes\top}(1\otimes Q^2\otimes1\otimes1)\bigr)=a_1\otimes a_2\in A\otimes A,
$$
where $a_1=(\operatorname{id}\otimes\rho')(W)\in A$ is immediate to see, while $a_2=Q^{-2}(\operatorname{id}\otimes\rho)(\widetilde{W}^{\top\otimes\top})Q^2\in A$
as well, because for $\rho=\omega_{\bar{\xi},\bar{\eta}}\in{\mathcal B}(\overline{\mathcal H})_*$, we can compute that 
$Q^{-2}(\operatorname{id}\otimes\omega_{\bar{\xi},\bar{\eta}})(\widetilde{W}^{\top\otimes\top})Q^2=(\operatorname{id}\otimes\omega_{Q^{-1}\eta,Q\xi})(W)$.

The remaining two results can be obtained by taking the adjoints of these two results.
\end{proof}

\begin{cor}
For any $a\in A$, we have: $\Delta a\in M(A\otimes A)$. 
\end{cor}

\begin{proof}
Let $a\in A$ and consider $\Delta a=W^*(1\otimes a)W$.  Let $b,c\in A$ be arbitrary.  Then by Proposition~\ref{delta_multiplier}, 
we have $(\Delta a)(b\otimes1)\in A\otimes A$, so $(\Delta a)(b\otimes c)\in A\otimes A$.  Similarly, we have 
$(b\otimes c)(\Delta a)\in A\otimes A$.  It follows that $\Delta a\in M(A\otimes A)$. 
\end{proof}

We thus have the ${}^*$-homomorphism $\Delta:A\to M(A\otimes A)$.  It satisfies the following density results 
(so $\Delta$ is ``full'').

\begin{prop}\label{delta_full}
Let $\Delta:A\to M(A\otimes A)$ be as defined above.  Then the following subspaces are norm-dense in $A$:
$$\operatorname{span}\bigl\{(\theta\otimes\operatorname{id})((a\otimes1)(\Delta b)):\theta\in A^*,a,b\in A\bigr\},
\ \ 
\operatorname{span}\bigl\{(\operatorname{id}\otimes\theta)((\Delta a)(1\otimes b)):\theta\in A^*,a,b\in A\bigr\},$$
$$\operatorname{span}\bigl\{(\theta\otimes\operatorname{id})((\Delta b)(a\otimes1)):\theta\in A^*,a,b\in A\bigr\},
\ \ 
\operatorname{span}\bigl\{(\operatorname{id}\otimes\theta)((1\otimes b)(\Delta a)):\theta\in A^*,a,b\in A\bigr\}.$$
\end{prop}

\begin{proof}
From Proposition~\ref{delta_multiplier}, we know $(a\otimes 1)(\Delta b)\in A\otimes A$, for all $a,b\in A$. So for $\theta\in A^*$, we have 
$(\theta\otimes\operatorname{id})((a\otimes1)(\Delta b))\in A$. We wish to show that these elements are norm-dense in $A$. As we know 
$A$ acts non-degenerately on ${\mathcal H}$, it is all right to take an arbitrary $\theta\in{\mathcal B}({\mathcal H})_*$. 

Without loss of generality, we may consider $b=(\operatorname{id}\otimes\omega)(W)\in A$, for $\omega\in{\mathcal B}({\mathcal H})_*$. 
Then by definition of $\Delta$ and by using Equation~\eqref{(mpi2)}, it becomes
\begin{align}
(\theta\otimes\operatorname{id})((a\otimes1)(\Delta b))&=(\theta\otimes\operatorname{id}\otimes\omega)\bigl((a\otimes1\otimes1)W_{12}^*W_{23}W_{12}\bigr)
\notag \\
&=(\theta\otimes\operatorname{id}\otimes\omega)\bigl((a\otimes1\otimes1)W_{13}W_{23}\bigr)=(\operatorname{id}\otimes\rho)(W),
\notag
\end{align}
where $\rho\in{\mathcal B}({\mathcal H})_*$ is such that $\rho(X)=(\theta\otimes\omega)\bigl((a\otimes1)W(1\otimes X)\bigr)$. We just need to show that 
the functionals of the form $\rho$ are dense in ${\mathcal B}({\mathcal H})_*$.

Suppose $X\in{\mathcal B}({\mathcal H})$ is an arbitrary nonzero operator. Find $u\ne0$ such that $Xu\ne0$. Since $W$ is full (Definition~\ref{fullness_condition}), 
we can find $\xi\in{\mathcal H}$ such that $W(\xi\otimes Xu)\ne0$. Find also $\eta,v\in{\mathcal H}$ such that $\bigl\langle W(\xi\otimes Xu),\eta\otimes v\bigr\rangle\ne0$. 
Here, though, knowing that $A(=A^*)$ acts on ${\mathcal H}$ nondegenerately, we may take $\eta=a^*\tilde{\eta}$, for $\tilde{\eta}\in{\mathcal H}$, $a\in A$. The above 
expression then becomes
$$
\bigl\langle (a\otimes1)W(\xi\otimes Xu),\tilde{\eta}\otimes v\bigr\rangle=\bigl\langle W(\xi\otimes Xu),a^*\tilde{\eta}\otimes v\bigr\rangle\ne0,
$$
which is equivalent to saying $\rho(X)\ne0$, where $\rho(X)=(\omega_{\xi,\tilde{\eta}}\otimes\omega_{u,v})\bigl((a\otimes1)W(1\otimes X)\bigr)$. 
In this way, we have shown that the functionals of the form $\rho$ are dense in ${\mathcal B}({\mathcal H})_*$, and that means elements of the form 
$(\theta\otimes\operatorname{id})((a\otimes1)(\Delta b))=(\operatorname{id}\otimes\rho)(W)$ are dense in $A$.

The proofs for the other three linear spaces being norm dense in $A$ can be done similarly.
\end{proof}

The following theorem clarifies the nature of the comultiplication map $\Delta:A\to M(A\otimes A)$.

\begin{theorem}\label{DeltaonA}
Consider the restriction of the map $\Delta:{\mathcal B}({\mathcal H})\to{\mathcal B}({\mathcal H}\otimes{\mathcal H})$ 
to the subalgebra $A$:
$$
\Delta(x)=W^*(1\otimes x)W, \qquad {\text { for $x\in A$.}}
$$
\begin{enumerate}
\item This determines a ${}^*$-homomorphism $\Delta:A\to M(A\otimes A)$.  
\item The comultiplication is ``full'', in the sense that the density results of Proposition~\ref{delta_full} hold.
\item We have: $\overline{\Delta(A)(A\otimes A)}^{\|\ \|}=E(A\otimes A)$, and $\overline{(A\otimes A)\Delta(A)}^{\|\ \|}=(A\otimes A)E$.
\item We do not have $\Delta$ nondegenerate, but it nonetheless extends to a ${}^*$-homomorphism $\Delta:M(A)\to M(A\otimes A)$. 
\item The coassociativity property holds: 
$$
(\Delta\otimes\operatorname{id})(\Delta x)=(\operatorname{id}\otimes\Delta)(\Delta x),\quad{\text { for any $x\in A$.}}
$$
As such, we will refer to the map $\Delta$ as the {\em comultiplication\/} on $A$.
\end{enumerate}
\end{theorem}

\begin{proof}
(1). Proposition~\ref{Delta_hom} showed that $\Delta$ is a ${}^*$-homomorphism, and Proposition~\ref{delta_multiplier} 
and its Corollary showed that $\Delta(A)\subseteq M(A\otimes A)$.

(2). This is Proposition~\ref{delta_full}.

(3). For any $a\in A$, we saw that $\Delta a\in M(A\otimes A)$.  So for $b,c\in A$, we can approximate $(\Delta a)(b\otimes c)$ 
by the elements of the form $a_1\otimes a_2$, where $a_1,a_2\in A$.  At the same time, note that $\Delta a=W^*(1\otimes a)W
=W^*WW^*(1\otimes a)W=E(\Delta a)$.  This means that any $(\Delta a)(b\otimes c)=E(\Delta a)(b\otimes c)$ can be approximated 
by the elements of the form $E(a_1\otimes a_2)$.  Therefore, we have: $\overline{\Delta(A)(A\otimes A)}^{\|\ \|}\subseteq E(A\otimes A)$. 
As $E$ is a projection, we have that $E(A\otimes A)$ is already norm-closed. 

For the other inclusion, let $(a_i)$ be an approximate unit in the $C^*$-algebra $A$.  We saw in Proposition~\ref{Estrictlimit} that 
$\Delta(a_i)\xrightarrow{\text{ strictly }}\Delta(1)=E$. So for any $b,c\in A$, we have the norm convergence, 
$\Delta(a_i)(b\otimes c)\xrightarrow{\text{ norm }}E(b\otimes c)$. 
In this way, we show that $E(A\otimes A)\subseteq \overline{\Delta(A)(A\otimes A)}^{\|\ \|}$.
The two opposite inclusions mean that we have $\overline{\Delta(A)(A\otimes A)}^{\|\ \|}=E(A\otimes A)$.

The proof for $\overline{(A\otimes A)\Delta(A)}^{\|\ \|}=(A\otimes A)E$ is similar, as we have $\Delta a=(\Delta a)E$, $a\in A$.

(4). As we noted in Proposition~\ref{Delta(1)}, we have: $\Delta(1_{M(A)})=E$.  Since $E$ is only a projection ($E\ne1\otimes1$), 
we do not have the nondegeneracy for $\Delta$. 

Nevertheless, it is possible to naturally extend $\Delta$ to the level of $M(A)$.  See Proposition~3.3 
of \cite{BJKVD_qgroupoid1}.  The results (1), (2), (3) above provide the necessary conditions for the proposition to apply. 
The resulting extension map, $\Delta:M(A)\to M(A\otimes A)$, is a ${}^*$-homomorphism that coincides with the extended 
$\Delta$ map observed in Proposition~\ref{Delta_hom}.

(5). The coassociativity of $\Delta$ has been already shown in Proposition~\ref{deltadeltahat}.
\end{proof}

\begin{rem}
An important point to be made from Theorem~\ref{DeltaonA} is that unless $W$ is unitary, the comultiplication map $\Delta$ is no longer non-degenerate. 
This is saying that the structure $(A,\Delta)$ we have is essentially a quantum groupoid  (assuming the rest of the structure maps are constructed) 
not a quantum group. Even when $A$ is unital or even when $A$ is finite-dimensional, this still remains the case, leading to a compact quantum groupoid 
or a finite quantum groupoid  \cite{BNSwha1}, \cite{BSzMultIso}, \cite{Valfqg}, \cite{NVfqg}.
\end{rem}

Next, replace $W$ with $\widehat{W}=\Sigma W^*\Sigma$, which is also a manageable multiplicative partial isometry. 
We noted earlier that $\overline{\bigl\{(\operatorname{id}\otimes\omega)(\widehat{W}):\omega\in{\mathcal B}({\mathcal H})_*\bigr\}}^{\|\ \|}
=\widehat{A}^*=\widehat{A}$.  As such, the results obtained in the earlier part of this section for $(A,\Delta)$ will 
all have corresponding results, with the role of the canonical idempotent being played by $\widehat{E}=\Sigma WW^*\Sigma$. 
The main results are summarized in the following Theorem~\ref{DeltaonAhat}, clarifying the coalgebra structure on $\widehat{A}$. 

\begin{theorem}\label{DeltaonAhat}
\begin{enumerate}
 \item Write $\widehat{E}=\widehat{W}^*\widehat{W}=\Sigma WW^*\Sigma$.  We have: $\widehat{E}\in M(\widehat{A}\otimes\widehat{A})$.
 \item $\widehat{E}=\Sigma WW^*\Sigma=\widehat{\Delta}(1_{M(\widehat{A})})=\widehat{\Delta}(1)$.
 \item $(\widehat{E}\otimes1)(1\otimes\widehat{E})=(1\otimes\widehat{E})(\widehat{E}\otimes1)$.
 \item The restriction of $\widehat{\Delta}$ to $\widehat{A}$ determines a ${}^*$-homomorphism $\widehat{\Delta}:\widehat{A}
\to M(\widehat{A}\otimes\widehat{A})$.  Namely,
$$ 
\widehat{\Delta}(y)=\Sigma W(y\otimes1)W^*\Sigma,\quad y\in\widehat{A}. 
$$ 
It is ``full'', in the sense that it satisfies the following subspaces are norm-dense in $\widehat{A}$:
$$
\operatorname{span}\bigl\{(\theta\otimes\operatorname{id})((c\otimes1)(\widehat{\Delta} d)):\theta\in \widehat{A}^*,c,d\in \widehat{A}\bigr\}, 
\ \ 
\operatorname{span}\bigl\{(\operatorname{id}\otimes\theta)((\widehat{\Delta} c)(1\otimes d)):\theta\in \widehat{A}^*,c,d\in \widehat{A}\bigr\},
$$
$$
\operatorname{span}\bigl\{(\theta\otimes\operatorname{id})((\widehat{\Delta} d)(c\otimes1)):\theta\in \widehat{A}^*,c,d\in \widehat{A}\bigr\},
\ \ 
\operatorname{span}\bigl\{(\operatorname{id}\otimes\theta)((1\otimes d)(\widehat{\Delta} c)):\theta\in \widehat{A}^*,c,d\in \widehat{A}\bigr\}.
$$
 \item We have: $\overline{\widehat{\Delta}(\widehat{A})(\widehat{A}\otimes\widehat{A})}^{\|\ \|}=\widehat{E}(\widehat{A}\otimes\widehat{A})$, 
and $\overline{(\widehat{A}\otimes\widehat{A})\widehat{\Delta}(\widehat{A})}^{\|\ \|}=(\widehat{A}\otimes\widehat{A})\widehat{E}$.
 \item $\widehat{\Delta}$ extends to a ${}^*$-homomorphism $\widehat{\Delta}:M(\widehat{A})\to M(\widehat{A}\otimes\widehat{A})$, 
and the coassociativity property holds:
$$
(\widehat{\Delta}\otimes\operatorname{id})\widehat{\Delta}(y)=(\operatorname{id}\otimes\widehat{\Delta})\widehat{\Delta}(y),\quad{\text { for any $y\in\widehat{A}$.}}
$$
\end{enumerate}
\end{theorem}

\section{Antipode}\label{sec4}

Let us turn our attention next to the construction of the antipode map. Recall that in the general theory of locally compact quantum groups (\cite{KuVa}, \cite{KuVavN}) 
or $C^*$-algebraic quantum groupoid of separable type (\cite{BJKVD_qgroupoid1}, \cite{BJKVD_qgroupoid2}), the construction of the antipode was done by requiring 
the existence of certain left and right invariant weights $\varphi$ and $\psi$. In this paper, we do not plan to consider such invariant weights.  Instead, let us point out 
that both in the quantum group theory and in the theory of $C^*$-algebraic quantum groupoid of separable type, it was noted that while the construction of the antipode 
map $S$ fundamentally involves the invariant weights, once it is constructed it can be shown that $S$ does not depend on the specific choice of the weights: See 
the Remark following Theorem~5.12 in \cite{BJKVD_qgroupoid2}. 

In a sense, among the primary roles of assuming the existence of the invariant weights are to construct the multiplicative partial isometries and the antipode map. 
In our case, the roles of the invariant weights become lessened because we started out with the multiplicative partial isometry $W$, and we will make a case 
below that having the manageability condition can help us gain insights on the construction of the antipode. 

Based on these perspectives, let us go ahead with the construction of the antipode map, as well as some relevant results. 

\subsection{Unitary antipode}\label{sub4.1}

In \S\ref{sec2} (see Proposition~\ref{Rhat_pre}), we considered a linear map $\widehat{R}:a\to a^{\widehat{R}}$ on the subalgebra $\widehat{\mathcal A}$, 
satisfying
\begin{equation}\label{(Rhat1)}
Y(1\otimes a^{\widehat{R}})Y^*=Z(a^{\top}\otimes1)Z^*,
\end{equation}
where $Y=\Sigma\widetilde{W}^{\top\otimes\top}[\widetilde{W}^*]^{\top\otimes\top}\Sigma$ and $Z=\Sigma\widetilde{W}^{\top\otimes\top}\Sigma\widetilde{W}$. 
(At that time, the map was defined on $\widehat{\mathcal A}\cup\widehat{\mathcal A}^*$, but we now know that $\widehat{\mathcal A}$ is self-adjoint.) See below 
that $\widehat{R}$ can be further extended to the $C^*$-algebra level:

\begin{prop}\label{Rhat}
The map $\widehat{R}$ extends to a ${}^*$-anti-isomomorphism $\widehat{R}_{\widehat{A}}:\widehat{A}\to\widehat{A}$ at the $C^*$-algebra level, 
satisfying $\widehat{R}_{\widehat{A}}\circ\widehat{R}_{\widehat{A}}=\operatorname{Id}$. It can be characterized by 
\begin{equation}\label{(Rhat2)}
\widehat{R}_{\widehat{A}}:(\omega\otimes\operatorname{id})(W^*)\mapsto(\omega^{\top}\otimes\operatorname{id})(\widetilde{W})
=\bigl(\omega(Q\,\cdot\,Q^{-1})\otimes\operatorname{id}\bigr)(W),
\ {\text { for $\omega\in{\mathcal B}({\mathcal H})_*$.}}
\end{equation}
\end{prop}

\begin{proof}
We saw in Proposition~\ref{Rhat_pre} that $\widehat{R}$ is a ${}^*$-map, satisfying $\widehat{R}\circ\widehat{R}=\operatorname{Id}$ and is anti-multiplicative. 
We also saw in Equation~\eqref{(Rhat_pre3)} an alternative characterization for the map: 
$$
\widehat{R}:(\omega\otimes\operatorname{id})(W^*)\mapsto(\omega^{\top}\otimes\operatorname{id})(\widetilde{W})
=\bigl(\omega(Q\,\cdot\,Q^{-1})\otimes\operatorname{id}\bigr)(W),
\ {\text { for $\omega\in{\mathcal B}({\mathcal H})_*$.}}
$$
Therefore $\widehat{R}$ is a ${}^*$-anti-homomorphism from $\widehat{\mathcal A}$ into itself, densely-defined with a dense range.  As such, 
we can naturally extend $\widehat{R}$ to the level of the $C^*$-algebra $\widehat{A}=\overline{\widehat{\mathcal A}}^{\|\ \|}$. Having 
$\widehat{R}\circ\widehat{R}=\operatorname{Id}$ also implies the injectivity. The extended map $\widehat{R}_{\widehat{A}}:\widehat{A}\to\widehat{A}$ 
becomes a ${}^*$-anti-isomorphism on $A$.
\end{proof}

\begin{defn}
We have a ${}^*$-anti-isomorphism $\widehat{R}_{\widehat{A}}:\widehat{A}\to\widehat{A}$, uniquely characterized by Equation~\eqref{(Rhat1)}, and 
also by Equation~\eqref{(Rhat2)}. We will call this the {\em unitary antipode\/} for $\widehat{A}$.
\end{defn}

We also noted during the proof of Proposition~\ref{Wtildepartialiso} that at least at the dense subspace level we may regard 
$\widetilde{W}=[W^*]^{\top\otimes\widehat{R}}$.  We can now give a proper proof:

\begin{prop}\label{Wtilde=}
$\widetilde{W}=[W^*]^{\top\otimes\widehat{R}}$.
\end{prop}

\begin{proof}
Recall from the proof of Proposition~\ref{Rhat_pre} that for any $\omega\in{\mathcal B}({\mathcal H})_*$ we have 
\begin{equation}\label{(Rhat3)}
\bigl[(\omega\otimes\operatorname{id})(W^*)\bigr]^{\widehat{R}}=(\omega^{\top}\otimes\operatorname{id})(\widetilde{W}).
\end{equation}
We now know that $\widehat{R}$ is fully extended to the $C^*$-algebra level, so Equation~\eqref{(Rhat3)} allows us to make sense of the expression 
$[W^*]^{\top\otimes\widehat{R}}$ as a bounded operator acting on all of $\overline{\mathcal H}\otimes{\mathcal H}$.  

To see that it is equal to $\widetilde{W}$, let $\theta\in{\mathcal B}({\mathcal H})_*$ be arbitrary and define $\theta^{\widehat{R}}\in{\mathcal B}({\mathcal H})_*$ 
by $\theta^{\widehat{R}}(a)=\theta(a^{\widehat{R}})$, for $a\in A$. Or $\theta^{\widehat{R}}(a^{\widehat{R}})=\theta(a)$, as $\widehat{R}\circ\widehat{R}
=\operatorname{Id}$.  Since $A$ is nondegenerately represented on ${\mathcal H}$, this is sufficient to characterize $\theta^{\widehat{R}}\in{\mathcal B}({\mathcal H})_*$.  
By applying $\theta^{\widehat{R}}$ to Equation~\eqref{(Rhat3)}, we have:
$$
(\omega\otimes\theta)(W^*)=(\omega^{\top}\otimes\theta^{\widehat{R}})(\widetilde{W})=(\omega\otimes\theta)(\widetilde{W}^{\top\otimes\widehat{R}}).
$$ 
Since $\omega,\theta\in{\mathcal B}({\mathcal H})_*$ are arbitrary, this proves that $\widetilde{W}^{\top\otimes\widehat{R}}=W^*$, or equivalently 
$\widetilde{W}=[W^*]^{\top\otimes\widehat{R}}$.
\end{proof}

By working with $\widehat{W}=\Sigma W^*\Sigma$ and $\widetilde{\widehat{W}}=(\Sigma\widetilde{W}^*\Sigma)^{\top\otimes\top}$, we can also consider the 
unitary antipode for $A$, which will be also a ${}^*$-anti-isomorphism $R_A:A\to A$. See below:

\begin{prop}\label{unitaryantipode}
There exists a ${}^*$-anti-isomorphism $R_A:A\to A$, satisfying $R_A\circ R_A=\operatorname{Id}$, characterized by
$$
R_A:(\operatorname{id}\otimes\omega)(W)\mapsto(\operatorname{id}\otimes\omega)(\widetilde{W}^*)^{\top}
=\bigl(\operatorname{id}\otimes\omega(Q\,\cdot\,Q^{-1})\bigr)(W^*),
\ {\text { for $\omega\in{\mathcal B}({\mathcal H})_*$.}}
$$
The map $R_A$ is referred to as the {\em unitary antipode\/} for $A$.
\end{prop}

\begin{proof}
Replace $W$, $\widetilde{W}$, $Q$ with $\widehat{W}=\Sigma W^*\Sigma$, $\widetilde{\widehat{W}}=(\Sigma\widetilde{W}^*\Sigma)^{\top\otimes\top}$, $Q$. 
In this way, we obtain a ${}^*$-anti-isomorphism $R_A:A\to A$, satisfying $R_A\circ R_A=\operatorname{Id}$. More specifically, we have the map $R:{\mathcal A}
\to{\mathcal A}$ at the ${}^*$-algebra level satisfying 
\begin{align}
R\bigl((\operatorname{id}\otimes\omega)(W)\bigr)&=R\bigl((\omega\otimes\operatorname{id})(\widehat{W}^*)\bigr)
=(\omega^{\top}\otimes\operatorname{id})\bigl((\Sigma\widetilde{W}^*\Sigma)^{\top\otimes\top}\bigr)  \notag \\
&=(\operatorname{id}\otimes\omega)(\widetilde{W}^*)^{\top}=\bigl(\operatorname{id}\otimes\omega(Q\,\cdot\,Q^{-1})\bigr)(W^*), 
\ {\text { for $\omega\in{\mathcal B}({\mathcal H})_*$,}}
\notag
\end{align}
which extends to the map at the $C^*$-algebra level, $R_A:A\to A$. The map may be also characterized in a similar way as in  Equation~\eqref{(Rhat1)}, 
using instead $\widetilde{Y}=\widetilde{W}^*\widetilde{W}$ and $\widetilde{Z}=\widetilde{W}^*\Sigma[\widetilde{W}^*]^{\top\otimes\top}\Sigma$.
\end{proof}

\begin{cor}
There exist additional relations between the operators $\widetilde{W}$ and $W$:
\begin{itemize}
\item $[\widetilde{W}^*]^{\top\otimes\top}=W^{R\otimes\top}$
\item $W^{R\otimes\widehat{R}}=W$
\end{itemize}
\end{cor}

\begin{proof}
In Proposition~\ref{Wtilde=}, we noted the characterization $\widetilde{W}=[W^*]^{\top\otimes\widehat{R}}$. By considering instead 
$\widehat{W}=\Sigma W^*\Sigma$ and $\widetilde{\widehat{W}}=(\Sigma\widetilde{W}^*\Sigma)^{\top\otimes\top}$, we have 
$\widetilde{\widehat{W}}=[\widehat{W}^*]^{\top\otimes R}$, which is equivalent to saying 
$$
[\widetilde{W}^*]^{\top\otimes\top}=W^{R\otimes\top}.
$$
Since $\widetilde{W}=[W^*]^{\top\otimes\widehat{R}}$, we have 
$W^{\operatorname{id}\otimes(\widehat{R}\circ\top)}=W^{R\otimes\top}$. Applying $R\otimes\top$ to both sides, this becomes 
$$
W^{R\otimes\widehat{R}}=W.
$$
\end{proof}

Finally, in analogy with the quantum group case, we can show that the unitary antipode flips the comultiplication:

\begin{prop}\label{comultiplicationR}
For $a\in A$, we have: 
$$
\Delta\bigl(R_A(a)\bigr)=(R_A\otimes R_A)\bigl(\Delta^{\operatorname{cop}}(a)\bigr).
$$
\end{prop}

\begin{proof}
Let $\omega\in{\mathcal B}({\mathcal H})_*$ be arbitrary, and consider $a=(\operatorname{id}\otimes\omega)(W)\in A$. Then by 
Proposition~\ref{unitaryantipode}, we have $R_A(a)=(\operatorname{id}\otimes\omega)(\widetilde{W}^*)^{\top}=(\operatorname{id}\otimes\omega^{\top})
([\widetilde{W}^*]^{\top\otimes\top})$. Thus we have: 
\begin{equation}\label{(comultiplicationR1)}
\Delta\bigl(R_A(a)\bigr)=W^*(1\otimes a)W=(\operatorname{id}\otimes\operatorname{id}\otimes\omega^{\top})\bigl(W_{12}^*[\widetilde{W}_{23}^*]^{\top\otimes\top}W_{12}\bigr).
\end{equation}
Apply here the property (1) of Proposition~\ref{hashcomposable}, namely $W_{12}^{\top\otimes\top}\widetilde{W}_{23}[W_{12}^*]^{\top\otimes\top}=\widetilde{W}_{13}\widetilde{W}_{23}$, 
which is also equivalent to $W_{12}^*[\widetilde{W}_{23}^*]^{\top\otimes\top}W_{12}=[\widetilde{W}_{13}^*]^{\top\otimes\top}[\widetilde{W}_{23}^*]^{\top\otimes\top}$ by applying 
$[(\,\cdot\,)^*]^{\top\otimes\top}$. Then Equation~\eqref{(comultiplicationR1)} becomes:
\begin{align}
\Delta\bigl(R_A(a)\bigr)&=(\operatorname{id}\otimes\operatorname{id}\otimes\omega^{\top})\bigl([\widetilde{W}_{13}^*]^{\top\otimes\top}[\widetilde{W}_{23}^*]^{\top\otimes\top}\bigr)
=(\operatorname{id}\otimes\operatorname{id}\otimes\omega^{\top})\bigl(W_{13}^{R\otimes\top}W_{23}^{R\otimes\top}\bigr),
\notag
\end{align}
by using $[\widetilde{W}^*]^{\top\otimes\top}=W^{R\otimes\top}$ (Corollary of Proposition~\ref{unitaryantipode}). The right side of the above equation can be expressed as 
$$
\dots=(\operatorname{id}\otimes\operatorname{id}\otimes\omega^{\top})\bigl([W_{23}W_{13}]^{R\otimes R\otimes\top}\bigr)
=(R_A\otimes R_A)\bigl((\operatorname{id}\otimes\operatorname{id}\otimes\omega)(W_{23}W_{13})\bigr).
$$

Meanwhile, for $a=(\operatorname{id}\otimes\omega)(W)$, we have
$$
\Delta a=W^*(1\otimes a)W=(\operatorname{id}\otimes\operatorname{id}\otimes\omega)(W_{12}^*W_{23}W_{12})
=(\operatorname{id}\otimes\operatorname{id}\otimes\omega)(W_{13}W_{23}),
$$
by Equation~\eqref{(mpi2)}. So the coopposite comultiplication, $\Delta^{\operatorname{cop}}(a)$ is none other than $\Delta^{\operatorname{cop}}(a)
=(\operatorname{id}\otimes\operatorname{id}\otimes\omega)(W_{23}W_{13})$. Combining these observations, we indeed see the following:
$$
\Delta\bigl(R_A(a)\bigr)=(R_A\otimes R_A)\bigl((\operatorname{id}\otimes\operatorname{id}\otimes\omega)(W_{23}W_{13})\bigr)
=(R_A\otimes R_A)\bigl(\Delta^{\operatorname{cop}}(a)\bigr).
$$
\end{proof}
An analogous result holds for $(\widehat{A},\widehat{\Delta})$, the unitary antipode map $\widehat{R}_{\widehat{A}}$ flipping the comultiplication.

\subsection{Scaling group}\label{sub4.2}

\begin{defn}\label{tau}
Let $W$, $\widetilde{W}$, $Q$ be as before (see Definition~\ref{manageable}). 
Write $\tau_t(a):=Q^{2it}aQ^{-2it}$, for $a\in A$, $t\in\mathbb{R}$. Then $(\tau_t)_{t\in\mathbb{R}}$ determines a one-parameter group of automorphisms of $A$. 
This will be referred to as the {\em scaling group\/}.
\end{defn}

While the definition is valid, we need to give some additional results on $(\tau_t)_{t\in\mathbb{R}}$:

\begin{prop}\label{tau_prop}
\begin{enumerate}
\item As defined above, $(\tau_t)_{t\in\mathbb{R}}$ is a norm-continuous one-parameter group of automorphisms on $A$. 
\item Let $\tau_{-\frac{i}2}$ be the analytic generator of $(\tau_t)_{t\in\mathbb{R}}$, at $z=-\frac{i}2$. The dense subalgebra 
${\mathcal A}=\bigl\{(\operatorname{id}\otimes\omega)(W):\omega\in{\mathcal B}({\mathcal H})_*\bigr\}$ forms a core for $\tau_{-\frac{i}2}$, and we have
$$
\tau_{-\frac{i}2}\bigl((\operatorname{id}\otimes\omega)(W)\bigr)=(\operatorname{id}\otimes\omega)(\widetilde{W})^{\top}.
$$
\end{enumerate}
\end{prop}

\begin{proof}
(1).
Consider $(\operatorname{id}\otimes\omega)(W)\in{\mathcal A}$, for $\omega\in{\mathcal B}({\mathcal H})_*$, $t\in\mathbb{R}$. 
By Proposition~\ref{WQQequality}, we know that $W=(Q^{-2it}\otimes Q^{-2it})W(Q^{2it}\otimes Q^{2it})$, for $t\in\mathbb{R}$.  So we have:
\begin{align}
\tau_t\bigl((\operatorname{id}\otimes\omega)(W)\bigr)&=Q^{2it}\bigl[(\operatorname{id}\otimes\omega)(W)\bigr]Q^{-2it}  \notag \\
&=Q^{2it}(\operatorname{id}\otimes\omega)\bigl((Q^{-2it}\otimes Q^{-2it})W(Q^{2it}\otimes Q^{2it})\bigr)Q^{-2it}
=(\operatorname{id}\otimes\omega_t)(W),
\label{(tau_eqn1)}
\end{align}
where $\omega_t\in{\mathcal B}({\mathcal H})_*$ is such that $\omega_t(\,\cdot\,)=\omega(Q^{-2it}\,\cdot\,Q^{2it})$.  It is clear 
$\|\omega_t-\omega\|_{{\mathcal B}({\mathcal H})_*}\longrightarrow0$, as $t\to0$.

From Equation~\eqref{(tau_eqn1)}, we observe that $\tau_t(a)\in{\mathcal A}$ for any $a\in{\mathcal A}$. In fact, as $\omega(Q^{-2it}\,\cdot\,Q^{2it})$ 
is dense in ${\mathcal B}({\mathcal H})_*$ for any $t\in\mathbb{R}$, we actually have $\tau_t({\mathcal A})={\mathcal A}$, for all $t\in\mathbb{R}$.
We note that $\tau_t(a)$ is a norm-continuous function on $t$.  In this way, we have a one-parameter group of automorphisms 
$(\tau_t)_{t\in\mathbb{R}}$ of $A$.

(2).
Let $\omega\in{\mathcal B}({\mathcal H})_*$.  Without loss of generality, we can take $\omega=\omega_{u,v}$, where $u\in{\mathcal D}(Q)$, 
$v\in{\mathcal D}(Q^{-1})$.  Then $\omega_t(\,\cdot\,)=\omega_{u,v}(Q^{-2it}\,\cdot\,Q^{2it})=\omega_{Q^{2it}u,Q^{-2it}v}$. By analytic continuation, 
we have: $\omega_{-\frac{i}2}=\omega_{Qu,Q^{-1}v}$. It follows from Equation~\eqref{(tau_eqn1)} that
$$
\tau_{-\frac{i}2}\bigl((\operatorname{id}\otimes\omega)(W)\bigr)
=(\operatorname{id}\otimes\omega_{Qu,Q^{-1}v})(W)=(\operatorname{id}\otimes\omega_{u,v})(\widetilde{W})^{\top},
$$
by Lemma~\ref{lem_QT}\,(1). In particular, note that $(\operatorname{id}\otimes\omega_{u,v})(W)\in{\mathcal D}(\tau_{-\frac{i}2})$. 
As the analytic generator $\tau_{-\frac{i}2}$ is a closed operator, this means that $(\operatorname{id}\otimes\omega)(W)\in{\mathcal D}(\tau_{-\frac{i}2})$ 
for any $\omega\in{\mathcal B}({\mathcal H})_*$. The computation also shows that $(\operatorname{id}\otimes\omega)(\widetilde{W})^{\top}\in{\mathcal A}$, 
for any $\omega\in{\mathcal B}({\mathcal H})_*$.
\end{proof}

Here is a result that relates the comultiplication with $\tau$:

\begin{prop}\label{Deltatau}
We have:
$\Delta\circ\tau_t=(\tau_t\otimes\tau_t)\circ\Delta$, for any $t\in\mathbb{R}$.
\end{prop}

\begin{proof}
For any $a\in A$, we have:
\begin{align}
\Delta\bigl(\tau_t(a)\bigr)&=W^*(1\otimes Q^{2it}aQ^{-2it})W
=W^*(Q^{2it}\otimes Q^{2it})(1\otimes a)(Q^{-2it}\otimes Q^{-2it})W   \notag \\
&=(Q^{2it}\otimes Q^{2it})W^*(1\otimes a)W(Q^{-2it}\otimes Q^{-2it})=(\tau_t\otimes\tau_t)(\Delta a),
\notag
\end{align}
because $W(Q^{-2it}\otimes Q^{-2it})=(Q^{-2it}\otimes Q^{-2it})W$ (see Proposition~\ref{WQQequality}).
\end{proof}

Results analogous to Propositions~\ref{tau_prop} and \ref{Deltatau} hold true for the case of $(\widehat{A},\widehat{\Delta})$, in which case 
the scaling group $(\hat{\tau}_t)$ is defined by $\hat{\tau}_t(a)=Q^{2it}aQ^{-2it}$, for $a\in\widehat{A}$.

\subsection{Antipode}\label{sub4.3}

\begin{theorem}\label{antipode}
There exists a closed linear map $S$ on $A$, defined in terms of a polar decomposition as follows:
$$
S=R_A\circ\tau_{-\frac{i}2}=\tau_{-\frac{i}2}\circ R_A,
$$
where $R_A$ is the involutive ${}^*$-anti-automorphism $R_A:A\to A$, the ``unitary antipode'', that appeared in \S\ref{sub4.1}, and  $\tau_{-\frac{i}2}$ is 
the analytic generator at $z=-\frac{i}{2}$ for the automorphism group $(\tau_t)$, the ``scaling group'' that was introduced in \S\ref{sub4.2}.

The map $S$ will be called the ``antipode''. The dense subalgebra ${\mathcal A}=\bigl\{(\operatorname{id}\otimes\omega)(W):\omega\in{\mathcal B}({\mathcal H})_*\bigr\}$ 
forms a core for $S$, and there exists an alternative characterization for the antipode map:
\begin{equation}\label{(antipode_characterization)}
S\bigl((\operatorname{id}\otimes\omega)(W)\bigr)=(\operatorname{id}\otimes\omega)(W^*),\ {\text { for $\omega\in{\mathcal B}({\mathcal H})_*$.}}
\end{equation}
We have $S(ab)=S(b)S(a)$, for $a,b\in{\mathcal D}(S)$, and we have: $S\bigl(S(x)^*\bigr)^*=x$ for any $x\in{\mathcal D}(S)$.
\end{theorem}

\begin{proof}
Define the map $S$, by 
$$
S:=R_A\circ\tau_{-\frac{i}2}.
$$
As $\tau_{-\frac{i}2}$ is a closed densely-defined map having ${\mathcal A}$ as a core (see above), so is $S$.  As $R_A$ is anti-multiplicative, 
so is $S$.  Under this map $S$, for $\omega=\omega_{u,v}$, $u\in{\mathcal D}(Q),v\in{\mathcal D}(Q^{-1})$, we have:
$$
S\bigl((\operatorname{id}\otimes\omega_{u,v})(W)\bigr)=R_A\bigl(\tau_{-\frac{i}2}((\operatorname{id}\otimes\omega_{u,v})(W))\bigr)
=R_A\bigl((\operatorname{id}\otimes\omega_{Qu,Q^{-1}v})(W)\bigr)=(\operatorname{id}\otimes\omega_{u,v})(W^*),
$$
by Propositions~\ref{tau_prop} and \ref{unitaryantipode}.
This gives an alternative characterization of $S$.  It is easy to observe that $R_A\circ\tau_{-\frac{i}2}=\tau_{-\frac{i}2}\circ R_A$. 

Finally, we can also see that for any $a=(\operatorname{id}\otimes\omega)(W)\in{\mathcal A}$, we have, 
\begin{align}
S\bigl(S(a)^*\bigr)^*&=S\bigl(S((\operatorname{id}\otimes\omega)(W))^*\bigr)^*
=S\bigl([(\operatorname{id}\otimes\omega)(W^*)]^*\bigr)^*  \notag \\
&=S\bigl((\operatorname{id}\otimes\overline{\omega})(W)\bigr)^*
=\bigl[(\operatorname{id}\otimes\overline{\omega})(W^*)\bigr]^*=(\operatorname{id}\otimes\omega)(W)=a.
\notag
\end{align}
\end{proof}

\begin{rem}
This construction of the antipode map is different from the way that was done in \cite{BJKVD_qgroupoid2}, which used the invariant weights. 
For instance, the $Q$ (or rather $Q^2$) operator that is being used here to define the scaling group is different from the $L$ operator used 
in that paper.  On the other hand, the characterization of $S$ given in Equation~\eqref{(antipode_characterization)} of Theorem~\ref{antipode} is exactly 
same as the one obtained in Proposition~4.27 of \cite{BJKVD_qgroupoid2}.  Moreover, from $S^2=\tau_{-i}$, we can see that the analytic generators of 
the scaling groups for the two formulations are same, meaning that the scaling groups coincide, so also the unitary antipode maps.  This means 
that $S$, $R_A$, $(\tau_t)$ are exactly same for the two formulations, even though the approaches to arriving at them were different.
\end{rem}

\section{The base algebras}\label{sec5}

\subsection{The $C^*$-subalgebras $B$ and $C$}\label{sub5.1}

Recall from Section~\ref{sec1} the following subspaces in ${\mathcal B}({\mathcal H})$: 
$$
B:={\overline{\operatorname{span}\bigl\{(\operatorname{id}\otimes\omega)(W^*W):\omega
\in{\mathcal B}({\mathcal H})_*\bigr\}}}^{\|\ \|},
\quad 
C:={\overline{\operatorname{span}\bigl\{(\omega\otimes\operatorname{id})(W^*W):\omega
\in{\mathcal B}({\mathcal H})_*\bigr\}}}^{\|\ \|},
$$
$$
\widehat{B}:={\overline{\operatorname{span}\bigl\{(\omega\otimes\operatorname{id})(WW^*):\omega
\in{\mathcal B}({\mathcal H})_*\bigr\}}}^{\|\ \|},
\quad
\widehat{C}:={\overline{\operatorname{span}\bigl\{(\operatorname{id}\otimes\omega)(WW^*):\omega
\in{\mathcal B}({\mathcal H})_*\bigr\}}}^{\|\ \|}.
$$
We expect these to be the source and the target algebras for $(A,\Delta)$ and $(\widehat{A},\widehat{\Delta})$. But first, we should show that 
$B$, $C$, $\widehat{B}$, $\widehat{C}$ are in fact $C^*$-subalgebras of ${\mathcal B}({\mathcal H})$. 

We begin with a lemma, showing that the generators of $B$, $C$, $\widehat{B}$, $\widehat{C}$ behave like multipliers in $M(A)$ or $M(\widehat{A})$. 
In the below, note that $b\in B$, $\hat{b}\in\widehat{B}$, $c\in C$, $\hat{c}\in\widehat{C}$.

\begin{lem} \label{BCinMA}
\begin{enumerate}
 \item Let $b=(\operatorname{id}\otimes\omega)(W^*W)$, where $\omega\in{\mathcal B}({\mathcal H})_*$ is arbitrary. 
Then for any $x\in A$, we have $bx\in A$.
 \item Let $\hat{b}=(\omega\otimes\operatorname{id})(WW^*)$, $\omega\in{\mathcal B}({\mathcal H})_*$.  
Then for any $y\in\widehat{A}$, we have $y\hat{b}\in\widehat{A}$.
 \item Let $c=(\omega\otimes\operatorname{id})(W^*W)$, $\omega\in{\mathcal B}({\mathcal H})_*$. 
Then for any $x\in A$, we have: $xc\in A$.  Also for any $y\in\widehat{A}$, we have: $cy\in\widehat{A}$.
 \item Let $\hat{c}=(\operatorname{id}\otimes\omega)(WW^*)$, $\omega\in{\mathcal B}({\mathcal H})_*$. 
Then for any $x\in A$, we have: $x\hat{c}\in A$.  Also for any $y\in\widehat{A}$, we have: $\hat{c}y\in\widehat{A}$.
\end{enumerate}
\end{lem}

\begin{proof}
(1). Let $x=(\operatorname{id}\otimes\theta)(W)\in A$, for an arbitrary $\theta\in{\mathcal B}({\mathcal H})_*$.  
By Equation~\eqref{(mpi6)}, we have: 
\begin{align}
bx&=(\operatorname{id}\otimes\omega\otimes\theta)(W_{12}^*W_{12}W_{13})
=(\omega\otimes\operatorname{id}\otimes\theta)(W_{13}W_{23}W_{23}^*)  \notag \\
&=(\operatorname{id}\otimes\theta)\bigl(W(1\otimes(\omega\otimes\operatorname{id})(WW^*))\bigr) 
=(\operatorname{id}\otimes\theta)\bigl(W(1\otimes q)\bigr)=(\operatorname{id}\otimes\rho)(W)\in A,
\notag
\end{align}
where $q=(\omega\otimes\operatorname{id})(WW^*)$, and $\rho(\,\cdot\,)=\theta(\,\cdot\,q)\in{\mathcal B}({\mathcal H})_*$.  
This shows that $bx\in A$ for any $x\in A$.

(2). Let $y=(\theta\otimes\operatorname{id})(W)\in\widehat{A}$. Again by using Equation~\eqref{(mpi6)}, we can show that 
$$
y\hat{b}=\dots=(\theta\otimes\operatorname{id})\bigl((p\otimes1)W\bigr)=\bigl(\theta(p\,\cdot\,)\operatorname{id}\bigr)(W)
\in\widehat{A},
$$
where $p=(\operatorname{id}\otimes\omega)(W^*W)$.  Since $\theta\in{\mathcal B}({\mathcal H})_*$ is arbitrary, 
this means $y\hat{b}\in\widehat{A}$, $\forall y\in\widehat{A}$.

(3). Consider $x=(\operatorname{id}\otimes\theta)(W)\in A$, $\theta\in{\mathcal B}({\mathcal H})_*$.  By a similar approach 
as above, but now using Equation~\eqref{(mpi9)}, we can show that 
$$
xc=\dots=(\operatorname{id}\otimes\theta)\bigl((1\otimes c)W\bigr)=\bigl(\operatorname{id}\otimes\theta(c\,\cdot\,)\bigr)(W)\in A.
$$
For $y=(\theta\otimes\operatorname{id})(W)\in\widehat{A}$, by using Equation~\eqref{(mpi10)}, we can show that
$$
cy=\dots=(\theta\otimes\operatorname{id})\bigl(W(c\otimes1)\bigr)=\bigl(\theta(\,\cdot\,c)\otimes\operatorname{id})(W)
\in\widehat{A}.
$$

(4). Similar proof as in (3).
\end{proof}

In the lemma above, we observe that while similar, the elements in $C$ and $\widehat{C}$ behave slightly differently 
than those in $B$ and $\widehat{B}$.  There seems to be a little more of a symmetric behavior going on for the 
elements in $C$ and $\widehat{C}$.  This is no accident, as we can see from the following proposition:

\begin{prop}\label{C=Chat}
We have: $C=\widehat{C}$.
\end{prop}

\begin{proof}
Let $\omega,\theta\in{\mathcal B}({\mathcal H})_*$ be arbitrary.  Write: $y=(\omega\otimes\operatorname{id})(W)\in\widehat{A}$ 
and $x=(\operatorname{id}\otimes\theta)(W)\in A$, and consider $\tilde{\omega}:=\omega(\,\cdot\,x)$ and $\tilde{\theta}:=\theta(y\,\cdot\,)$. 
As $\omega,\theta$ are arbitrary and since $A$ and $\widehat{A}$ act on ${\mathcal H}$ in a non-degenerate way, it is evident that 
the functionals of the form $\tilde{\omega}$ and $\tilde{\theta}$ are dense in ${\mathcal B}({\mathcal H})_*$. 
Observe that
\begin{align}
(\operatorname{id}\otimes\tilde{\theta})(WW^*)&=(\operatorname{id}\otimes\theta)\bigl((1\otimes y)WW^*\bigr)
=(\omega\otimes\operatorname{id}\otimes\theta)(W_{13}W_{23}W_{23}^*)  \notag \\
&=(\omega\otimes\operatorname{id}\otimes\theta)(W^*_{12}W_{12}W_{13})  
=(\omega\otimes\operatorname{id})\bigl(W^*W(x\otimes1)\bigr)
=(\tilde{\omega}\otimes\operatorname{id})(W^*W).
\label{(eqnC=Chat)}
\end{align}
by Equation~\eqref{(mpi6)}. As the elements $(\operatorname{id}\otimes\tilde{\theta})(WW^*)$ are dense in $\widehat{C}$ and the elements 
$(\tilde{\omega}\otimes\operatorname{id})(W^*W)$ are dense in $C$, Equation~\eqref{(eqnC=Chat)} indicates that $\widehat{C}=C$.
\end{proof}

\begin{rem}
There is no such result for $B$ and $\widehat{B}$.  We have $B\ne\widehat{B}$, in general.
\end{rem}

We next turn our attention to proving that $C$ is a subalgebra of ${\mathcal B}({\mathcal H})$.

\begin{prop}\label{Cisalgebra}
$C(=\widehat{C})$ is a subalgebra in ${\mathcal B}({\mathcal H})$. 
\end{prop}

\begin{proof}
Consider arbitrary $\rho,\theta,\omega\in{\mathcal B}({\mathcal H})_*$, and let $x=(\operatorname{id}\otimes\theta)(W)\in A$
and $y=(\omega\otimes\operatorname{id})(W)\in\widehat{A}$.  These elements generate $A$ and $\widehat{A}$, respectively. 
Consider $c=(\rho\otimes\operatorname{id})(W^*W)\in C$, and also consider $c'=\bigl(\omega(\,\cdot x)\otimes\operatorname{id}\bigr)(W^*W)
=(\omega\otimes\operatorname{id})\bigl(W^*W(x\otimes1)\bigr)\in C$. As $A$ acts on ${\mathcal H}$ nondegenerately, the functionals 
$\omega(\,\cdot x)$ are dense in ${\mathcal B}({\mathcal H})_*$. As such, we can see that elements of the form $c'$ (as well as those 
of the form $c$) are dense in $C$.

Meanwhile, by Equation~\eqref{(eqnC=Chat)}, we know that 
$(\omega\otimes\operatorname{id})\bigl(W^*W(x\otimes1)\bigr)=(\operatorname{id}\otimes\theta)\bigl((1\otimes y)WW^*\bigr)$. 
As a result, we have:
\begin{align}
cc'=(\rho\otimes\operatorname{id})(W^*W)(\omega\otimes\operatorname{id})\bigl(W^*W(x\otimes1)\bigr)
&=(\omega\otimes\rho\otimes\operatorname{id})\bigl(W_{23}^*W_{23}W_{13}^*W_{13}(x\otimes1\otimes1)\bigr)  \notag \\
&=(\operatorname{id}\otimes\rho\otimes\theta)\bigl(W_{23}^*W_{23}(1\otimes1\otimes y)W_{13}W_{13}^*\bigr) \notag \\
&=(\operatorname{id}\otimes\theta)\bigl((1\otimes cy)WW^*\bigr)=(\operatorname{id}\otimes\tilde{\theta})(WW^*), \notag
\end{align}
where $\tilde{\theta}=\theta(cy\,\cdot\,)$.  So $cc'\in\widehat{C}$, thus $cc'\in C$, as we know from 
Proposition~\ref{C=Chat} that $C=\widehat{C}$.
\end{proof}

We next wish to show that $B$ and $\widehat{B}$ are also subalgebras in ${\mathcal B}({\mathcal H})$.  But before jumping into proving 
these results, let us first consider the following result, which is a consequence of the condition~(3) of the manageability of $W$ 
(Definition~\ref{manageable}).

\begin{prop}\label{kappa}
Let $b=(\operatorname{id}\otimes\omega)(W^*W)$, where $\omega=\omega_{r,s}\in{\mathcal B}({\mathcal H})_*$, for $r\in{\mathcal D}(Q^{-1})$, 
$s\in{\mathcal D}(Q)$.  Such elements generate $B$.  Write:  
$$
\kappa(b)=Q(\omega^{\top}\otimes\operatorname{id})(\widetilde{W}\widetilde{W}^*)Q^{-1},
$$
where $\omega^{\top}\in{\mathcal B}(\overline{\mathcal H})_*$ is such that $\omega^{\top}(m^{\top})=\omega(m)$, for $m\in{\mathcal B}({\mathcal H})$. 

Then $\kappa:b\mapsto\kappa(b)$ becomes a well-defined linear map from a dense subspace of $B$ into ${\mathcal B}({\mathcal H})$, which can be 
alternatively characterized by 
\begin{equation}\label{(kappa)}
E(b\otimes1)=E\bigl(1\otimes\kappa(b)\bigr).
\end{equation}
\end{prop}

\begin{proof}
From the condition~(3) of Definition~\ref{manageable}, we have: $\widetilde{W}_{13}\widetilde{W}_{23}\widetilde{W}^*_{23}
=W^{\top\otimes\top}_{12}{W^*}^{\top\otimes\top}_{12}\widetilde{W}_{13}$.  Apply here 
$\operatorname{id}\otimes\omega^{\top}\otimes\operatorname{id}$. Then we have: 
$$
\widetilde{W}(1\otimes y)=(x\otimes1)\widetilde{W},
$$  
where $y=(\omega^{\top}\otimes\operatorname{id})(\widetilde{W}\widetilde{W}^*)$ and 
$x=(\operatorname{id}\otimes\omega^{\top})(W^{\top\otimes\top}{W^*}^{\top\otimes\top})$.  Note that 
$$
x^{\top}=\bigl[(\operatorname{id}\otimes\omega^{\top})(W^{\top\otimes\top}{W^*}^{\top\otimes\top})\bigr]^{\top}
=(\operatorname{id}\otimes\omega)(W^*W)=b,
$$
because $m\mapsto m^{\top}$ is an anti-homomorphism. This means $x=b^\top$, and we have: 
\begin{equation}\label{(kappa_eq00)}
(b^{\top}\otimes1)\widetilde{W}=\widetilde{W}(1\otimes y).
\end{equation}

As a consequence, for any $\eta,v\in{\mathcal D}(Q)$, $\xi,u\in{\mathcal D}(Q^{-1})$, we have:
\begin{equation}\label{(kappa_eq0)}
\bigl\langle\widetilde{W}(\bar{\eta}\otimes v),\overline{b\xi}\otimes u\bigr\rangle
=\bigl\langle(b^{\top}\otimes1)\widetilde{W}(\bar{\eta}\otimes v),\bar{\xi}\otimes u\bigr\rangle
=\bigl\langle\widetilde{W}(\bar{\eta}\otimes yv),\bar{\xi}\otimes u\bigr\rangle,
\end{equation}
because $\overline{b\xi}=\overline{(b^*)^*\xi}=[b^*]^{\top}\bar{\xi}=[b^{\top}]^*\bar{\xi}$ and by Equation~\eqref{(kappa_eq00)}.  Furthermore, 
writing $r=QQ^{-1}r$ and $s=Q^{-1}Qs$, we have, for any $\zeta\in{\mathcal D}(Q)$ the following:
\begin{align}
\langle b\xi,\zeta\rangle&=\bigl\langle(\operatorname{id}\otimes\omega_{r,s})(W^*W)\xi,\zeta\bigr\rangle
=\bigl\langle W^*W(\xi\otimes QQ^{-1}r),(\zeta\otimes Q^{-1}Qs)\bigr\rangle \notag \\
&=\bigl\langle W^*W(Q^{-1}\xi\otimes Q^{-1}r),(Q\zeta\otimes Qs)\bigr\rangle
=\bigl\langle Q(\operatorname{id}\otimes\omega_{Q^{-1}r,Qs})(W^*W)Q^{-1}\xi,\zeta\bigr\rangle,
\notag
\end{align}
because $W(Q\otimes Q)\subseteq (Q\otimes Q)W$. In other words, we have $b\xi=Q(\operatorname{id}\otimes\omega_{Q^{-1}r,Qs})(W^*W)Q^{-1}\xi$. 
In particular, $b\xi\in{\mathcal D}(Q^{-1})$. Knowing this, apply to both sides of equation~\eqref{(kappa_eq0)} the characterizing equation for $\widetilde{W}$ 
from Proposition~\ref{manageable_alt}.  Then it becomes:
\begin{equation}\label{(kappa_eq1)}
\bigl\langle W(Q^{-1}b\xi\otimes v),Q\eta\otimes u\bigr\rangle=\bigl\langle W(Q^{-1}\xi\otimes yv),Q\eta\otimes u\bigr\rangle.
\end{equation}

Re-writing Equation~\eqref{(kappa_eq1)}, we then have: 
$$
\bigl\langle (Q^{-1}bQ\otimes1)(Q^{-1}\xi\otimes v),W^*(Q\eta\otimes u)\bigr\rangle=\bigl\langle ((1\otimes y)(Q^{-1}\xi\otimes v),W^*(Q\eta\otimes u)\bigr\rangle.
$$
Compare the two sides, while noting that the elements $W^*(Q\eta\otimes u)$ generate $\operatorname{Ran}(W^*)=\operatorname{Ran}(E)$. 
Since $\operatorname{Ran}(E)\subsetneq {\mathcal H}\otimes{\mathcal H}$, this does not necessarily mean that $Q^{-1}bQ\otimes1=1\otimes y$. 
Nevertheless, knowing $E=W^*W$, we can at least say the following:
\begin{equation}\label{(kappa_eq2)}
E(Q^{-1}bQ\otimes1)=E\bigl(1\otimes y).
\end{equation}
Equivalently, as we know $E(Q\otimes Q)\subseteq (Q\otimes Q)E$, we also have:
\begin{equation}\label{(kappa_eq3)}
E(b\otimes1)=E\bigl(1\otimes QyQ^{-1})=E\bigl(1\otimes\kappa(b)\bigr),
\end{equation}
where $\kappa(b)=QyQ^{-1}=Q(\omega^{\top}\otimes\operatorname{id})(\widetilde{W}\widetilde{W}^*)Q^{-1}$. By the same reason as above, 
our choice of $\omega$ means that the expression $QyQ^{-1}$ is valid.

Note that Equation~\eqref{(kappa)} completely determines the map $\kappa:b\mapsto\kappa(b)$. To see this, assume $b=0$. 
Then by Equation~\eqref{(kappa)} we have $E\bigl(1\otimes\kappa(b)\bigr)=0$.  As $W$ is full, so is $E=W^*W$. Therefore, 
the only way we can have $E\bigl(1\otimes\kappa(b)\bigr)=0$ is when $\kappa(b)=0$.  This indicates that $b\mapsto\kappa(b)$ is 
a well-defined linear map. As its domain contains all elements of the form $b=(\operatorname{id}\otimes\omega_{r,s})(W^*W)$, 
for $r\in{\mathcal D}(Q^{-1}),s\in{\mathcal D}(Q)$, we see that $\kappa$ is a densely-defined map from $B$ into ${\mathcal B}({\mathcal H})$.
\end{proof}

We will return to further discussion of the map $\kappa$ later, but utilizing Proposition~\ref{kappa} we can now show that $B$ and $\widehat{B}$ are 
also a subalgebras:

\begin{prop}
We have:

$B={\overline{\operatorname{span}\bigl\{(\operatorname{id}\otimes\omega)(W^*W):\omega
\in{\mathcal B}({\mathcal H})_*\bigr\}}}^{\|\ \|}$ is a subalgebra in ${\mathcal B}({\mathcal H})$.

$\widehat{B}:={\overline{\operatorname{span}\bigl\{(\omega\otimes\operatorname{id})(WW^*):\omega
\in{\mathcal B}({\mathcal H})_*\bigr\}}}^{\|\ \|}$ is a subalgebra in ${\mathcal B}({\mathcal H})$.
\end{prop}

\begin{proof}
(1). Let $b=(\operatorname{id}\otimes\omega)(W^*W)$, where $\omega=\omega_{r,s}\in{\mathcal B}({\mathcal H})_*$, for $r\in{\mathcal D}(Q^{-1})$, 
$s\in{\mathcal D}(Q)$.  Consider also $b'=(\operatorname{id}\otimes\theta)(W^*W)$, for $\theta\in{\mathcal B}({\mathcal H})_*$.  We know that the 
elements of the form $b$, $b'$ above span a dense subset in $B$.  Observe:
\begin{align}
b'b&=(\operatorname{id}\otimes\theta\otimes\omega)\bigl(W_{12}^*W_{12}W_{13}^*W_{13}\bigr)
=(\operatorname{id}\otimes\theta)\bigl(W^*W(b\otimes1)\bigr)
=(\operatorname{id}\otimes\theta)\bigl(E(b\otimes1)\bigr)
\notag \\
&=(\operatorname{id}\otimes\theta)\bigl(E(1\otimes\kappa(b))\bigr)
=(\operatorname{id}\otimes\rho)(W^*W),
\notag
\end{align}
where $\rho\in{\mathcal B}({\mathcal H})_*$ such that $\rho=\theta\bigl(\,\cdot\,\kappa(b)\bigr)$, using Proposition~\ref{kappa}. 
This shows that $b'b\in B$.

(2). Replacing the role of $W$ with that of $\widehat{W}=\Sigma W^*\Sigma$, we can show that $\widehat{B}$ is also an algebra.
\end{proof}

\begin{prop}
The spaces $B$, $C$, $\widehat{B}$, $\widehat{C}$ are $C^*$-subalgebras of ${\mathcal B}({\mathcal H})$, which are nondegenerately represented.
We have $B\subseteq M(A)$, $C\subseteq M(A)$, and $\widehat{B}\subseteq M(\widehat{A})$, $\widehat{C}\subseteq M(\widehat{A})$.
\end{prop}

\begin{proof}
We have shown that $C$, $\widehat{C}$ and $B$, $\widehat{B}$ are norm-closed subalgebras in ${\mathcal B}({\mathcal H})$, and they are 
already closed under taking the involution.  This means that they are $C^*$-algebras.  As $W$ is full, it is evident that they are all nondegenerately 
represented.  To see this, suppose $u\in{\mathcal H}$, $u\ne0$.  Find $\xi\in{\mathcal H}$ such that $W(u\otimes\xi)\ne0$ (see Definition~\ref{fullness_condition}). 
Note that $W(u\otimes\xi)\in\operatorname{Ran}(W)$, and as $W^*$ is a partial isometry, we know that $W^*|_{\operatorname{Ran}(W)}$ is an isometry.  
Thus $W^*W(u\otimes\xi)\ne0$. So we can find $v,\eta\in{\mathcal H}$ such that $\bigl\langle W^*W(u\otimes\xi),v\otimes\eta\bigr\rangle\ne0$, or 
$\bigl\langle u\otimes\xi,W^*W(v\otimes\eta)\bigr\rangle=\bigl\langle u,bv\bigr\rangle\ne0$, for $b=(\operatorname{id}\otimes\omega_{\eta,\xi})(W^*W)\in B$.  
This implies that only the zero vector is orthogonal to $B{\mathcal H}$, which means $B$ acts on ${\mathcal H}$ in a non-degenerate way. 
Similar for the other algebras.

Finally, from the results of Lemma~\ref{BCinMA}, we can also see that $B\subseteq M(A)$, $C\subseteq M(A)$, 
and $\widehat{B}\subseteq M(\widehat{A})$, $\widehat{C}\subseteq M(\widehat{A})$.
\end{proof}

We saw in Proposition~\ref{Delta(1)} that $E\in M(A\otimes A)$. Meanwhile, by definition we recognize that $B\subseteq M(A)$ and $C\subseteq M(A)$ 
are in fact the left and the right legs of $E$. It is natural to guess that we may have $E\in M(B\otimes C)$. The following proposition gives the proof. 

\begin{prop}\label{EinBtensorC}
We have: $E=W^*W\in M(B\otimes C)$. Similarly, $\widehat{E}=\Sigma WW^*\Sigma\in M(\widehat{B}\otimes\widehat{C})$.
\end{prop}

\begin{proof}
Write $N=B''\bigl(\subseteq{\mathcal B}({\mathcal H})\bigr)$ and $L=C''\bigl(\subseteq{\mathcal B}({\mathcal H})\bigr)$, the enveloping von Neumann algebras 
of $B$ and $C$, respectively. As $B$ and $C$ are the left and the right legs of $E$, respectively, it is straightforward to show that for any $x\otimes y\in B'\otimes C'$, 
we have $E(x\otimes y)=(x\otimes y)E$.  So $E\in B''\otimes C''$, with respect to the von Neumann algebra tensor product.

The $W^*$-algebras $B''$ and $C''$ are generated by the elements of the form $p=(\operatorname{id}\otimes\omega)(W^*W)$, 
$\omega\in{\mathcal B}({\mathcal H})_*$, and $q=(\omega'\otimes\operatorname{id})(W^*W)$, $\omega'\in{\mathcal B}({\mathcal H})_*$, respectively. 
Therefore, for any $\xi,\eta\in{\mathcal H}$, we can find $p_i$'s and $q_i$'s of such form, so that
$$
\left\|\left(\sum_{i=1}^Np_i\otimes q_i-E\right)(\xi\otimes\eta)\right\|\longrightarrow0,
$$
in ${\mathcal H}\otimes{\mathcal H}$. Meanwhile, note that due to $B$ and $C$ being non-degenerately represented on ${\mathcal H}$, we may, 
without loss of generality, regard the vectors as $\xi=b\tilde{\xi}$ and $\eta=c\tilde{\eta}$ for $b\in B$, $c\in C$, $\tilde{\xi},\tilde{\eta}\in{\mathcal H}$. 
Therefore the SOT convergence $\sum_{i=1}^Np_i\otimes q_i\longrightarrow E$ coincides with the strict convergence in $M(B\otimes C)$. 
In other words, for any $b\in B$, $c\in C$, we have the norm convergence:
$$
\left\|\sum_{i=1}^Np_ib\otimes q_ic-E(b\otimes c)\right\|\xrightarrow{\text{ norm }}0.
$$
As $\sum_{i=1}^Np_ib\otimes q_ic\in B\otimes C$ (now back to using the spatial tensor product), this shows that for any $b\otimes c\in B\otimes C$ 
we have $E(b\otimes c)\in B\otimes C$. Therefore $E\in M(B\otimes C)$.
\end{proof}

The next proposition shows how the comultiplication behaves on these subalgebras:

\begin{prop}\label{DeltaonBandC}
\begin{enumerate}
\item For any $b\in B$, we have: $\Delta b=E(1\otimes b)=(1\otimes b)E$. It is also true that for any $x\in M(B)$, we have: 
$\Delta x=E(1\otimes x)=(1\otimes x)E$.
\item For any $c\in C$, we have $\Delta c=(c\otimes 1)E=E(c\otimes1)$. It is also true that for any $y\in M(C)$, we have: 
$\Delta y=(y\otimes1)E=E(y\otimes1)$.

\end{enumerate}
Analogous results hold for $\widehat{\Delta}$ restricted to the subalgebras $\widehat{B}$ and $\widehat{C}$.
\end{prop}

\begin{proof}
Consider $b=(\operatorname{id}\otimes\omega)(W^*W)$, for $\omega\in{\mathcal B}({\mathcal H})_*$. Then 
$$
\Delta b=W^*(1\otimes b)W=(\operatorname{id}\otimes\operatorname{id}\otimes\omega)(W_{12}^*W_{23}^*W_{23}W_{12})
=(\operatorname{id}\otimes\operatorname{id}\otimes\omega)(W_{12}^*W_{12}W_{23}^*W_{23})=E(1\otimes b),
$$
by Equation~\eqref{(mpi3)}. It is also equal to $(1\otimes b)E$, as a consequence of Proposition~\ref{BCcommute}.

Going further, let $x\in M(B)$. Then for any $b\in B$ we have $xb\in B$. Then
$$
\Delta(xb)=E(1\otimes xb)=E(1\otimes x)(1\otimes b) \ {\text { and }} \ \Delta(xb)=(\Delta x)(\Delta b)=(\Delta x)(1\otimes b),
$$
by using $\Delta b=E(1\otimes b)$ and $(\Delta x)E=\Delta x$. As $b\in B$ is arbitrary, this shows that $\Delta x=E(1\otimes x)$.

Similar proofs can be given for $\Delta c$, for $c\in C$ and $\Delta y$, for $y\in M(C)$. Also, analogous results hold for 
$\widehat{\Delta}$ restricted to the subalgebras $\widehat{B}$ and $\widehat{C}$.
\end{proof}

\begin{rem}
If $W$ is a multiplicative unitary, we have $W^*W=1=WW^*$, so we will have $B=C=\widehat{B}=\widehat{C}=\mathbb{C}$.  In our case, 
however, as $W$ is a partial isometry, we have to work with the non-trivial base subalgebras above, which indicates that the structure 
we have is essentially a quantum groupoid not a quantum group.  [An interesting question could be posed for the intermediate case 
of an isometry, that is, either $WW^*=1$ or $W^*W=1$ but not necessarily unitary. But then, considering the result that $C=\widehat{C}$ 
(Proposition~\ref{C=Chat}), it will still end up with $B=C=\widehat{B}=\widehat{C}=\mathbb{C}$ and $W$ becomes a unitary. So no 
genuinely intermediate class is expected.]
\end{rem}

\subsection{The maps $\kappa$ and $\kappa'$}\label{sub5.2}

In comparison with the general theory of $C^*$-algebraic quantum groupoids of separable type (see \cite{BJKVD_qgroupoid1}, \cite{BJKVD_qgroupoid2}), 
the results of Propositions~\ref{Delta(1)}, \ref{Estrictlimit}, \ref{E_weakcomultiplicative}, \ref{DeltaonA}, \ref{EinBtensorC} mean that our $E$ now satisfies 
many of the properties of being the {\em canonical idempotent\/}. One aspect missing is it being a {\em separability idempotent\/}, which involves 
defining certain ``distinguished weights'' at the level of the subalgebras $B$ and $C$ (see \cite{BJKVD_SepId}). In our case, however, by taking 
advantage of the manageability property that we have, we will try to navigate away from having to introduce any weights while gathering 
most of the relevant properties.

As a step in this direction let us revisit the map $b\mapsto\kappa(b)$ that we saw earlier (see Proposition~\ref{kappa}). We saw that $\kappa$ 
is completely determined by the characterization given in Equation~\eqref{(kappa)}, namely $E(b\otimes1)=E\bigl(1\otimes\kappa(b)\bigr)$. 
We can prove the following result on $\kappa$:

\begin{prop}\label{kappa_map}
Recall the map $\kappa:(\operatorname{id}\otimes\omega)(W^*W)\mapsto Q(\omega^{\top}\otimes\operatorname{id})(\widetilde{W}\widetilde{W}^*)Q^{-1}$, 
for $\omega=\omega_{r,s}$, for $r\in{\mathcal D}(Q^{-1})$, $s\in{\mathcal D}(Q)$, as given in Proposition~\ref{kappa}. It was shown that 
$E(b\otimes1)=E\bigl(1\otimes\kappa(b)\bigr)$ uniquely determines $\kappa$ as a densely-defined linear map from $B$ into ${\mathcal B}({\mathcal H})$.  

The map $\kappa$ is closable, and its closure (also written $\kappa$) is an injective and anti-multiplicative map, such that $\kappa$ is densely-defined 
on $B$ having a dense range in $C$.
\end{prop}

\begin{proof}
We saw in Proposition~\ref{kappa} that $b\mapsto\kappa(b)$ is a valid function, which is a densely-defined linear function on $B$, 
as $\operatorname{span}\bigl\{(\operatorname{id}\otimes\omega_{r,s})(W^*W):r\in{\mathcal D}(Q^{-1}),s\in{\mathcal D}(Q)\bigr\}$ 
is dense in $B$. It is easy to see from $E(b\otimes1)=E\bigl(1\otimes\kappa(b)\bigr)$ that $\kappa$ is closable. So from now on we will 
write $\kappa$ to be its closure.

To show that $\kappa$ is injective, suppose $\kappa(b)=0$. Then for any $\omega\in{\mathcal B}({\mathcal H})_*$, we have:
$$
(\operatorname{id}\otimes\omega)(W^*W)b=(\operatorname{id}\otimes\omega)\bigl(E(b\otimes1)\bigr)
=(\operatorname{id}\otimes\omega)\bigl(E(1\otimes\kappa(b))\bigr)=0.
$$
The elements $(\operatorname{id}\otimes\omega)(W^*W)$ generate $B$, which is non-degenerately represented. So we must have $b=0$.  
This shows $\operatorname{Ker}(\kappa)=\{0\}$, which means $\kappa$ is injective.

To prove the anti-multiplicativity, consider $b_1,b_2\in{\mathcal D}(\kappa)$.  Then 
$$
E(b_1b_2\otimes1)=E(b_1\otimes1)(b_2\otimes1)=E\bigl(b_2\otimes\kappa(b_1)\bigr)=E\bigl(1\otimes\kappa(b_2)\kappa(b_1)\bigr).
$$
By the characterization of $\kappa$ given by Equation~\eqref{(kappa)}, this means that $b_1b_2\in{\mathcal D}(\kappa)$ and 
that $\kappa(b_1b_2)=\kappa(b_2)\kappa(b_1)$.

Finally, to see where $\kappa$ maps into, suppose $b\in{\mathcal D}(\kappa)$ and recall Equation~\eqref{(kappa)}: 
$E(b\otimes1)=E\bigl(1\otimes\kappa(b)\bigr)$. Apply here $\omega\otimes\operatorname{id}$ for arbitrary $\omega\in{\mathcal B}({\mathcal H})_*$. 
Then we have 
$$
(\omega\otimes\operatorname{id})(E)\kappa(b)=(\omega\otimes\operatorname{id})\bigl(E(1\otimes\kappa(b)\bigr)
=(\omega\otimes\operatorname{id})\bigl(E(b\otimes1)\bigr)=\bigl(\omega(\,\cdot b)\otimes\operatorname{id}\bigr)(E).
$$
Note that elements of the form $(\omega\otimes\operatorname{id})(E)$, $\omega\in{\mathcal B}({\mathcal H})_*$, are dense in $C$, and the 
above computation shows that multiplication by $\kappa(b)$ to such an element becomes also an element 
$\bigl(\omega(\,\cdot\,b)\otimes\operatorname{id}\bigr)(E)\in C$. This proves that $\kappa(b)\in M(C)$. 

But we can actually go further. Note that for $\omega\in{\mathcal B}({\mathcal H})_*$, by using an earlier observation that 
$[\widetilde{W}^*]^{\top\otimes\top}=W^{R\otimes\top}$ (see Corolloary of Proposition~\ref{unitaryantipode}), we have
\begin{align}
(\omega^{\top}\otimes\operatorname{id})(\widetilde{W}\widetilde{W}^*)
&=(\omega\otimes\operatorname{id})\bigl([\widetilde{W}^*]^{\top\otimes\top}\widetilde{W}^{\top\otimes\top}\bigr)^{\top}
=(\omega\otimes\operatorname{id})\bigl(W^{R\otimes\top}[W^*]^{R\otimes\top}\bigr)^{\top}   \notag \\
&=(\omega\otimes\operatorname{id})\bigl((W^*W)^{R\otimes\top}\bigr)^{\top}
=(\omega^R\otimes\operatorname{id})(W^*W)\,\in C.
\notag
\end{align}
Here, as before, $\omega^{\top}\in{\mathcal B}(\overline{\mathcal H})_*$ is such that $\omega^{\top}(m^\top)=\omega(m)$ 
for $m\in{\mathcal B}({\mathcal H})$, and $\omega^R\in{\mathcal B}({\mathcal H})_*$ is such that $\omega^R(a)=a^R$, for $a\in A$. 
Since $Q(\omega^R\otimes\operatorname{id})(W^*W)Q^{-1}=\bigl(\omega^R(Q^{-1}\,\cdot\,Q)\otimes\operatorname{id}\bigr)(W^*W)$ 
as a consequence of $W(Q\otimes Q)\subseteq (Q\otimes Q)W$, this shows that 
$\kappa(b)=\dots=Q(\omega^R\otimes\operatorname{id})(W^*W)Q^{-1}\in C$. It is also evident that the range of $\kappa$ is dense in $C$.
\end{proof}

In the proof of Proposition~\ref{kappa}, in Equation~\eqref{(kappa_eq2)}, we also saw that 
$$
E\bigl(Q^{-1}bQ\otimes1\bigr)=E(1\otimes y),
$$
where $b=(\operatorname{id}\otimes\omega)(W^*W)$, $\omega=\omega_{r,s}$ with $r\in{\mathcal D}(Q^{-1})$, $s\in{\mathcal D}(Q)$, 
and $y=(\omega^{\top}\otimes\operatorname{id})(\widetilde{W}\widetilde{W}^*)$.  Comparing with Equation~\eqref{(kappa)}, 
we can see that $Q^{-1}bQ\in{\mathcal D}(\kappa)$ and that $\kappa(Q^{-1}bQ)=y$.

Next, let us consider a different map $R_{\kappa}:b\mapsto Q^{-1}\kappa(b)Q$, for $b\in{\mathcal D}(\kappa)$. In particular, if
$b=(\operatorname{id}\otimes\omega)(W^*W)$ as above, we have
\begin{equation}\label{(Rkappa)}
R_{\kappa}(b)=Q^{-1}\kappa\bigl((\operatorname{id}\otimes\omega)(W^*W)\bigr)Q=(\omega^{\top}\otimes\operatorname{id})(\widetilde{W}\widetilde{W}^*)
=y.
\end{equation}
We show below that this map extends to a bounded map on $B$.

\begin{prop}\label{Rkappa_map}
\begin{enumerate}
 \item Consider the map $R_{\kappa}:b\mapsto Q^{-1}\kappa(b)Q$, for $b=(\operatorname{id}\otimes\omega)(W^*W)\in{\mathcal D}(\kappa)$, 
 as given in Equation~\eqref{(Rkappa)} above. 
It extends to a bounded map $R_{\kappa}:B\to C$, which becomes an injective ${}^*$-anti-isomorphism. It can be 
alternatively characterized by 
\begin{equation}\label{(Rkappa_alt)}
R_{\kappa}(b)=\kappa(Q^{-1}bQ)=\kappa\bigl(Q^{-1}(\operatorname{id}\otimes\omega)(W^*W)Q\bigr).
\end{equation}
 \item Write $T:=Q(\,\cdot\,)Q^{-1}$. We have:
$$
\kappa=T\circ R_{\kappa}=R_{\kappa}\circ T.
$$
\end{enumerate}
\end{prop}

\begin{proof}
(1). As $Q^{-1}(\,\cdot\,)Q$ naturally preserves multiplication, and obviously injective, we can see quickly from Proposition~\ref{kappa_map} that $R_{\kappa}$ 
is an anti-multiplicative injective map.  We also know that $\kappa$ is a densely-defined map on $B$ having a dense range in $C$.

Meanwhile, for $b=(\operatorname{id}\otimes\omega)(W^*W)\in{\mathcal D}(\kappa)$, by Equation~\eqref{(Rkappa)} we have: 
$$
R_{\kappa}(b^*)=R_{\kappa}\bigl((\operatorname{id}\otimes\bar{\omega})(W^*W)\bigr)
=(\bar{\omega}^{\top}\otimes\operatorname{id})(\widetilde{W}\widetilde{W}^*)
=\bigl[(\omega^{\top}\otimes\operatorname{id})(\widetilde{W}\widetilde{W}^*)\bigr]^*=y^*=\bigl[R_{\kappa}(b)\bigr]^*,
$$
because $\bar{\omega}^{\top}=\overline{\omega^{\top}}$.  This means that $R_{\kappa}$ is also a ${}^*$-map.
So $R_{\kappa}$ is a ${}^*$-anti-homomorphism, so bounded.  Therefore, it extends to all of $B$, becoming a ${}^*$-anti-isomorphism 
$R_{\kappa}:B\to C$.

Finally, from $E\bigl(Q^{-1}bQ\otimes1\bigr)=E(1\otimes y)$, we obtain a different characterization: 
$R_{\kappa}(b)=\kappa(Q^{-1}bQ)=\kappa\bigl(Q^{-1}(\operatorname{id}\otimes\omega)(W^*W)Q\bigr)$. 
As $R_{\kappa}$ is shown to be bounded, there is no reason to worry about its domain.

(2). From (1), we have $\kappa(b)=QR_{\kappa}(b)Q^{-1}=R_{\kappa}(QbQ^{-1}),\forall b$. So $\kappa=T\circ R_{\kappa}=R_{\kappa}\circ T$.
\end{proof}

Similar to the map $\kappa$, there exists another map $\kappa'$, densely-defined on $C$ having a dense range in $B$. See below:

\begin{prop}\label{kappa'}
For $c=(\omega\otimes\operatorname{id})(W^*W)$, for $\omega\in{\mathcal B}({\mathcal H})_*$, define $\kappa'(c)\in{\mathcal B}({\mathcal H})$ by 
$$
\kappa'(c)=R_{\kappa}^{-1}(QcQ^{-1})=(R_{\kappa}^{-1}\circ T)(c)=(T\circ R_{\kappa}^{-1})(c).
$$
Here $T(\,\cdot\,)=Q(\,\cdot\,)Q^{-1}$ as before. The fact that $R_{\kappa}^{-1}\circ T=T\circ R_{\kappa}^{-1}$ is a quick consequence of the previous 
proposition, where we saw $T\circ R_{\kappa}=R_{\kappa}\circ T$.

Then $c\mapsto\kappa'(c)$ is characterized by the following:
\begin{equation}\label{(kappa')}
(1\otimes c)E=\bigl(\kappa'(c)\otimes1\bigr)E.
\end{equation}
In this way, we obtain a well-defined closed linear map $\kappa':c\mapsto\kappa'(c)$. It is an injective and anti-multiplicative map, 
which is densely-defined on $C$ having a dense range in $B$.
\end{prop}

\begin{proof}
From Proposition~\ref{Rkappa_map}, we saw that $R_{\kappa}:B\to C$ is a ${}^*$-anti-isomorphism. Therefore so is $R_{\kappa}^{-1}:C\to B$. 
It is evident that $T$ is injective and preserves multiplication, and that $T$ leaves $C$ invariant. So we can see quickly that $\kappa'=R_{\kappa}^{-1}\circ T$ 
becomes a valid closed linear map, injective and anti-multiplicative, that is densely-defined on $C$ having a dense range in $B$.

To verify Equation~\eqref{(kappa')}, suppose $c\in{\mathcal D}(\kappa')$. As $R_{\kappa}$ is an isomorphism between $B$ and $C$, we may write 
$c=R_{\kappa}(b)$, for some $b\in B$. By Equation~\eqref{(kappa_eq2)}, we know $E(Q^{-1}bQ\otimes1)=E\bigl(1\otimes R_{\kappa}(b)\bigr)$, 
which can be expressed as
$$
E\bigl(Q^{-1}R_{\kappa}^{-1}(c)Q\otimes1)=E\bigl(1\otimes c\bigr).
$$
Taking the adjoint, the expression becomes:
$$
\bigl(QR_{\kappa}^{-1}(c^*)Q^{-1}\otimes 1\bigr)E=(1\otimes c^*)E,
$$
as $R_{\kappa}$ preserves the involution and since $E$, $Q$ are self-adjoint. Remembering the definition of $\kappa'$, this is exactly saying
$$
(1\otimes c^*)E=\bigl((T\circ R_{\kappa}^{-1})(c^*)\otimes1\bigr)E=\bigl(\kappa'(c^*)\otimes1\bigr)E.
$$
This proves Equation~\eqref{(kappa')}. As was the case for $\kappa$ with Equation~\eqref{(kappa)}, this actually completely characterizes $\kappa'$.
\end{proof}

\begin{rem}
We thus have two densely-defined, dense range maps $\kappa:B\to C$ and $\kappa':C\to B$. However, we have $\kappa'\ne\kappa^{-1}$. 
Note that $\kappa'\circ\kappa=(T\circ R_{\kappa}^{-1})\circ(R_{\kappa}\circ T)=T^2\ne\operatorname{Id}$. 
\end{rem}

Neither of $\kappa$ and $\kappa'$ are ${}^*$-preserving. Nonetheless, we have the following:

\begin{prop}\label{kappa*}
For $b\in{\mathcal D}(\kappa)$, we have $\kappa(b)^*{\mathcal D}(\kappa')$, and we have: $\kappa'\bigl(\kappa(b)^*\bigr)^*=b$. 
Similarly for $c\in{\mathcal D}(\kappa')$, we have: $\kappa\bigl(\kappa'(c)^*\bigr)^*=c$. 
\end{prop}

\begin{proof}
For $b\in{\mathcal D}(\kappa)$, we have: $E(b\otimes1)=E\bigl(1\otimes\kappa(b)\bigr)$, by the characterization in Equation~\eqref{(kappa)}. 
Take the adjoint: $(b^*\otimes1)E=\bigl(1\otimes\kappa(b)^*\bigr)E$. By comparing it with the characterization of $\kappa'$ given in 
Equation~\eqref{(kappa')}, we see that $\kappa(B)^*\in{\mathcal D}(\kappa')$ and that $\kappa'\bigl(\kappa(b)^*\bigr)=b^*$. 
\end{proof}

The results of Propositions~\ref{EinBtensorC}, \ref{kappa_map}, \ref{kappa'}, \ref{kappa*} indicate that our canonical idempotent $E$ behaves 
very much like a {\em separability idempotent\/} (in the sense of \cite{VDsepid}, \cite{BJKVD_SepId}, \cite{BJKVD_qgroupoid2}) in the 
theory of $C^*$-algebraic quantum groupoids of separable type. The maps $\kappa$ and $\kappa'$ behave exactly like the maps $\gamma_B$ 
and $\gamma_C$, even though we were able to construct $\kappa$ and $\kappa'$ without having to introduce any ``distinguished weights'' 
on $B$ and $C$. Having the manageability condition on the multiplicative partial isometry operator $W$ allowed us to navigate without the weights. 
Sort of similar as in Section~\ref{sec4}, where the antipode map $S$ was constructed without needing to consider the left and right invariant 
weights, by taking advantage of the manageability of our multiplicative partial isometry $W$.

\begin{rem}
While we do not explicitly write them down, analogous results hold for $\widehat{E}$, with the maps $\hat{\kappa}=T\circ\widehat{R}_{\hat{\kappa}}$
and $\hat{\kappa}'={\widehat{R}_{\hat{\kappa}}}^{-1}\circ T$, between the subalgebras $\widehat{B}$ and $\widehat{C}$. The $\widehat{R}_{\hat{\kappa}}$ 
map is a ${}^*$-anti-isomorphism and $T$ as same as above. 

Meanwhile, this means we have the ${}^*$-anti-isomorphisms $R_{\kappa}:B\to C$ and $\widehat{R}_{\hat{\kappa}}:\widehat{B}\to\widehat{C}$. 
We know from Proposition~\ref{C=Chat}, we actually have $C=\widehat{C}$.  It follows that $B\cong\widehat{B}$.  However, in general $B\ne\widehat{B}$.
\end{rem}

\subsection{Final remarks}

So far, from a multiplicative partial isometry $W$  satisfying certain conditions including the {\em manageability\/},
we have constructed a $C^*$-algebra $A$; the comultiplication map $\Delta:A\to M(A\otimes A)$; the $C^*$-subalgebras $B\subseteq M(A)$ 
and $C\subseteq M(A)$; the canonical idempotent element $E\in M(B\otimes C)$; the ${}^*$-anti-isomorphism $R_{\kappa}:B\to C$; and the closed 
densely-defined maps $\kappa:B\to C$ and $\kappa':C\to B$. We also constructed the antipode map $S=R_A\circ\tau_{-\frac{i}2}$ on $A$, 
in terms of the unitary antipode and the scaling group.

Loosely speaking, the $C^*$-algebra $A$ plays the role of $C_0(G)$ for a (quantum) groupoid $G$; $\Delta$ is the comultiplication map; 
the subalgebras $B$ and $C$ are the source and the target algebras based on the unit space $G^{(0)}$; and $E=\Delta(1)$. The antipode 
map $S$ corresponds to the inverse map on $G$. The $\kappa$, $\kappa'$ maps also correspond to the inverse map, but at the base level, 
flipping between the source and the target.

Comparing with the definition of a {\em $C^*$-algebraic quantum groupoid of separable type\/} (See Definition~4.8 of \cite{BJKVD_qgroupoid1} 
or Definition~1.2 of \cite{BJKVD_qgroupoid2}), we notice that our construction in this paper does not involve distinguished weights at the 
base algebra level and does not use the left and the right invariant weights $\varphi$ and $\psi$. This makes our setting slightly different. 

Nonetheless, due to our having started out with a multiplicative partial isometry operator $W$ satisfying the fullness and the manageability 
conditions, we could construct all the essential structure maps such as $E$, $S$, $\kappa$, $\kappa'$, which were shown to behave 
the same way as the corresponding maps $E$, $S$, $\gamma_B$, $\gamma_C$ for a $C^*$-algebraic quantum groupoid of separable type 
\cite{BJKVD_qgroupoid1} , \cite{BJKVD_qgroupoid2}. This observation means that even though the ways these structure maps have been 
obtained were different, we can expect that the results obtained in \cite{BJKVD_qgroupoid2} regarding the antipode map $S$ and the maps 
$\gamma_B$, $\gamma_C$ will carry over to our maps $S$, $\kappa$, $\kappa'$. (Please refer to that paper for more details and other results.)

Going the other way, if we begin with such a quantum groupoid (see Definition~1.2 in \cite{BJKVD_qgroupoid2}), we can construct from 
the defining axioms a multiplicative partial isometry $W$, as well as the antipode map and its polar decomposition.  It turns out that 
we can always find a positive operator $P$ implementing the scaling group, where $P^{-\frac12}$ behaves quite like a $Q$ operator 
as in Definition~\ref{manageable}, thereby showing that $W$ is in fact manageable. An analogous result is known in the quantum group 
case \cite{KuVa}.

Since any multiplicative partial isometry constructed from the axioms for quantum groupoids of separable type would turn out to be manageable, 
the results from the current paper will allow us to find a convenient alternative way to construct a dual quantum groupoid of the same type. 
This aspect will be discussed more carefully in our future paper.

\bigskip\bigskip




\end{document}